\documentclass{elsarticle}

\usepackage[utf8]{inputenc}
\usepackage[english]{babel}
\usepackage{hhline}
\usepackage{lscape}
\usepackage[table]{xcolor}
\usepackage{multirow}
\usepackage{amsmath,amssymb,amsfonts,amsthm}
\usepackage[hidelinks,colorlinks = true]{hyperref}
\usepackage{enumitem}

\usepackage{tcolorbox}
\newtcolorbox{mybox}{colback=red!5!white,colframe=red!75!black, sharp corners = all}
\usepackage{bclogo} 
\def\ucsign{\bcpanchant}

\DeclareMathOperator{\rk}{rk}
\DeclareMathOperator{\id}{id}
\DeclareMathOperator{\Ker}{ker}

\DeclareMathOperator{\Char}{char}
\DeclareMathOperator{\Span}{span}
\DeclareMathOperator{\Res}{Res}
\DeclareMathOperator{\D}{D}
\DeclareMathOperator{\B}{{\mathfrak B}}

\DeclareMathOperator{\rad}{rad}
\DeclareMathOperator{\diag}{diag}
\DeclareMathOperator{\colsp}{\mathbf C}
\DeclareMathOperator{\rowsp}{\mathbf R}

\DeclareMathOperator{\tr}{tr}

\def\F{{\mathbb F}}
\def\R{{\mathfrak R}}
\def\M{{\mathcal M}}
\def\L{{\mathbf L}}
\def\FL{{\mathcal L}}
\def\mbf{\mathbf}

\newenvironment{psmallmatrix}
  {\left(\begin{smallmatrix}}
  {\end{smallmatrix}\right)}

\newtheorem{theorem}{Theorem}[section]
\newtheorem*{theorem*}{Theorem}
\newtheorem{lemma}[theorem]{Lemma}

\newtheorem{corollary}[theorem]{Corollary}
\numberwithin{equation}{section}

\theoremstyle{definition}
\newtheorem{definition}[theorem]{Definition}
\newtheorem{example}[theorem]{Example}
\newtheorem{remark}[theorem]{Remark}

\begin{document}

\title{Linear maps preserving the matrix discriminant. An enhanced approach}
\author{Andrey Yurkov}
\date{}

\ead{andrey.yurkov@biu.ac.il}

\affiliation{organization={Department of Mathematics, Bar-Ilan University},addressline={Ramat Gan},postcode={5290002}, country={Israel}}

\begin{keyword}
\MSC[2020] 15A15 \sep 15A86 \sep 47B49\\
matrix discriminant \sep linear preservers    
\end{keyword}

\begin{abstract}
The discriminant of a square matrix is defined as a discriminant of its characteristic polynomial. In this paper we provide a description of linear maps preserving the matrix discriminant assuming that the ground field is large enough. Our description generalizes the known results to the case where the field may have a finite cardinality and its characteristic may divide the size of considered matrices. 

To obtain this description, we extend the approach used in our previous work on the subject and, in addition to that, introduce new invariants for the action of all bijective linear maps preserving a given polynomial, which are on independent interest and could be applied to other problems in the theory of linear preservers.
\end{abstract}

\maketitle

\section{Introduction}
\label{sec:Introduction}

The problem of finding the description of linear maps preserving different matrix invariants attracted the attention of many mathematicians since the late 19th century, when Frobenius in~1897 provided the characterisation of linear maps preserving the determinant of a square matrix (see~\cite{GF}). This field of mathematics advanced significantly in the 20th century and remains an active area of research, with applications in quantum information theory, operator theory and functional analysis. We refer the reader to~\cite{hogben_handbook_2014} for a brief introduction and to~\cite{LAMA199233} for a overview of the studies of the subject up to the end of the 20th century.

The matrix discriminant, i.e., the discriminant of the characteristic polynomial is an example of a polynomial matrix invariant. In particular, it indicates whether the eigenvalues are all distinct. This property of the matrix discriminant finds applications in atomic and molecular physics (see, e.g.,~\cite{Bhattacharya2007}). The studies of the matrix discriminant originated by Kummer, who used it in~\cite{Kummer1843} to show that all eigenvalues of a real symmetric $3\times 3$ matrix are real. 

The first result in characterising linear maps preserving the matrix discriminant appeared in 1978 in the work of Pierce~\cite{Pierce1979}. It is formulated as follows:

\begin{theorem}[Cf.~{\cite[Main Theorem]{Pierce1979}}]\label{thm:PierceThm}
Assume that $n \ge 6$ and $\F$ is an infinite field. If $T$ is a linear map on the space of all $n \times n$ matrices preserving the matrix discriminant, then $T$ has one of two forms:
$$
T(X) = cS^{-1}XS+\alpha(X)I_n\qquad\mbox{or}\qquad T(X) = cS^{-1}X^tS+\alpha(X)I_n
$$
where $S$ is an invertible matrix, $c^n = 1$, and $\alpha$ is a linear functional on the space of all $n\times n$ matrices, $X^t$ stands for the transpose of $X$, $I_n$ is the identity matrix.
\end{theorem}

Accordingly, the cases where $\F$ is a finite field or $n \ge 6$ are uncovered by this theorem. In addition, as it is discussed in~\cite{Yurkov2025}, the paper of Pierce requires several improvements and corrections. In particular, the proof of Pierce does not carry out properly the case when $\Char(\F)$ divides $n$ (we refer the reader to~\cite[Proposition~A.8]{Yurkov2025} for our detailed comment on that issue). This case also is not considered in our previous work~\cite{Yurkov2025} on the subject. Accordingly, in the current paper we extend the existing description to the case when $\Char(\F)$ may divide $n$ and $n$ may be less than $6$. Thus, our main result states the following:

\begin{theorem}[Cf.~{Theorem~\ref{thm:DPresHasForm}}]\label{thm:MainThmInIntro} Assume that $n \ge 3$ and $|\F| > n(n-1)$. Let $T$ be a linear map on the space of all $n\times n$ square matrices preserving the matrix discriminant. Then there are an invertible matrix $M$, a linear functional $\alpha$ on the space of all $n\times n$ matrices and $c \in \F$ such that $c^{n(n-1)} = 1$ and
\begin{equation}\notag 
T(X) = cM^{-1}XM + \alpha(X)I_n\;\;\mbox{for every $n\times n$ matrix $X$}
\end{equation}
or
\begin{equation}\notag 
T(X) = cM^{-1}X^tM + \alpha(X)I_n\;\;\mbox{for every $n\times n$ matrix $X$}.\footnote{The reader may notice that in contrast to the requirement $c^{n} = 1$ in the formulation of Theorem~\ref{thm:PierceThm} we require only that $c^{n(n-1)} = 1$. Indeed, from the property~\ref{enum:DiscPropertyH} of the matrix discriminant presented below we obtain that the linear map $X \mapsto cX$ preserves the matrix discriminant for all $c \in \F$ with $c^{n(n-1)} = 1$.}
\end{equation}
\end{theorem}

The case of $2\times 2$ matrices is discussed in Remark~\ref{rem:CaseN2}. As it is shown in Example~\ref{rem:CaseN2Ch2Example}, if $\Char(\F) = 2$, then there exist linear maps preserving the matrix discriminant which do not admit the description given in Theorem~\eqref{thm:MainThmInIntro}.

For $n = 1$, the matrix discriminant of every $1\times 1$ matrix is equal to $1$ and consequently every linear map on the space of all $1\times 1$ matrices preserves it. Thus, in this paper we solve the linear preserver problem for the matrix discriminant for all fields with sufficiently large cardinality (i.e., $|\F| > n(n-1)$) and all possible $n$.

The proof of  Theorem~\eqref{thm:MainThmInIntro} is based on the similar ideas as in~\cite{Yurkov2025}. We begin with the following observation: if $P$ is a polynomial function of matrix coefficients, then for the given matrices $A, B$ the expression $P(A + \lambda B)$ is a polynomial function of $\lambda,$ and for a certain $P$ it is possible to relate $\max_A (\deg_\lambda P(A + \lambda B))$ with various properties of $B.$ Namely, assuming that $P$ is equal to matrix discriminant and $B$ is a given $n\times n$ matrix, we establish that $\deg_\lambda P(A + \lambda B) \le 2n - 2$ for all $A$ if and only if $B$ is a linear combination of a scalar matrix and a rank-one matrix, where several exceptions are made for $B$ when $n \le 5$ (Lemma~\ref{lem:LowDegImpliesLowRankOrSeveralCasesForAlgebraicallyClosed}).

In order to characterise the matrices $B$ which are equal to a linear combination of a scalar matrix and a rank-one matrix in exceptional cases, we expand the approach developed in~\cite{YurkovPolyPres2026} and introduce finer invariants for the action of all bijective linear maps preserving a given polynomial $P$ on $\F^n$  (Definition~\ref{def:PolyPointInv} and Lemma~\ref{lem:DegAndPartPreservance}). Using these invariants, we show that the orbit of linear combination of a scalar matrix and a rank-one $n\times n$ matrix consists of matrices of the same form (Corollary~\ref{cor:IfPresDThenPresRk1PScFull}). In our opinion, these invariants may have other applications in the theory of linear preservers.

The further proof is carried out by considering the action of a given linear map $T$ preserving the matrix discriminant on the set of linear spaces of $n\times n$ matrices, which are generated by scalar matrices and matrices of rank 1 (see Definition~\ref{def:LMatSpace}). This allows to transform $T$ to the required form using matrix similarity operation (Lemma~\ref{lem:TSendLtoLHasRightForm}) and establish the main result of this paper.

\vspace{\topsep}

The paper is organised as follows. In Section~\ref{sec:NotationDefinitions}, Section~\ref{sec:ResDiscPrelim} and Section~\ref{sec:FunctionRad} we provide necessary notations and definitions, introduce the notion of the matrix discriminant and the notion of the radical of a function, correspondingly; in Section~\ref{sec:SupLinAlg} we formulate and prove the necessary basic facts from the theory of polynomials, linear algebra and study the properties of the set $\R_n(\F) \subseteq \M_{n}(\F)$ of matrices which could be presented as a sum of rank-one matrix and a scalar matrix; in Section~\ref{sec:OrbInvariants} we extend the results from~\cite{YurkovPolyPres2026} and introduce new invariants which are preserved by any linear maps preserving the given polynomial; in Section~\ref{sec:MatsSmallDegDesc} we provide a description of $n\times n$ matrices $X$ satisfying the inequality $\deg_\lambda \left(\D(A + \lambda X)\right) \le 2n - 2$ for all $A \in \M_{n}(\F)$; Section~\ref{sec:InvariantsOfSmallMatrices} is devoted to calculation of necessary explicit values and estimates of the invariants introduced in Section~\ref{sec:OrbInvariants}; in Section~\ref{sec:RadDAndRadB} the radical of $\D$ is calculated; Section~\ref{sec:VecSpacesInR1} contains the theory of vector spaces of the form  $\mbf L(\mbf v) \subseteq \R_n(\F)$ corresponding to every $0 \neq \mbf v \in \F^n$ (Definition~\ref{def:LMatSpace}); finally, in Section~\ref{sec:LinDiscPres} we provide the characterisation of  linear maps preserving the matrix discriminant as it is stated in Theorem~\ref{thm:MainThmInIntro}.

\section{Notation and basic definitions}
\label{sec:NotationDefinitions}

Throughout this paper, by $\F$ we denote a field without any preliminary restrictions on its cardinality and characteristic. By $\F[x_1,\ldots, x_n]$ we denote the ring of multivariate polynomials with coefficients belonging to $\F$, and\linebreak $\F^{\le d}[x_1,\ldots, x_n] \subseteq \F[x_1,\ldots, x_n]$  denotes the vector space of all polynomials whose degree not greater than $d$. If $P \in \F[x; x_1, \ldots, x_n]$ and $k \ge 0$ is an integer, then $[k]P \in \F[x_1,\ldots, x_n]$ denotes the coefficient of $x^k$ in $P$.

If $\mathsf R$ is equal to $\F$ or $\F[x_1,\ldots, x_n]$, then $\M_{m}(\mathsf R)$ denotes the set of all square matrices of size $m\times m$ with entries belonging to $\mathsf R$. Let $I_m \in \M_{m}(\F)$ denote the unit matrix, $O_{m}$ denote the zero matrix and $E_{i\,j}\in \M_{m}(\F)$ the matrix unit, i.e. all elements of are $E_{i\,j}$ equal to 0 except the element belonging to the intersection of the $i$-th row and $j$-th column which is equal to 1. By $\diag (a_1,\ldots, a_m) \in \M_{m}(\F)$ we denote a \emph{diagonal} matrix, in which the entries of the diagonal are given by a vector $(a_1,\ldots, a_m)$ and the entries outside the main diagonal are all equal to zero.

$x_{i\,j}$ denotes the entry of the matrix $X \in \M_{m}(\F)$ belonging to the intersection of the $i$-th row and $j$-th column. 

The matrix $X \in \M_{m}(\F)$ is called \emph{scalar} if $X = cI_m$ for some $c \in \F.$

Let $X^t \in \M_{m}(\F)$ denote a \emph{transpose} of $X,$ i.e.  $x^{t}_{i\,j} = x_{j\,i}$ for all $1 \le i, j \le m.$

For given $X \in \M_k(\F),\; Y \in \M_{m}(\F)$  we denote by $X \oplus Y \in \M_{(k + m)}(\F)$ a \emph{direct sum} of matrices $X$ and $Y,$ i.e.
\[
X \oplus Y = \left[\begin{array}{ c | c }
    X & 0 \\
    \hline
    0 & Y
  \end{array}\right].
\]

\begin{definition}For a given $A \in \M_{m}(\F)$ ,we denote by $p_A(x) \in \F[x]$  its \emph{characteristic polynomial}, i.e.
\[
p_A(x) = \det(xI_m - A).
\]
We will use formula 
\begin{equation}\label{eqqq:formulaChPoly}
p_A(x) = x^m - E_1(A)x^{m-1} + E_2(A)x^{m-2} + \ldots + (-1)^m E_{m}(A),
\end{equation}
where $E_k(A)$ denotes the sum of all maximal minors of $A$. This formula could be found, e.g., in~\cite{Horn_Johnson_1985}.
\end{definition}

Let $V$ be a vector space. By $V^*$ denote a vector space which is \emph{dual} to $V$, i.e., $V^*$ consists of all linear functionals $\alpha \colon V \to \F$. If $V$ is a vector space and $S \subseteq$ is a set, then $\langle S \rangle$ denotes a vector space \emph{generated} by $S,$ i.e. the set of all linear combinations of elements of $S$. In addition, by $S^0 \subseteq V^*$ we denote an \emph{annihilator} of $S$, i.e., a vector space defined by $S^0 = \{f \in V^*: f(\mbf x) = 0\;\;\mbox{for all}\;\; \mbf x \in S\}.$ If vector spaces $U, V \subset W$ are subspaces of vector space $W$, then by $U + V$ we denote a \emph{sum of spaces} $U$ and $V$ which is defined by $U + V = \langle U \cup V\rangle$. Similarly we define and denote a sum of three and more vector subspaces of the same vector space. By $\mbf e^{(1)}, \ldots, \mbf e^{(m)}$ we denote the elements of \emph{standard basis} of $\F^m$, i.e., $\mbf e^{(1)} = (1, 0, \ldots, 0), \ldots, \mbf e^{(m)} = (0, \ldots, 0, 1)$. If $\mbf v, \mbf w \in \F^m,$ then $(\mbf v,\mbf w)$ denotes their \emph{inner product} defined by $(\mbf v, \mbf w) = \mbf v_1\mbf w_1 + \ldots + \mbf v_m\mbf w_m$.

If $X \in \M_{n}(F)$ then by $\colsp(X) \subseteq \F^n$ and $\rowsp(X) \subseteq \F^n$ we denote the \emph{column space} of $X$ and the \emph{row space} of $X$, respectively. That is, if $X = (x_{i\,j}) \in \M_{n}(\F)$, then $\colsp(X) = \left\langle \begin{psmallmatrix}x_{1\,1} \\ \vdots \\ x_{n\,1} \end{psmallmatrix}, \ldots, \begin{psmallmatrix}x_{1\,n} \\ \vdots \\ x_{n\,n} \end{psmallmatrix}\right\rangle$ and $\rowsp(X) = \left\langle \begin{psmallmatrix}x_{1\,1} \\ \vdots \\ x_{1\,n} \end{psmallmatrix}, \ldots, \begin{psmallmatrix}x_{n\,1} \\ \vdots \\ x_{n\,n} \end{psmallmatrix}\right\rangle$.

\section{Preliminary facts about the discriminant and the matrix discriminant}
\label{sec:ResDiscPrelim}

Before we provide the definition of the discriminant, we need to introduce the notion of the resultant of two polynomials.

\begin{definition}Let $f,g \in \F[x]$ be polynomials and
\[f(x) = a_0 + \ldots + a_m x^m,\qquad g(x) = b_0 + \ldots + b_k x^k,\]
where $m + k \ge 1$, but $a_m$ or $b_k$ may be equal to zero. Then by\linebreak $\Res_{m\, k} (f,g) = \Res_{m\, k}(a_0,\ldots, a_m; b_0, \ldots, b_k)$ we denote the \emph{resultant} of $f$ and $g$ which is defined by
\begin{equation}\label{eq:DefRes}
\Res_{m\, k} (f,g) =
\begin{vmatrix}
a_m &  \multicolumn{1}{c}{\cdots} & a_1 & a_{0} & \multicolumn{1}{c}{\cdots} & 0\\
\vdots & \multicolumn{1}{c}{\ddots} & \vdots & \vdots & \multicolumn{1}{c}{\ddots} & \vdots\\
0 & \multicolumn{1}{c}{\cdots} & a_m & a_{m-1} & \multicolumn{1}{c}{\cdots} & a_0\\
b_{k} & \multicolumn{1}{c}{\cdots} & b_1 & b_0 & \multicolumn{1}{c}{\cdots} & 0 \\

\vdots & \multicolumn{1}{c}{\ddots} & \vdots & \vdots & \multicolumn{1}{c}{\ddots} & \vdots\\
0 & \multicolumn{1}{c}{\cdots} & b_k & b_{k-1} & \multicolumn{1}{c}{\cdots} & b_{0}
\end{vmatrix}.
\end{equation}
\end{definition}

\begin{remark}The order of the determinant in the equality~(\ref{eq:DefRes}) is equal to $m + n.$
\end{remark}

The following properties of the resultant follow immediately from the formula~\eqref{eq:DefRes}.

\begin{corollary}\label{lem:ResMultOneByScal}If $m + k \ge 1$, $f(x) = a_0 + \ldots + a_m x^m$ and $g(x) = b_0 + \ldots + b_k x^k$, then 
\begin{equation}\label{lem:ResMultOneByScal:eqq1}
\Res_{m\, k}(\lambda f, \mu g) = \lambda^k \mu^m \Res_{m\, k}(f, g).
\end{equation}
\end{corollary}
\begin{corollary}\label{lem:ResReverse}If $m + k \ge 1$, $f(x) = a_0 + \ldots + a_m x^m$, $g(x) = b_0 + \ldots + b_k x^k$ and $f' = a_m + \ldots + a_0x^m$, $g' = b_k + \ldots + b_0x^k$, then 
\begin{equation}\label{lem:ResReverse:eqq1}
\Res_{m\, k}(f',g') = (-1)^{mk}\Res_{m\, k}(f,g).
\end{equation}
\end{corollary}
\begin{corollary}\label{cor:ResXX0G}If $f(x) = x - x_0$, $g \in \F^{\le k}$, then 
\begin{equation}\label{cor:ResXX0G:eqq1}
\Res_{1\, k}(f,g) = g(x_0).
\end{equation}
\end{corollary}

Recall the main properties of resultant (in the square brackets we include the references to the corresponding equalities from~\cite[Chapter~12]{Gelfand1994}).

\begin{enumerate}[label=\textsc{R\arabic*.}, ref=\textsc{R\arabic*}]
\item\label{enum:ResTwoSplit} \[\Res_{m\, k}\left(\prod_{i = 1}^{m} (x - x_i) ,\prod_{j = 1}^{k} (x - y_j)\right) = \prod_{\substack{1 \le i \le m\\ 1 \le j \le k}} (x_i - y_j)\] for all $x_1, \ldots, x_m, y_1, \ldots, y_k \in \F.$ \hfill [1.3]
\item $\Res_{m\, k}\left(\prod_{i = 1}^{m} (x - x_i), \prod_{j = 1}^{k} (x - y_j)\right) = 0 \Leftrightarrow \exists 1 \le i \le m, 1 \le j \le k \colon x_i = y_j.$
\item\label{enum:ResProduct} $\Res_{(m_1 + m_2)\, k}(f_1\cdot f_2, g) = \Res_{m_1\, k}(f_1,g)\cdot\Res_{m_2\, k}(f_2,g)$ for all $f_1 \in \F^{\le m_1}[x],$ $f_2 \in \F^{\le m_2}[x],$ $g \in \F^{\le k}[x].$ \hfill [1.4]
\item\label{enum:ResPropertyC}$\Res_{m\,k}(f(x + c), g(x+c)) = \Res_{m\,k}(f(x),g(x))$ for all $f \in \F^{\le m}[x],$ $g \in \F^{\le k}[x],$ $c \in \F.$ \hfill [1.8]
\item\label{enum:ResWeightedMulLin} $\Res_{m\,k}(\lambda^0 a_0, \ldots, \lambda^m a_m; \lambda^0 b_0, \ldots, \lambda^k b_k) = \lambda^{mk} \Res(a_0, \ldots, a_m; b_0, \ldots, b_k)$ for all $f = a_0 + \ldots + a_mx^m, g = b_0 + \ldots + b_kx^k, \lambda \in \F[x].$ \hfill [1.6]
\end{enumerate}

\begin{corollary}[from the property~\ref{enum:ResProduct} and Corollary~\ref{cor:ResXX0G}]\label{cor:ResSplitPoly}
\begin{equation}\label{cor:ResSplitPoly:eqq1}
\Res_{m\,k}\left(\prod_{i = 1}^{k} (x - x_i), g \right) = \prod_{i= 1}^m g(x_i)\quad\mbox{for all}\;\; x_1,  \ldots, x_m \in \F, g \in \F^{\le k}[x].
\end{equation}

\end{corollary}

\begin{corollary}[from the property~\ref{enum:ResProduct} and Corollary~\ref{cor:ResSplitPoly}]\label{cor:ResCommonRoot}
If $m + k \ge 1$ and $f \in \F^{\le m}[x], g \in \F^{\le k}[x] $ have a common root in $\F$, then $\Res_{m\,k} (f,g) = 0$.
\end{corollary}

\begin{lemma}\label{lem:ResReverseGradHomogen}If $f = a_0 + \ldots + a_mx^m, g = b_0 + \ldots + b_kx^k, \lambda \in \F[x],$ then \begin{equation}\label{lem:ResReverseGradHomogen:eqq1}
\Res_{m\,k}(\lambda^m a_0, \ldots, \lambda^0 a_m; \lambda^k b_0, \ldots, \lambda^0 b_k) = \lambda^{mk} \Res_{m\,k}(a_0, \ldots, a_m; b_0, \ldots, b_k).
\end{equation}
\end{lemma}
\begin{proof}
Indeed,
\begin{multline*}
\Res_{m\,k}(\lambda^m a_0, \ldots, \lambda^0 a_m; \lambda^k b_0, \ldots, \lambda^0 b_k)\\
\overset{\eqref{lem:ResReverse:eqq1}}{=\joinrel=} (-1)^{mk} \Res_{m\,k}(\lambda^0 a_m, \ldots, \lambda^m a_0; \lambda^0 b_k, \ldots, \lambda^0 b_0)\\
\overset{\eqref{enum:ResWeightedMulLin}}{=\joinrel=} (-1)^{mk}\cdot \lambda^{mk} \Res_{m\,k}(a_m, \ldots, a_0; b_k, \ldots,  b_0)\\
\overset{\eqref{lem:ResReverse:eqq1}}{=\joinrel=}  (-1)^{mk}\cdot \lambda^{mk} \cdot (-1)^{mk} \Res_{m\,k}(a_0, \ldots, a_m; b_0, \ldots, b_k)\\
 = \lambda^{mk}\Res_{m\,k}(a_0, \ldots, a_m; b_0, \ldots, b_k).
\end{multline*}
\end{proof}

\vspace{\topsep}

\begin{definition}\label{def:PolyDiscriminant}Let $m \ge 1$ be integer and let $f \in \F[x]$ be a polynomial with 
\begin{equation}\label{def:PolyDiscriminant:eq1}
f(x) = a_0 + \ldots + a_m x^m,\quad a_0, \ldots, a_m \in \F
\end{equation}
We denote by\linebreak $\D_{m}(f) = \D(a_0,\ldots, a_m)$  \emph{a discriminant} of $f,$ which is defined by
\[
\D_m(f)  = \frac{(-1)^{n(n-1)/2}}{a_m} \Res_{m\,(m-1)}(f, f'),
\]
where $f'$ is a formal derivative $f$. In addition, if $f$ in~\eqref{def:PolyDiscriminant:eq1} is monic, then we allow to omit the subscript and use the notion $\D(f)$ instead of $\D_m(f)$.
\end{definition}

Recall the main properties of discriminant (in the square brackets we include the references to the corresponding equalities from~\cite[Chapter~12]{Gelfand1994}).

\begin{enumerate}[label=\textsc{$\D$\arabic*.}, ref=\textsc{$\D$\arabic*}]
\item $\D\left(\prod_{i = 1}^{m} (x - x_i)\right) = \prod_{1 \le i < j \le m} (x_i - x_j)^2$ for all $x_1, \ldots, x_m \in \F.$
\item\label{enum:DPropertySplit} $\D\left(\prod_{i = 1}^{m} (x - x_i)\right) = 0 \Leftrightarrow \exists 1 \le i < j \le m \colon x_i = x_j.$
\item $\D(f\cdot g) = \D(f)\cdot \D(g)\cdot \left(\Res_{m\,k}(f,g)\right)^2$ for all $f \in \F^{\le m}[x], g \in \F^{\le k}[x].$  \hfill [1.32]
\item $\D(f(x + c)) = \D(f(x))$ for all $f \in \F[x], c \in \F.$ \hfill [1.26]
\item $\D(\lambda^0 a_0, \lambda^1 a_1, \ldots, \lambda^m a_m) = \lambda^{m(m-1)} \D(a_0, \ldots, a_m)$ for all $\lambda \in \F.$ \hfill [1.24]
\end{enumerate}

In~\cite[Chapter~12]{Gelfand1994} there are also explicit formulas for discriminant of polynomial of degree 2, 3, 4 and 5 (formulas [1.33-1.36]).

\begin{lemma}[Cf.~{\cite[Formula~1.33 in Chapter~12]{Gelfand1994}}]\label{lem:DPolyN2}
Let $f =  x^2 + a_1 x + a_0$. Then 
\begin{equation}\label{lem:DPolyN2:eqq1}
\D(f) = 4a_0 - a_1^2.
\end{equation}
\end{lemma}

\begin{corollary}\label{cor:D2PolyN2Char2}
Assume that $\Char(\F) = 2$. If $f = x^2 + a_1 x + a_0$, then $\D(f) = \B^2(f)$ for $\B(f)$ defined by
\begin{equation}\label{lem:D2PolyN2Char2:eqq1}
\B(f) = a_1.
\end{equation}
\end{corollary}

\begin{lemma}[Cf.~{\cite[Formula~1.34 in Chapter~12]{Gelfand1994}}]\label{lem:DPolyN3}
Let $f = x^3 + a_2 x^2 + a_1 x + a_0$. Then 
\begin{equation}\label{lem:DPolyN3:eqq1}
\D(f) = -27a_0^2 - 4a_1^3 + 18a_0a_1a_2 + a_1^2a_2^2 - 4a_0^2a_2^3.
\end{equation}
\end{lemma}

\begin{corollary}\label{cor:D2PolyN3Char2}
Assume that $\Char(\F) = 2$. If $f = x^3 + a_2 x^2 + a_1 x + a_0$, then $\D(f) = \B^2(f)$ for $\B(f)$ defined by
\begin{equation}\label{lem:D2PolyN3Char2:eqq1}
\B(f) = a_0 + a_1a_2.
\end{equation}
\end{corollary}

\begin{corollary}\label{cor:DPolyN3Char3}
Assume that $\Char(\F) = 3$. Let $f = x^3 + a_2 x^2 + a_1 x + a_0$. Then 
\begin{equation}\label{lem:DPolyN3Char3:eqq1}
\D(f) = 2a_1^3 + a_0^2a_1^2 + 2a_0^2a_2^3.
\end{equation}
\end{corollary}

\begin{lemma}[Cf.~{\cite[Formula~1.35 in Chapter~12]{Gelfand1994}}]\label{lem:DiscN4FormulaCHN2}Let $f = x^4 + a_3 x^3 + a_2 x^2 + a_1 x + a_0$. Then 
\begin{multline}\label{lem:DiscN4FormulaCHN2:eqq1}
\D(f) = 256 a_0^3 - 27 a_1^4 + 144 a_0 a_1^2 a_2 - 128 a_0^2 a_2^2 - 4 a_1^2 a_2^3 + 16 a_0 a_2^4\\
 - 192 a_0^2 a_1 a_3 + 18 a_1^3 a_2 a_3 - 80 a_0 a_1 a_2^2 a_3 - 6 a_0 a_1^2 a_3^2 + 144 a_0^2 a_2 a_3^2\\
  + a_1^2 a_2^2 a_3^2 - 4 a_0 a_2^3 a_3^2 - 4 a_1^3 a_3^3 + 18 a_0 a_1 a_2 a_3^3 - 27 a_0^2 a_3^4.
\end{multline}
\end{lemma}

\begin{corollary}\label{cor:Disc2N4FormulaCH2}Assume that $\Char(\F) = 2$. If $f = x^4 + a_3 x^3 + a_2 x^2 + a_1 x + a_0$, then $\D(f) = \B^2(f)$ for $\B(f)$ defined by 
\begin{equation}\label{lem:Disc2N4FormulaCH2:eqq1}
\B(f) = a_1^2  + a_1 a_2 a_3 + a_0 a_3^2.
\end{equation}
\end{corollary}

We will not use the explicit formula for the discriminant of the polynomial of the 5th degree, but we will use the following direct consequence from it:

\begin{corollary}[Corollary from {\cite[Formula~1.36 in Chapter~12]{Gelfand1994}}]\label{cor:Disc2N5FormulaCH2}
Assume that $\Char(\F) = 2$. If $f = x^5 + a_4x^4 + a_3 x^3 + a_2 x^2 + a_1 x + a_0$, then $\D(f) = \B^2(f)$ for $\B(f)$ defined by 
\begin{equation}\label{lem:Disc2N5FormulaCH2:eqq1}
\B(f) = a_0^2 + a_1 a_2^2 + a_0 a_2 a_3 + a_1 a_2 a_3 a_4 + a_0 a_3^2 a_4 + a_1^2 a_4^2.
\end{equation}
\end{corollary}

From the facts above we conclude that if $\Char(\F) = 2$ and $1 \le n \le 5$, then $\D$ is a square of another polynomial. We will use this polynomial and its properties to introduce matrix invariants which are not expressible through $\D$ (see, for example, Lemma~\ref{lem:DistinctN3CH2RK1NNIL} and Lemma~\ref{lem:DistinctN3CH2RK2} and further lemmas in the corresponding section). 

\begin{lemma}\label{lem:D2ExistsAndUniqueN2to5}Assume that $1 \le n \le 5$ and $\Char(\F) = 2$. Then there exists a unique polynomial $\B \in \F[a_{0}, \ldots, a_{n-1}]$ such that $\D(f) = \B^2(f)$ for all monic $f \in \F^{\le n}[x]$.
\end{lemma}
\begin{proof}The existence of such a polynomial follows from Corollary~\ref{cor:D2PolyN2Char2}, Corollary~\ref{cor:D2PolyN3Char2}, Corollary~\ref{cor:Disc2N4FormulaCH2}, Corollary~\ref{cor:Disc2N5FormulaCH2} and from the definition of the discriminant for $n =1$. The uniqueness follows from the fact that if $\left(P_1\right)^2 = \left(P_2\right)^2$ for some $P_1, P_2 \in  \F[a_{0}, \ldots, a_{n-1}]$, then $(P_1 - P_2)^2 = 0$ and consequently $P_1 - P_2 = 0$ because the ring  $\F[a_{0}, \ldots, a_{n-1}]$ is a domain.
\end{proof}

\begin{remark}The statement of Lemma~\ref{lem:D2ExistsAndUniqueN2to5} is true for all $n$, but we do not need to establish it in the full generality in this paper.
\end{remark}

\begin{definition}Assume that $1 \le n \le 5$  and $\Char(\F) = 2$. We define a polynomial $\B \in \F[a_{0}, \ldots, a_{n-1}]$ which is defined for the monic polynomials belonging to $\F^{\le n}[x]$ as it is done in Corollary~\ref{cor:D2PolyN2Char2}, Corollary~\ref{cor:D2PolyN3Char2}, Corollary~\ref{cor:Disc2N4FormulaCH2}  and Corollary~\ref{cor:Disc2N5FormulaCH2}. For $n = 1$ we define $\B$ by $\B = 1$.

Lemma~\ref{lem:D2ExistsAndUniqueN2to5} implies that $\B$ is uniquely determined by the equality $\D = \B^2$.
\end{definition}

Let us list the properties of $\B$ following immediately from the definition of $\D$ (here we assume that $1 \le n \le 5$). 

\begin{enumerate}[label=\textsc{$\B$\arabic*.}, ref=\textsc{$\B$\arabic*}]
\item If $f  = (x - a_1)\ldots (x - a_n),$ then
$$
\B(a_1,\ldots, a_n)=\B(f) = \prod_{i < j} (a_i + a_j).
$$
\item Suppose that $f \in \F[x]$ is monic and $g(x) = f(x+c)$ for some $c \in \F$. Then $\B(g) = \B(f).$
\item $\B(f\cdot g) = \B(f) \Res_{m\, k} (f, g) \B(g)$ for all monic $f \in \F^{\le m}[x], g \in \F^{\le k}[x]$ such that $1 \le m \le 5$. 
\end{enumerate}

\vspace{\topsep}

\begin{definition}If $X \in \M_{m}(\F)$, then by $\D (X)$ we denote a discriminant of the characteristic polynomial of $X.$ It is called  \emph{a discriminant} of $X.$ In other words,
\[
\D(X) = \D(p_X(x)) = \D (\det (x I_m - X)).
\]
\end{definition}

It is clear that matrix discriminant is a polynomial of matrix coefficients. The following properties of the matrix discriminant are derived from the corresponding properties of the polynomial discriminant.

\begin{enumerate}[label=\textsc{M$\D$\arabic*.}, ref=\textsc{M$\D$\arabic*}, leftmargin=*]
\item\label{enum:DiscPropertyDiag} If $X = \diag(a_1, \ldots, a_m)$, then
\[
\D(X) = \prod_{1 \le i < j \le m} (a_i - a_j)^2.
\]
\item\label{enum:DiscPropertyProd} If $X \in \M_{m}, Y \in \M_{k}$, then
\[
\D(X\oplus Y) = \D\left(p_{X}(x)p_{Y}(x)\right) = \D(X)\cdot \D(Y) \cdot \left(\Res_{m\,k}(p_{X}(x),p_{Y}(x))\right)^2.
\]
\item\label{enum:DiscPropertyC} If $c \in \F$,  then $\D(X + cI_m) = \D(X).$
\item\label{enum:DiscPropertyT} If $X \in \M_{n}(\F),$ then $\D(X^t) = \D(X)$.
\item\label{enum:DiscPropertyS} $\D(M^{-1}XM) = \D(X)$ for all invertible matrices $M \in \M_m(\F).$
\item\label{enum:DiscPropertyH} $\D(\lambda X) = \lambda^{m(m-1)}\D(X)$ for all $\lambda \in \F.$
\end{enumerate}
\noindent In addition, the property~\ref{enum:DiscPropertyH} implies that:
\begin{enumerate}[resume, label=\textsc{M$\D$\arabic*.}, ref=\textsc{M$\D$\arabic*}, leftmargin=*]
\item\label{enum:DiscPropertyHP} $\D$ is a homogeneous polynomial of matrix coefficients of degree $m(m-1).$
\end{enumerate}

\begin{remark}
The above formulas show how the matrix discriminant of a matrix $X$ is expressed as a resultant of its characteristic polynomial $p_X$ and the derivative of $p_X$. Apart from that, is also possible to express $\D(X)$ through the traces of powers of $X$  (see~\cite{Parlett2002}). 
\end{remark}

\begin{lemma}\label{lem:N3CharPoly}Let $A \in \M_{2}(\F)$. Then 
\begin{equation}\label{eq:D2Formula}
\D(A) = 4\det(A) - \tr(A)^2.
\end{equation}
\end{lemma}
\begin{proof}Indeed, if $A \in \M_{2}(\F)$, then $p_A(x) = x^2 - \tr(A)x + \det(A)$ by the formula~\eqref{eqqq:formulaChPoly}. Then, by applying Lemma~\ref{lem:DPolyN2} to $p_A$, we obtain~\eqref{eq:D2Formula}.
\end{proof}

\begin{lemma}\label{lem:DiscDirSum}Let $A \in \M_{(n + m)}(\F), A_0 \in \M_n(\F)$ and $A' \in \M_{m}(\F)$ be such that $A'$ is strictly upper triangular and
\[
A = \left[\begin{array}{ c | c }
    A_0 & * \\
    \hline
    0 & A'
  \end{array}\right],
\]
where $*$ is an $n\times m$ matrix. Then for every $B_0 \in \M_{n}(\F)$ and $b_1,\ldots, b_m \in \F$ 
\begin{equation}\label{lem:DiscDirSum:eqq1}
\D\left(A + B\right) = \D(A_0 + B_0)  \prod_{1 \le k \le m} p_{(A_0 + B_0)}^2 (b_k) \prod_{1 \le i < j \le m} (b_i - b_j)^2,
\end{equation}
where $B = \left(B_0 \oplus \diag(b_1, \ldots, b_k)\right).$
\end{lemma}
\begin{proof}Indeed, since $A'$ is strictly upper triangular and $*$ does not affect on $p_{(A+B)}$, then
\begin{equation}\label{lem:DiscDirSum:eq1}
p_{(A+B)} = p_{(A_0 + B_0)} \oplus \diag(b_1, \ldots, b_m)
\end{equation}
and consequently
\begin{multline*}
\D\left(A + B\right) = \D(p_{(A + B)})
\overset{\eqref{lem:DiscDirSum:eq1}}{=\joinrel=\joinrel=}\D \left(p_{(A_0 + B_0)} \oplus p_{\diag(b_1,\ldots, b_m)}\right)\\
 \overset{\mbox{\scriptsize Property~\ref{enum:DiscPropertyProd}}}{=\joinrel=\joinrel=\joinrel=\joinrel=\joinrel=\joinrel=\joinrel=\joinrel=} \D(A_0 + B_0) \Res_{n\,m}^2(p_{(A_0 + B_0)}, p_{\diag(b_1,\ldots, b_m)}) D(\diag(b_1,\ldots, b_m))\\
= \D(A_0 + B_0) \Res_{n\,m}^2(p_{(A_0 + B_0)}, \prod_{1 \le k \le m} (x - b_k) D(\diag(b_1,\ldots, b_m)\\
\overset{\mbox{\scriptsize Property~\ref{enum:DiscPropertyDiag}}}{=\joinrel=\joinrel=\joinrel=\joinrel=\joinrel=\joinrel=\joinrel=} \D(A_0 + B_0) \Res_{n\,m}^2(p_{(A' + B_0)}, \prod_{1 \le k \le m} (x - b_k) \prod_{1 \le i < j \le n} (b_i - b_j)^2\\
\overset{\mbox{\scriptsize Property~\ref{enum:ResProduct}}}{=\joinrel=\joinrel=\joinrel=\joinrel=\joinrel=\joinrel=\joinrel=} \D(A_0 + B_0)   \prod_{1 \le k \le n} p_{(A_0 + B_0)}^2 (b_k) \prod_{1 \le i < j \le n} (b_i - b_j)^2.
\end{multline*}
\end{proof}

\begin{definition}Assume that $\Char(\F) = 2$ and $1 \le m \le 5$. If $X \in \M_{m}(\F)$, then $\B(X)$ is defined by
\[
\B(X) = \B(p_X(x)) = \B(\det (x I_m - X)).
\]
\end{definition}

The following properties of $\B$ are derived from the corresponding properties of $\D$.

\begin{enumerate}[label=\textsc{M$\B$\arabic*.}, ref=\textsc{M$\B$\arabic*}, leftmargin=*]
\item If $X = \diag(a_1, \ldots, a_m)$, then
\[
\B(X) =\prod_{i < j} (a_i + a_j).
\]
\item If $X \in \M_{m}, Y \in \M_{k}$, then
\[
\B(X\oplus Y) = \B\left(p_{X}(x)p_{Y}(x)\right) = \B(X) \cdot \left(\Res_{m\, k}(p_{X}(x),p_{Y}(x))\right) \cdot \B(Y).
\]
\item\label{enum:D2PropertyC} If $c \in \F$,  then $\B(X + cI_m) = \B(X).$
\item\label{enum:D2PropertyS} $\B(M^{-1}XM) = \B(X)$ for all invertible matrices $M \in \M_m(\F).$
\item\label{enum:D2PropertyH} $\B(\lambda X) = \lambda^{\frac{m(m-1)}{2}}\B(X)$ for all $\lambda \in \F.$
\end{enumerate}

\begin{lemma}\label{lem:PresDEqPresD2}
Assume that $\Char(\F) = 2$. Let $T\colon \M_{n}(\F) \to \M_{n}(\F)$ be a linear map. Then $T$ preserves $\D$ if and only if $T$ preserves $\B$.\end{lemma}

\begin{proof}Indeed, if $\Char (\F) = 2$, then the conditions $a^2 = b^2$ are $a = b$ equivalent for $a, b \in \F$. 
\end{proof}

\section{The radical of a function}
\label{sec:FunctionRad}

The property~\ref{enum:DiscPropertyC} of the matrix discriminant motivates us to use the notion of the radical of function which was introduced by Waterhouse in~\cite{Waterhouse1983}.

\begin{definition}[Cf.~{\cite[text at the beginning of Section~1]{Waterhouse1983}}]\label{def:radical}Let $\F$ be a field, $V$ be a finite-dimensional vector space over $\F$, and let $f$ be a function from $V$ to $\F$. The \emph{radical} of $f$, denoted by $\rad (f)$ is a subset of $V$ defined by
\[
\rad (f) = \{\mathbf w \mid f(\mathbf v + \lambda \mathbf w) = f(\mathbf v)\;\;\mbox{for all}\;\; \mathbf v \in V,\;\; \lambda \in \F\}.
\]
\end{definition}

The following properties of the radical are used throughout this paper.

\begin{lemma}\label{lem:KernelContainedInRad}If $T\colon V \to V$ is a linear map preserving $f$, then $\Ker T \subseteq \rad(f)$.
\end{lemma}
\begin{proof}Indeed, assume that $\mathbf w \in \Ker T.$ Then
\begin{equation*}
f(\mathbf v + \lambda\mathbf w)= f(T(\mathbf v + \lambda\mathbf w))\\
= f(T(\mathbf v) + \lambda T(\mathbf w)) = f(T(\mathbf v) + \mathbf 0) = f(\mathbf v)
\end{equation*}
for all $\mathbf v \in V$ and $\lambda \in \F$. Therefore, $\mathbf w \in \rad(f)$.
\end{proof}

\begin{remark}The lemma above is implicitly proved in~\cite{Waterhouse1983} (see~\cite[Proposition~1]{Waterhouse1983}).
\end{remark}

\begin{lemma}[Cf.~{\cite[Proposition~2]{Waterhouse1983}}]\label{lem:RadicalInvariant}
Let $V$ be a vector space over $\F$, $f\colon V \to \F$ be an arbitrary function. If $f$ has trivial radical, then the set of all linear maps preserving $f$ form a group. If $f$ has nontrivial radical $\rad(f)$, let $\pi_{rad}\colon V \to V/\rad(f)$ be the projection, and write $f = f_{\rad} \circ \pi_{\rad}$. Then the linear map $T \colon V \to V$ preserves $f$ if and only if
\begin{enumerate}[label=(\alph*), ref=(\alph*)]
\item\label{lem:RadicalInvariant:part1} $T(\rad(f)) \subseteq \rad(f)$  and
\item\label{lem:RadicalInvariant:part2} the map $T_{\rad} \colon V / \rad(f) \to V / \rad(f)$ that is given by $T_{\rad}(\pi_{\rad}(\mathbf v)) = \pi_{\rad}(T(\mathbf v))$ (and is well defined because of~\ref{lem:RadicalInvariant:part1}) preserves $f_{\rad}$.
\end{enumerate}
\end{lemma}

\section{Supplementary facts from the theory of polynomials and linear algebra}
\label{sec:SupLinAlg}

In order to have the possibility to convert the set-theoretic condition ``$P(\mathbf x) = Q(\mathbf x)$ for all $\mathbf x \in \F^n$'' for given polynomials $P$ and $Q$ to the algebraic condition on the coefficients of the polynomial, we use the following general facts regarding polynomials over a field. 

\begin{lemma}[{Cf.~\cite[IV, \S 1, Corollary 1.7.]{Lang2002}}]Let $\F$ be an infinite field. Let $f$ be a
polynomial in $n$ variables over $\F$. If $f$ defines the zero function on $\F^{n}$, then $f = 0$.
\end{lemma}

\begin{lemma}[{Cf.~\cite[IV, \S 1, Corollary 1.8]{Lang2002}}]Let $\F$ be a finite field with $q$ elements. Let $f$ be a
polynomial in $n$ variables over $\F$ such that the degree of $f$ in each variable
is less than $q$. If $f$ defines the zero function on $\F^{n}$, then $f = 0$.
\end{lemma}

\begin{corollary}\label{DegFLessZero}Let $\F$ be an arbitrary field. Let $f$ be a
polynomial in $n$ variables over $\F$ such that either $\F$ is infinite or the degree of $f$ in each variable
is less than $|\F|$. If $f$ defines the zero function on $\F^{n}$, then $f = 0$.
\end{corollary}

\begin{corollary}\label{DegFGLessEqual}Let $f, g$ be a
polynomials in $n$ variables over $\F$ such that either $\F$ is infinite or the degree of $f$ and $g$ in each variable
is less than $|\F|$. If $f$ and $g$ define the same function on $\F^{n}$, then $f = g$.
\end{corollary}

\begin{lemma}\label{lem:ManyPolyDensityArg}
Let $\mathcal S \subseteq \F[x_1, \ldots, x_n]$ and $0 \neq Q \in \F[x_1,\ldots, x_n]$  be such that $\deg(P) + \deg(Q) < |\F|$ for all $P \in \mathcal S$. If 
\begin{equation}\label{lem:ManyPolyDensityArg:eqq1}
\left(Q(\mathbf x) \neq 0\right) \Rightarrow \left(\forall P \in \mathcal S\colon P(\mathbf x) = 0 \right)\;\;\mbox{for all}\;\;\mathbf x \in \F^n,
\end{equation}
  then 
\begin{equation}\label{lem:ManyPolyDensityArg:eqq2}
P(\mathbf x) = 0\;\;\mbox{for all}\;\;\mathbf x \in \F^n, P \in \mathcal S\;\;\mbox{(i.e., $\mathcal S = \{0\}$)}.
\end{equation}
\end{lemma}
\begin{proof}Indeed, if $0 \neq P \in \mathcal S$, then $PQ \neq 0$ and $\deg(PQ) = \deg(P) + \deg(Q) \le |\F|$ by the assumptions of the lemma. Hence, there is $\mbf x \in \F^n$ such that $PQ(\mathbf x) \neq 0$ by Corollary~\ref{DegFLessZero}. This implies that $Q(\mathbf x) \neq 0$ and $P(\mbf x) \neq 0$, which contradicts~\eqref{lem:ManyPolyDensityArg:eqq1}. Therefore, the equality~\eqref{lem:ManyPolyDensityArg:eqq2} holds for all $\mathbf x \in \F^n, P \in \mathcal S.$
\end{proof}

As we mentioned in the introduction, the proof of the main theorem involves in particular matrices of rank one. For the convenience, the following representation of such matrices is used.

\begin{definition}[{\cite[Definition~2.14]{Yurkov2025}}]\label{def:JabMat}
Given $\mbf a, \mbf b \in \F^n,$ we denote
\[
J(\mbf a, \mbf b) = \begin{pmatrix}
\mbf a_1\begin{pmatrix}
\mbf b_1\\
\vdots\\
\mbf b_n\\
\end{pmatrix} & \cdots & \mbf a_n\begin{pmatrix}
\mbf b_1\\
\vdots\\
\mbf b_n\\
\end{pmatrix}
\end{pmatrix}.
\]
i.e., $(J(\mbf a,\mbf b))_{i\,j} = \mbf b_i \mbf a_j,$ $1 \le i, j \le n.$
\end{definition}

\begin{remark}
It is possible to use the language of tensor algebra  to discuss about the matrices of rank 1 as it done in certain papers (e.g. in~\cite{Botta1983}). However, it requires the explicit specification of correspondence between tensors and matrices.
\end{remark}

Let us provide the main properties of matrices having the form $J(a,b)$ for some $a, b\in \F^n$. The proofs of the properties~\ref{prop:J1} and~\ref{prop:J2} are trivial. The properties~\ref{prop:JABUNiqExpr}, \ref{prop:RankOneExpression} and~\ref{prop:JNilpotentCrit} are proven in~\cite{Yurkov2025}. 

\begin{enumerate}[label=\textsc{J\arabic*.}, ref=\textsc{J\arabic*}]
\item\label{prop:J0} $\rk \left(J(\mbf a, \mbf b)\right) \le 1$ for all $\mbf a, \mbf b \in \F^n$.
\item\label{prop:J1} $J(\alpha \mbf a, \mbf b) = \alpha J(\mbf a, \mbf b)$  and $J(\mbf a,\beta \mbf b) = \beta J(\mbf a,\mbf b)$ for all $\mbf a, \mbf b \in \F^n$ and $\alpha, \beta \in \F.$
\item\label{prop:J2} $J(\mbf a,\mbf b+\mbf c) = J(\mbf a,\mbf b) + J(\mbf a,\mbf c)$ and $J(\mbf a + \mbf b, \mbf c) = J(\mbf a, \mbf c) + J(\mbf b,\mbf c)$ for all $\mbf a, \mbf b, \mbf c \in \F^n.$
\item\label{prop:JABUNiqExpr} If $J(\mbf a,\mbf b) = J(\mbf a',\mbf b')$ for some $\mbf a, \mbf b, \mbf a', \mbf b' \in \F^n$, then either $J = O_n$ or $\mbf b' = \frac{1}{\alpha}\mbf b$ for some nonzero $\alpha \in \F.$
\item\label{prop:RankOneExpression} Let $X \in \M_n(\F)$ and $0 \neq \mathbf v \in \F^n$. Then $\colsp(X) = \langle\mathbf v\rangle$ if and only if $X = J(\mathbf a,\mathbf v)$ for some nonzero $a \in \F^n.$
\item\label{prop:JNilpotentCrit} $J(\mbf a,\mbf b)$ is a nilpotent matrix if and only if $(\mbf a, \mbf b) = 0.$

\end{enumerate}

The following properties of certain kinds of rank one matrices are also used.

\begin{lemma}\label{lem:MatRkOneConstDiagDecomp} If $C \in \M_{n}(\F)$ is such that $\rk(C) = 1$ and $c_{i\,i} = c_{j\,j} \neq 0$ for all $1 \le i, j \le n$, then there are $0\neq c, c_1, \ldots, c_n \in \F$ such that
\begin{equation}\label{lem:MatRkOneConstDiagDecomp:eqq1}
c_{i\,j} = c c_{i} c^{-1}_{j}\;\;\mbox{for all}\;\;1 \le i, j \le n.
\end{equation}
\end{lemma}
\begin{proof}Let $c \in F$ be the common value of the entries belonging to the main diagonal of $C$, i.e., $c$ is such that 
\begin{equation}\label{lem:MatRkOneConstDiagDecomp:eq2}
c = c_{i\,i} \neq 0\;\;\mbox{for all}\;\;1 \le i \le n.
\end{equation}
Since $\rk(C) = 1$, then $\colsp(C) = \langle\mathbf v\rangle$ for some $0 \neq \mathbf v \in \F^n$. Hence, there is $0 \neq \mbf a \in \F^n$ such that
\begin{equation}\label{lem:MatRkOneConstDiagDecomp:eq1}
C = J(\mathbf a,\mathbf v).
\end{equation}
The equality~\eqref{lem:MatRkOneConstDiagDecomp:eq1} implies in particular that
\begin{equation}\label{lem:MatRkOneConstDiagDecomp:eq3}
\mathbf a_{i}\mathbf v_{i} = c_{i\,i} \overset{\eqref{lem:MatRkOneConstDiagDecomp:eq2}}{=\joinrel=\joinrel=} c\;\;\mbox{for all}\;\;1 \le i \le n.
\end{equation}
Thus, since $c \neq 0$, then $\mbf a_{i}, \mbf v_{i} \neq 0$ for all $1 \le i \le n$. This allows to rewrite~\eqref{lem:MatRkOneConstDiagDecomp:eq3} as follows
\begin{equation}\label{lem:MatRkOneConstDiagDecomp:eq4}
\mathbf a_{i} =  c \mathbf v^{-1}_{i}\;\;\mbox{for all}\;\;1 \le i \le n.
\end{equation}
Now, by setting 
\begin{equation}\label{lem:MatRkOneConstDiagDecomp:eq5}
c_i = \mathbf v_{i}\;\;\mbox{for all}\;\;1 \le i \le n
\end{equation}
we obtain that $c$ and $c_1, \ldots, c_n$ satisfy the condition~\eqref{lem:MatRkOneConstDiagDecomp:eqq1}. Indeed,
\begin{multline*}
c_{i\,j} = \left(J(\mathbf a,\mathbf v)\right)_{i\,j} \overset{\mbox{\scriptsize Definition~\ref{def:JabMat}}}{=\joinrel=\joinrel=\joinrel=\joinrel=\joinrel=\joinrel=\joinrel=\joinrel=} \mbf a_{j} \mbf v_{i} = \mbf v_{i} \mbf a_{j}\\
\overset{\eqref{lem:MatRkOneConstDiagDecomp:eq4}}{=\joinrel=\joinrel=}  \mbf v_{i} \left(c \mbf v^{-1}_{j}\right) \overset{\eqref{lem:MatRkOneConstDiagDecomp:eq5}}{=\joinrel=\joinrel=} c  c_{i} c_{j}^{-1}\;\;\mbox{for all}\;\;1 \le i, j \le n.
\end{multline*}
\end{proof}

\begin{lemma}\label{lem:RK1Sim} Let $\F$ be an arbitrary field and $J \in \M_{n}(\F)$ be a rank-one matrix. Then
\begin{enumerate}[label=(\alph*), ref=(\alph*)]
\item $J$ is nilpotent if and only if $J$ is similar to $E_{1\,2}$.
\item $J$ is non-nilpotent if and only if $J$ is similar to $cE_{1\,1}$ for some $0 \neq c \in \F$.
\end{enumerate} 
\end{lemma}
\begin{proof}It is clear that if $J$ is similar to $E_{1\,2}$, then it is nilpotent, and if $J$ is similar to $cE_{1\,1}$ for some $0 \neq c \in \F$, then $J$ is non-nilpotent. Thus, it is suffices to show the following:
\begin{enumerate}[label=Claim~(\roman*), ref=(\roman*), leftmargin=*]
\item\label{lem:RK1Sim:claim1} If $J$ is nilpotent, then $J$ is similar to $E_{1\,2}$.
\item\label{lem:RK1Sim:claim2} If $J$ is non-nilpotent, then $J$ is similar to $cE_{1\,1}$ for some $0 \neq c \in \F$.
\end{enumerate}

\paragraph{Proof of Claim~\ref{lem:RK1Sim:claim1}} Let $0 \neq \mbf v$ be such that $\colsp (J) = \langle v \rangle$.  Since $J$ is nilpotent, then 
\begin{equation}\label{lem:RK1Sim:eq12}
J\mbf v = \mbf 0
\end{equation}
and consequently $\mbf v \in \ker J$. Let $\mbf v'$ be such that
\begin{equation}\label{lem:RK1Sim:eq11}
J\mbf v' = \mbf v.
\end{equation}

By rank-nullity theorem, $\dim(\ker J) = n - 1$. Let us extend $\mbf v$ to a basis $\mbf v^{(1)} = \mbf v, \mbf v^{(2)}, \ldots, \mbf v^{(n-1)}$ of $\ker J$. The equality~\eqref{lem:RK1Sim:eq11} implies that the vectors $\mbf v', \mbf v^{(1)}, \ldots, \mbf v^{(n-1)}$ are linearly independent. Therefore, these vectors form a basis of $\F^n$. 

Let $A \in \M_{n}(\F)$ be a matrix such that 
\begin{equation}\label{lem:RK1Sim:eq1}
A\mbf e^{(1)} = \mbf v^{(1)},\; A\mbf e^{(2)} = \mbf v',\;\mbox{and}\; A\mbf e^{(i)} = A\mbf v^{(i-2)}\;\;\mbox{for all}\;\;3 \le i \le n.
\end{equation}
Since $\mbf v', \mbf v^{(1)}, \ldots, \mbf v^{(n-1)}$ form a basis of $\F^n$, then $A$ is invertible. Let $J' = A^{-1}JA$. Then
\begin{equation*}
J'\mbf e^{(1)} = A^{-1}JA\mbf e^{(1)} \overset{\eqref{lem:RK1Sim:eq1}}{=\joinrel=\joinrel=} A^{-1}J\mbf v \overset{\eqref{lem:RK1Sim:eq12}}{=\joinrel=\joinrel=} \mbf 0,
\end{equation*}
\begin{equation*}
J'\mbf e^{(2)} = A^{-1}JA\mbf e^{(2)} \overset{\eqref{lem:RK1Sim:eq1}}{=\joinrel=\joinrel=} A^{-1}J\mbf v' \overset{\eqref{lem:RK1Sim:eq11}}{=\joinrel=\joinrel=} A^{-1}\mbf v \overset{\eqref{lem:RK1Sim:eq1}}{=\joinrel=\joinrel=} \mbf e^{(1)},
\end{equation*}
and
\begin{equation*}
J'\mbf e^{(i)} = A^{-1}JA\mbf e^{(i)} \overset{\eqref{lem:RK1Sim:eq1}}{=\joinrel=\joinrel=} A^{-1}J\mbf v^{(i-2)} \overset{\mbf v^{(i-2)} \in \ker J}{=\joinrel=\joinrel=\joinrel=\joinrel=\joinrel=\joinrel=} \mbf 0\;\;\mbox{for all}\;\;3 \le i \le n.
\end{equation*}
Therefore, $J' = E_{1\,2}$  and Claim~\ref{lem:RK1Sim:claim1} is established.
 
\paragraph{Proof of Claim~\ref{lem:RK1Sim:claim2}} Let $0 \neq \mbf v$ be such that $\colsp (J) = \langle v \rangle$.  Since $J$ is non-nilpotent, then 
\begin{equation}\label{lem:RK1Sim:eq13}
J\mbf v = c \mbf v\;\;\mbox{for some}\;\;0 \neq c \in \F.
\end{equation}
By rank-nullity theorem, $\dim(\ker J) = n - 1$. Let $\mbf u^{(1)}, \ldots, \mbf u^{(n-1)}$ be a basis of $\ker J$. The equality~\eqref{lem:RK1Sim:eq13} implies that $\mbf v, \mbf u^{(1)}, \ldots, \mbf u^{(n-1)}$ are linearly independent. Therefore, these vectors form a basis of $\F^n$. 

Let $A \in \M_{n}(\F)$ be a matrix such that 
\begin{equation}\notag
A\mbf e^{(1)} = \mbf v^{(1)},\;\mbox{and}\;  A\mbf e^{(i)} = A\mbf u^{(i-1)}\;\;\mbox{for all}\;\;2 \le i \le n.
\end{equation}
Since $\mbf v, \mbf u^{(1)}, \ldots, \mbf u^{(n-1)}$ form a basis of $\F^n$, then $A$ is invertible. Let $J' = A^{-1}JA$. By the same calculations as for the previous case we obtain that $J' = cE_{1\,1}$ which establishes Claim~\ref{lem:RK1Sim:claim2}.
\end{proof}

The property~\ref{enum:DiscPropertyC} of the matrix discriminant implies that if the pair the two matrices are differ by a scalar matrix, then they are indistinguishable by $\D$ and any invariants derived from it. Accordingly, we consider the following class of matrices extending the class of all rank one matrices.

\begin{definition}\label{def:Rset}By $\R_n(\F) \subseteq \M_{n}(\F)$ we denote the set of all matrices of the form $J + cI_n$, where $\rk(J) \le 1$ and $c \in \F$.
\end{definition}

\begin{lemma}\label{lem:RsetProp}If $M \in \M_{n}(\F)$ is an invertible matrix and $c \in \F$, then\linebreak $M\R_n(\F)M^{-1} + cI_n = \R_n(\F)$.
\end{lemma}
\begin{proof}It is clear that $\R_n(\F) + cI_n = \R_n(\F)$. Hence, it is sufficient to show only that $M\R_n(\F)M^{-1} = \R_{n}(\F)$. 
Let $X \in \R_n(\F)$. By the definition, there are $J \in \M_{n}(\F)$ and $c \in \F$ such that $\rk(J) \le 1$ and $X = J + cI_n$. Then 
\begin{equation}\label{lem:RsetProp:eq1}
MXM^{-1} = MJM^{-1} + cI_n.
\end{equation}
The similarity transformation preserves the rank of a matrix 
(see, for example~\cite[Theorem 3.4]{Friedberg2002-wu}). Therefore, the equality~\eqref{lem:RsetProp:eq1} implies that 
\begin{equation}\label{lem:RsetProp:eq2}
MXM \in \R_{n}(\F).
\end{equation}
The inclusion~\eqref{lem:RsetProp:eq2} holds for all $X \in \R_n(\F)$. This establishes the inclusion $M\R_n(\F)M^{-1} \subseteq \R_{n}(\F)$. In order to prove that $M\R_n(\F)M^{-1} = \R_{n}(\F)$ we need to establish the reverse inclusion, which is obtained similarly.
\end{proof}

\begin{lemma}\label{lem:RankOneScalarFromExtToGroundField}Assume that $\F' \supseteq \F$ are fields and $n \ge 3$. Let $B \in \M_{n}(\F)$. If $B \in \R_n(\F'),$ then $B \in \R_n(\F).$
\end{lemma}
\begin{proof}Let  $B \in \M_{n}(\F)$, $c' \in \F'$ and $J'\in \M_{n}(\F')$ be such that
\begin{equation}\label{lem:RankOneScalarFromExtToGroundField:eqq1}
B =  J' + c'I_n
\end{equation}
and $\rk(J') \le 1$. If $\rk(J')=0$, then $J' = O_n$. Hence, the equality~\eqref{lem:RankOneScalarFromExtToGroundField:eqq1} and the condition $B \in \M_{n}(\F)$ imply that $c' \in \F$. Therefore, $J \in \M_{n}(\F)$ and $c' \in \F$ which implies that $B \in \R_{n}(\F)$.

Assume now that $\rk(J') = 1$. Note that this equality~\eqref{lem:RankOneScalarFromExtToGroundField:eqq1} implies that all entries of $J'$ not lying on its main diagonal belong to $\F$. There are two cases possible:
\begin{enumerate}
  \item All entries of $J'$ not lying on the main diagonal are equal to zero.
  \item There are $1 \le i \neq j \le n$ such that $J'_{i\,j} \neq 0$.
\end{enumerate}
We consider them separately.

\paragraph{All entries of $J'$ not lying on the main diagonal are equal to zero} Since $\rk\left(J'\right) = 1$, then this assumption implies that there is only one non-zero element in $J'$ lying on the main diagonal. Since $n \ge 3$, from this and~\eqref{lem:RankOneScalarFromExtToGroundField:eqq1} we conclude that $c' \in \F$. Therefore, $J' \in \M_{n}(\F)$. Thus, all the summands in the right-hand side of~\eqref{lem:RankOneScalarFromExtToGroundField:eqq1} belong to $\M_{n}(\F)$ consequently  $B \in \R_{n}(\F)$.

\paragraph{There are $1 \le i \neq j \le n$ such that $J'_{i\,j} \neq 0$} Since $n \ge 3,$ there is $1 \le k \le n$ such that $i\,j \neq k$. Since $J'$ is a rank-one matrix, then every $2\times 2$ minor of $J'$ is equal to zero. In particular, $\det\left(J''[i, k| j, k]\right) = 0.$ By expressing it through elements of $J'',$ we obtain that
\[
J'_{i\,j}J'_{k\,k} - J'_{i\,k} J'_{k\,j} = 0.
\] 
Since $J'_{i\,k}$ and  $J'_{k\,j}$ are not lying on the main diagonal of $J''$, then they belong to $\F$. Hence, 
\[
J'_{i\,j}J'_{k\,k} \in \F
\]
Similarly to $J'_{i\,k}$ and  $J'_{k\,j}$, the multiplier $J'_{i\,j}$ belongs to $\F$ and is not equal to zero by the assumption made for this case. Hence, $J'_{k\,k} \in \F$. By considering the $(k,k)$-th elements in the summands of both the sides of the equality~\eqref{lem:RankOneScalarFromExtToGroundField:eqq1}, we obtain that
\[
b_{k\,k} = c' + J'_{k\,k} \in \F.
\]
Therefore, $c' \in \F$, which implies that $J' \in \M_{n}(\F)$. This means that all the summands in the right-hand side of~\eqref{lem:RankOneScalarFromExtToGroundField:eqq1} belong to $\M_{n}(\F)$ which infers that $B \in \R_{n}(\F)$.
\end{proof}

\begin{lemma}\label{lem:DiagonalR1Crit}Let $A \in \M_{n}(\F)$ be a diagonal matrix. Then $A \in \R_{n}(\F)$ if and only if
\begin{equation}\label{lem:DiagonalR1Crit:eqq1}
\forall 1 \le i_1 < i_2 < i_3 \le n\;\; \exists 1 \le \alpha < \beta \le 3  \colon a_{i_\alpha\,i_\alpha} = a_{i_\beta\,i_\beta} 
\end{equation}
and
\begin{equation}\label{lem:DiagonalR1Crit:eqq2}
\forall 1 \le i_1 < i_2 < i_3 < i_4 \le n\;\; \exists 1 \le \alpha < \beta < \gamma \le 4 \colon a_{i_\alpha\,i_\alpha} = a_{i_\beta\,i_\beta} = a_{i_\gamma\,i_\gamma}.
\end{equation}
\end{lemma}
\begin{proof}Assume that $A \in \R_{n}$. This implies that there is $c \in \F$ such that $\rk(A - cI_n) \le 1$. Hence, $A - cI_n$ has at most one nonzero entry and consequently satisfies the conditions~\eqref{lem:DiagonalR1Crit:eqq1} and~\eqref{lem:DiagonalR1Crit:eqq2}. These conditions are preserved by addition of a scalar matrix. Therefore, $A$ satisfies these conditions as well.

Assume now that $A$ is diagonal matrix such that the conditions~\eqref{lem:DiagonalR1Crit:eqq1} and~\eqref{lem:DiagonalR1Crit:eqq2} hold. The condition~\eqref{lem:DiagonalR1Crit:eqq1} implies that $|\{a_{i\,i} \mid 1 \le i \le n\}| \le 2$. If $|\{a_{i\,i} \mid 1 \le i \le n\}| = 1$, then $A - a_{1\,1} = O_n$ and the statement of the lemma is established. Consider now the case where $|\{a_{i\,i} \mid 1 \le i \le n\}| = 2$. 

Let $1 \le i_1 < i_2 \le n$ are such that $a_{i_1\,i_1} \neq a_{i_2\,i_2}$. The condition~\eqref{lem:DiagonalR1Crit:eqq2} implies that either $|\{i \mid a_{i\,i} = a_{i_1\,i_1}, 1 \le i \le n\}| = 1$, or $|\{i \mid a_{i\,i} = a_{i_2\,i_2}, 1 \le i \le n\}| = 1.$ Assume without loss of generality that $|\{i \mid a_{i\,i} = a_{i_1\,i_1}, 1 \le i \le n\}| = 1.$ Then $A - a_{i_2\,i_2}I_n$ is a diagonal matrix with only one nonzero entry. Therefore, $\rk(A - a_{i_2\,i_2}I_n) = 1$.
\end{proof}

\begin{lemma}\label{lem:DegDetAMinusXIn} Let $A \in \M_n(\F)$, $1 \le i_1 \neq i_2 \le n$, $1 \le j_1 \neq j_2 \le n$ and  $P \in \F[x; a_{1\,1}, \ldots, a_{n\,n}$ be defined by
\[
P = \det\left(A^{(x)}[i_1,i_2|j_1,j_2]\right),
\]
where  $A^{(x)} \in \M_{n}(\F)$ is defined by $A^{(x)} = A - xI_n$. Then 
\begin{enumerate}[label=(\alph*), ref=(\alph*)]
\item\label{lem:DegDetAMinusXIn:prop1} $\deg_x(P) \le 2$.
\item\label{lem:DegDetAMinusXIn:prop2} If $\deg_x(P) = 2$, then ($i_1 = j_1$ and $i_2 = j_2$) or ($i_1 = j_2$ and $i_2 = j_1$). 
\item\label{lem:DegDetAMinusXIn:prop3} $[i]P$ is a homogeneous polynomial of degree $2-i$ for all $0 \le i \le 2$.
 \end{enumerate}
\end{lemma}

\begin{proof}The properties~\eqref{lem:DegDetAMinusXIn:prop1} and~\eqref{lem:DegDetAMinusXIn:prop3} follow directly from the fact that the determinant of $2\times 2$ matrix is a homogeneous polynomial of degree 2 in the entries of a given matrix.

Assume now that $\deg_x(P) = 2$. This implies that the matrix
$A^{(x)}[i_1,i_2|j_1,j_2]$ contains two diagonal elements of the matrix $A^{(x)}$. Therefore, $i_1 = j_1$ and $i_2 = j_2$.
\end{proof}

\begin{lemma}\label{lem:RDetProp}Let $f, g, h \in \F^{\le 2}[x]$. Consider $R_{f\,g\,h} \in \M_{3}(\F)$ defined by
\begin{equation}\label{lem:RDetProp:eqq1}
R_{f\,g\,h} = \begin{pmatrix}[2]f & [1]f & [0]f\\ [2]g & [1]g & [0]g\\ [2]h & [1]h & [0]h\end{pmatrix}.
\end{equation}
Then 
\begin{enumerate}[label=(\alph*), ref=(\alph*)]
\item\label{lem:RDetProp:prop1} if $c \in \F$,  $f' = f(x + c), g' = g(x + c), h' = h(x+c)$, then $\det\left(R_{f'\,g'\,h'}\right) = \det\left(R_{f\,g\,h}\right);$
\item\label{lem:RDetProp:prop2}  if there is $x_0 \in \F$ such that $f(x_0) = g(x_0) = h(x_0)$, then $\det\left(R_{f\,g\,h}\right) = 0$.
 \end{enumerate}
\end{lemma}
\begin{proof}The property~\ref{lem:RDetProp:prop1} of $R_{f\,g\,h}$ follows from the multilinearity of the determinant along the columns of the matrix $R_{f\,g\,h}$.

Now assume that $f(x_0) = g(x_0) = h(x_0)$ for some $x_0 \in \F$. This implies that the columns of the matrix $R_{f\, g\, h}$ are linearly dependent. Hence, $\det\left(R_{f\,g\,h}\right) = 0$. This establishes~\ref{lem:RDetProp:prop2}.
\end{proof}

\begin{lemma}\label{lem:GeneralR1Crit}Assume that $n \ge 3$. Let $A \in \M_n(\F)$ and $A^{(x)} \in \M_{n}(\F[x])$ be defined by $A^{(x)} = A - xI_n$. Then $A \in \R_n(\F)$ if and only if
\begin{equation}\label{lem:GeneralR1Crit:eqq1}
 \Res_{1\,2}\left(\det\left(A^{(x)}[i_1,k|j_1,k]\right), \det\left(A^{(x)}[i'_1,i'_2|j'_1,j'_2]\right)\right) = 0
\end{equation}
 for all $1 \le i_1, j_1, k \le n,$ $1 \le i'_1 < i'_2 \le n,$ $1 \le j'_1 < j'_2 \le n$ such that $k \neq i_1, j_1$, $i_1 \neq j_1$,
\begin{equation}\label{lem:GeneralR1Crit:eqq12}
 \Res_{2\,2}\left(\det\left(A^{(x)}[k_1,k_2|k_1,k_2]\right), \det\left(A^{(x)}[k_3,k_4|k_3,k_4]\right)\right) = 0
\end{equation}
 for all pairwise distinct $1 \le k_1, k_2, k_3, k_4 \le n$
and
 \begin{equation}\label{lem:GeneralR1Crit:eqq2}
 \begin{vmatrix}
 [2]\det \left(A^{(x)}[i,j|i,j]\right) & [1]\det \left(A^{(x)}[i,j|i,j]\right) & [0]\det \left(A^{(x)}[i,j|i,j]\right)\\
  [2]\det \left(A^{(x)}[i,j|i,j]\right) & [1]\det \left(A^{(x)}[i,k|i,k]\right) & [0]\det \left(A^{(x)}[i,k|i,k]\right)\\
  [2]\det \left(A^{(x)}[i,j|i,j]\right) & [1]\det \left(A^{(x)}[j,k|j,k]\right) & [0]\det \left(A^{(x)}[j,k|j,k]\right)
 \end{vmatrix} = 0
\end{equation}
for all $1 \le i < j < k \le n$.
\end{lemma}
\begin{proof}First note that Lemma~\ref{lem:RDetProp} implies that the resultants in~\eqref{lem:GeneralR1Crit:eqq1}\linebreak and~\eqref{lem:GeneralR1Crit:eqq12} are defined correctly.

Let us prove the necessity. Assume that $A \in \R_{n}$. This implies that there is $c \in \F$ such that 
\begin{equation}\label{lem:GeneralR1Crit:eq1}
\rk(A') \le 1,
\end{equation}
 where $A' = A - cI_n$. The condition~\eqref{lem:GeneralR1Crit:eq1} implies that every $2\times 2$ minor of $A'$ is zero and consequently all polynomials of the form $\det\left(\left(A'\right)^{(x)}[i_1,i_2|j_1,j_2]\right)$ for $1 \le i_1 < i_2 \le n,$ $1 \le  j_1 < j_2 \le n$  have the common root $x = 0$. Therefore, $A'$ satisfies the conditions~\eqref{lem:GeneralR1Crit:eqq1}, \eqref{lem:GeneralR1Crit:eqq12} by Corollary~\ref{cor:ResCommonRoot} and the condition~\eqref{lem:GeneralR1Crit:eqq2} by Lemma~\ref{lem:RDetProp}\ref{lem:RDetProp:prop2}.

The conditions~\eqref{lem:GeneralR1Crit:eqq1} and~\eqref{lem:GeneralR1Crit:eqq12} are preserved by addition of a scalar matrix by the properties~\ref{enum:ResPropertyC} and~\ref{enum:DiscPropertyC}, while the condition~\eqref{lem:GeneralR1Crit:eqq2} is preserved by this operation by Lemma~\ref{lem:RDetProp}\ref{lem:RDetProp:prop1}. Therefore, $A = A' + cI_n$ satisfies these conditions as well.

Let us prove the sufficiency of the conditions~\eqref{lem:GeneralR1Crit:eqq1}  and~\eqref{lem:GeneralR1Crit:eqq2} by contradiction. Suppose that $A \in \M_{n}(\F)$ satisfies these conditions and 
\begin{equation}\label{lem:GeneralR1Crit:eq2}
A \not\in \R_n(\F).
\end{equation}
 Let us consider the following two cases separately
\begin{enumerate}[label=Case (\alph*), ref=(\alph*), leftmargin = *]
\item\label{lem:GeneralR1Crit:case1} $A$ is non-diagonal
\item\label{lem:GeneralR1Crit:case2} $A$ is diagonal
\end{enumerate}

\paragraph{Case~\ref{lem:GeneralR1Crit:case1}}  Let $1 \le i_0 < j_0 \le n$ be such that 
\begin{equation}\label{lem:GeneralR1Crit:eq10002}
a_{i_0\,j_0} \neq 0.
\end{equation}
Since $n \ge 3,$ then there is $1 \le k_0 \le n$ such that $k_0 \neq i_0, j_0$. Let $P(x) \in \F[x]$ be defined by
\begin{equation}\label{lem:GeneralR1Crit:eq2222}
P = \det(A^{(x)}[i_0, k_0| j_0, k_0]) =   
a_{i_0\,k_0}a_{k_0\,j_0} - a_{i_0\,j_0}(a_{k_0\,k_0} - x)
\end{equation}
Since $a_{i_0\,j_0} \neq 0,$ there is $c_0 \in \F$ such that $P(c_0) = 0$. From~\eqref{lem:GeneralR1Crit:eq2222} we obtain that 
\begin{equation}\label{lem:GeneralR1Crit:eq10003}
c_0 = a_{k_0\,k_0} - \frac{a_{i_0\,k_0}a_{k_0\,j_0}}{a_{i_0\,j_0}}.
\end{equation}
From the assumption that $A \not\in R_{n}(\F)$ we conclude that $\rk(A - c_0I_n) \ge 2$ and consequently there are $1 \le i_1 < i_2 \le n$ and $1 \le j_1 < j_2 \le 2$ such that 

\begin{equation}\label{lem:GeneralR1Crit:eq10004}
\det((A - c_0I_n)[i_1, i_2| j_1, j_2]) \neq 0.
\end{equation}
Therefore,
\begin{multline*}
\Res_{1\,2}\left(\det\left(A^{(x)}[i_0,k_0|j_0,k_0]\right), \det\left(A^{(x)}[i_1,i_2|j_1,j_2]\right)\right) = \\
\overset{\eqref{lem:GeneralR1Crit:eq2222}}{=\joinrel=} \Res_{1\,2}\left(a_{i_0\,k_0}a_{k_0\,j_0} - a_{i_0\,j_0}(a_{k_0\,k_0} - x), \det\left(A^{(x)}[i_1,i_2|j_1,j_2]\right)\right)\phantom{X}\\
\phantom{XX}\overset{\eqref{lem:ResMultOneByScal:eqq1}}{=\joinrel=} a_{i_0\,j_0}^2 \Res_{1\,2}\left(x + \frac{a_{i_0\,k_0}a_{k_0\,j_0}}{a_{i_0\,j_0}} - a_{k_0\,k_0}, \det\left(A^{(x)}[i_1,i_2|j_1,j_2]\right)\right)\\
\overset{\eqref{lem:GeneralR1Crit:eq10003}}{=\joinrel=} a_{i_0\,j_0}^2 \Res_{1\,2}\left(x - c_0, \det\left(A^{(x)}[i_1,i_2|j_1,j_2]\right)\right)\\
\overset{\eqref{cor:ResXX0G:eqq1}}{=\joinrel=}  a_{i_0\,j_0}^2  \det((A - c_0I_n)[i_1, i_2| j_1, j_2]) \overset{\eqref{lem:GeneralR1Crit:eq10002},\, \eqref{lem:GeneralR1Crit:eq10004}}{=\joinrel=\joinrel=\joinrel\neq\joinrel=\joinrel=\joinrel=\joinrel=} 0,
\end{multline*} 
 which contradicts~\eqref{lem:GeneralR1Crit:eqq1}.

\paragraph{Case~\ref{lem:GeneralR1Crit:case2}} Since $A$ is diagonal, then the condition~\eqref{lem:GeneralR1Crit:eq2} and Lemma~\ref{lem:DiagonalR1Crit} imply that $A$ violates at least one of the conditions~\eqref{lem:DiagonalR1Crit:eqq1} or~\eqref{lem:DiagonalR1Crit:eqq2}. Let us consider these subcases separately.

\subparagraph{$A$ does not satisfy the condition~\eqref{lem:DiagonalR1Crit:eqq1}} If $1 \le i  < j < k \le n$ are such that the condition~\eqref{lem:DiagonalR1Crit:eqq1} does not hold for $A$ and $i_1 = i, i_2 = j, i_3 = k$, then $a_{i\,i}, a_{j\,j}, a_{k\,k}$ are pairwise distinct and consequently
\begin{multline*}
 \begin{vmatrix}
 [2]\det \left(A^{(x)}[i,j|i,j]\right) & [1]\det \left(A^{(x)}[i,j|i,j]\right) & [0]\det \left(A^{(x)}[i,j|i,j]\right)\\
  [2]\det \left(A^{(x)}[i,k|i,k]\right) & [1]\det \left(A^{(x)}[i,k|i,k]\right) & [0]\det \left(A^{(x)}[i,k|i,k]\right)\\
  [2]\det \left(A^{(x)}[j,k|j,k]\right) & [1]\det \left(A^{(x)}[j,k|j,k]\right) & [0]\det \left(A^{(x)}[j,k|j,k]\right)
 \end{vmatrix}\\
=  \begin{vmatrix}
 1 & -a_{i\,i} - a_{j\,j} & a_ia_j\\
  1 & -a_{i\,i} - a_{k\,k} & a_ia_k\\
  1 & -a_{j\,j} - a_{k\,k} & a_ja_k
 \end{vmatrix} = (a_{i\,i} - a_{j\,j})(a_{i\,i} - a_{k\,k})(a_{j\,j}-a_{k\,k}) \neq 0,
\end{multline*}
 which contradicts~\eqref{lem:GeneralR1Crit:eqq2}.
 
 \subparagraph{$A$ does not satisfy the condition~\eqref{lem:DiagonalR1Crit:eqq2}} If $1 \le i_1 < i_2 < i_3 < i_4 \le n$ are such that the condition~\eqref{lem:DiagonalR1Crit:eqq2} does not hold for $A$, then assume without loss of generality that 
\begin{equation}\label{lem:GeneralR1Crit:eq22203}
a_{i_1\, i_1} \neq a_{i_3\,i_3}, a_{i_3\,i_3}, a_{i_4\,i_4}\quad\mbox{and}\quad a_{i_2\, i_2} \neq a_{i_3\,i_3}, a_{i_4\,i_4}.
\end{equation}
Since $A$ is diagonal, then
\begin{equation}\label{lem:GeneralR1Crit:eq22201}
 \det\left(A^{(x)}[i_1,i_2|i_1,i_2]\right) = (x - a_{i_1\,i_1})(x - a_{i_2\,i_2}),
\end{equation}
\begin{equation}\label{lem:GeneralR1Crit:eq22202}
 \det\left(A^{(x)}[i_3,i_4|i_3,i_4]\right) = (x - a_{i_3\,i_3})(x - a_{i_4\,i_4}).
\end{equation}
Hence,
\begin{multline*}
\Res_{2\,2}\left(\det\left(A^{(x)}[i_1,i_2|i_1,i_2]\right), \det\left(A^{(x)}[i_3,i_4|i_3,i_4]\right)\right)\\
\overset{\eqref{lem:GeneralR1Crit:eq22201},\, \eqref{lem:GeneralR1Crit:eq22202}}{=\joinrel=\joinrel=\joinrel=\joinrel=\joinrel=\joinrel=\joinrel=} \Res_{2\,2}\left((x - a_{i_1\,i_1})(x - a_{i_2\,i_2}), (x - a_{i_3\,i_3})(x - a_{i_4\,i_4})\right)\\
\overset{\ref{enum:ResTwoSplit}}{=\joinrel=} \prod_{\substack{1 \le \alpha \le 2\\ 3 \le \beta \le 4}} (x_{i_\alpha\, i_\alpha} - x_{i_\beta\, i_\beta}) \neq 0,
\end{multline*}
 which contradicts~\eqref{lem:GeneralR1Crit:eqq12}.
\end{proof}

\begin{lemma}\label{lem:R1IndicatorPolyDeg}If $R \in \F[a_{1\,1}, \ldots, a_{n\,n}]$ is a polynomial appearing on the left-hand side of~\eqref{lem:GeneralR1Crit:eqq1}, \eqref{lem:GeneralR1Crit:eqq12} or~\eqref{lem:GeneralR1Crit:eqq2}, then $\deg (R) \le 4$.
\end{lemma}
\begin{proof}First, assume that $R \in \F[a_{1\,1}, \ldots, a_{n\,n}]$ is a polynomial appearing on the left-hand side of~\eqref{lem:GeneralR1Crit:eqq1} or \eqref{lem:GeneralR1Crit:eqq12}. Hence, there are $1 \le i_1 \neq i_2 \le n$, $1 \le j_1 \neq j_2 \le n$,  $1 \le i'_1 \neq i'_2 \le n,$ $1 \le j'_1 \neq j'_2 \le n$ such that
\begin{equation}\notag
R = \Res_{d\, 2}(P, Q),
\end{equation}
where
\[
P = \det\left(A^{(x)}[i_1,i_2|j_1,j_2]\right),\quad Q = \det\left(A^{(x)}[i'_1,i'_2|j'_1,j'_2]\right),\quad d \ge \deg_x\left(P\right)
\]
Let  $\lambda$ be a formal variable. Then
\begin{multline}\label{lem:R1IndicatorPolyDeg:eq2}
R(\lambda A) = \Res_{d\,2}\left(P(\lambda A), Q(\lambda A)\right)\\
= \Res_{d\,2}\left([0]P(\lambda A), \ldots , [d]P(\lambda A); [0]Q(\lambda A),  \ldots, [2]Q(\lambda A)\right)\\
\overset{\mbox{\scriptsize Lemma~\ref{lem:DegDetAMinusXIn}\ref{lem:DegDetAMinusXIn:prop3}}}{=\joinrel=\joinrel=\joinrel=\joinrel=\joinrel=\joinrel=\joinrel=\joinrel=} \Res_{d\,2}\left(\lambda^2 [0]P(A), \ldots , \lambda^{2 - d}[d]P(A); \lambda^2 [0]Q(A),  \ldots, [2]Q(A)\right)\\
\overset{\eqref{lem:ResMultOneByScal:eqq1}}{=\joinrel=} \lambda^{2(2-d)}\Res_{d\,2}\left(\lambda^d [0]P(A), \ldots , [d]P(A); \lambda^2 [0]Q( A),  \ldots, [2]Q(A)\right)\\
\overset{\eqref{lem:ResReverseGradHomogen:eqq1}}{=\joinrel=} \lambda^{2(2 - d)}\cdot \lambda^{d \cdot 2}\Res_{d\,2}\left( [0]P(A), \ldots , [d]P(A); [0]Q(A),  \ldots, [2]Q(A)\right)\\
 = \lambda^{4} R(A).
\end{multline}
Therefore, $R$ is a homogeneous polynomial and $\deg(R) \le 4$. 

Second, assume that $R \in \F[a_{1\,1}, \ldots, a_{n\,n}]$ is a polynomial appearing on the left-hand side of~\eqref{lem:GeneralR1Crit:eqq2}. Hence, there are $1 \le i < j < k \le n$ such that
\begin{equation}\notag
R = \begin{vmatrix}[2]F & [1]F & [0]F\\ [2]G & [1]G & [0]G\\ [2]H & [1]H & [0]H\end{vmatrix},
\end{equation}
where
\[
F = \det\left(A^{(x)}[i, j|i, j]\right),\;\;G = \det\left(A^{(x)}[i, k|i, k]\right),\;\;G = \det\left(A^{(x)}[j, k|j, k]\right),
\]
Let $\lambda$ be a formal variable. Then
\begin{multline*}
R(\lambda A) = \begin{vmatrix}[2]F(\lambda A) & [1]F(\lambda A) & [0]F(\lambda A)\\ [2]G(\lambda A) & [1]G(\lambda A) & [0]G(\lambda A)\\ [2]H(\lambda A) & [1]H(\lambda A) & [0]H(\lambda A)\end{vmatrix}\\
\overset{\mbox{\scriptsize Lemma~\ref{lem:DegDetAMinusXIn}\ref{lem:DegDetAMinusXIn:prop3}}}{=\joinrel=\joinrel=\joinrel=\joinrel=\joinrel=\joinrel=\joinrel=\joinrel=}  \begin{vmatrix}[2]F(A) & \lambda [1]F(A) & \lambda^2 [0]F(A)\\ [2]G(A) & \lambda [1]G(A) & \lambda^2 [0]G(A)\\ [2]H(A) & \lambda [1]H( A) & \lambda^2 [0]H(A)\end{vmatrix}\\
= \lambda^{1 + 2}\begin{vmatrix}[2]F(A) & [1]F(A) & [0]F(A)\\ [2]G(A) & [1]G(A) &  [0]G(A)\\ [2]H(A) &  [1]H( A) & [0]H(A)\end{vmatrix} = \lambda^3 R(A).\\
\end{multline*}
Therefore, $R$ is a homogeneous polynomial and $\deg(R) \le 3$. Thus, the statement of the lemma holds in all possible cases.
\end{proof}

\begin{lemma}\label{lem:RNonzeroRAll}Assume that $|\F| > 6$ and $n  \ge 3$. Let $c_1, \ldots, c_m \in \F$ and $B^{(1)}, \ldots, B^{(m)} \in \M_{n\,k}(\F)$ be such that $(c_1, \ldots, c_m)^t \neq \mathbf 0$. 
If 
\begin{equation}\label{lem:RNonzeroRAll:eqq1}
\alpha_1 B^{(1)} + \ldots + \alpha_n B^{(m)} \in \R_n(\F)\;\;\mbox{for all}\;\; \alpha_1, \ldots, \alpha_m \in \F\;\;\mbox{such that}\;\;\sum\limits_{i=1}^{m} c_i\alpha_i\neq 0,
\end{equation}
then
\begin{equation}\label{lem:RNonzeroRAll:eqq2}
\alpha_1 B^{(1)} + \ldots + \alpha_n B^{(m)} \in \R_n(\F)\quad\mbox{for all}\;\;\alpha_1, \ldots, \alpha_m \in \F.
\end{equation}
\end{lemma}

\begin{proof}Let $\mathcal S$ be the set of all the polynomials $P \in \F[a_{1\,1}, \ldots, a_{n\,n}]$ appearing on the left-hand side  of~\eqref{lem:GeneralR1Crit:eqq1}, \eqref{lem:GeneralR1Crit:eqq12} or~\eqref{lem:GeneralR1Crit:eqq2}. Lemma~\ref{lem:GeneralR1Crit} means that 
\begin{equation}\label{lem:RNonzeroRAll:eq1}
A \in \R_n(\F) \Leftrightarrow \left(\forall P \in \mathcal S \colon P(A) = 0\right)\;\;\mbox{for all}\;\; A \in \M_n(\F)
\end{equation}
In addition,
\begin{equation}\label{lem:RNonzeroRAll:eq2}
\deg(P) \le 4\;\;\mbox{for all}\;\; P \in \mathcal S
\end{equation}
by Lemma~\ref{lem:R1IndicatorPolyDeg}.
Let $\mathcal S' \subseteq \F[x_1,\ldots, x_m]$ be defined by
\begin{equation}
\mathcal S' = \left\{P\left(A\left(x_1,\ldots, x_m\right)\right) | P \in \mathcal S\right\},
\end{equation}
where $A\left(x_1,\ldots, x_m\right) = \left(a_{i\,j}(x_1,\ldots, x_m)\right) \in \M_{n}\left(\F[x_1,\ldots, x_m]\right)$ is defined by
$$A(x_1,\ldots, x_m) = x_1 B^{(1)} + \ldots + x_m B^{(m)}.$$
From~\eqref{lem:RNonzeroRAll:eq1} and~\eqref{lem:RNonzeroRAll:eq2} we conclude that
\begin{equation}\label{lem:RNonzeroRAll:eq3}
A\left(x_1,\ldots, x_m\right) \in \R_n(\F) \Leftrightarrow \left(\forall P' \in \mathcal S' \colon P\left(A\left(x_1,\ldots, x_m\right)\right) = 0\right)
\end{equation}
for all $x_1,\ldots, x_m \in \F$ and
\begin{equation}\label{lem:RNonzeroRAll:eq4}
\deg(P) \overset{\eqref{lem:RNonzeroRAll:eq2}}{\le} 4\;\;\mbox{for all}\;\; P \in \mathcal S'
\end{equation}
Let $Q = c_1x_1 + \ldots + c_mx_m$. The assumptions of the lemma imply that $Q \neq 0$. In addition,
\begin{equation}\notag
\deg(PQ) \overset{\eqref{lem:RNonzeroRAll:eq4}}{\le} 4 + 1 = 5 < |\F|\;\;\mbox{for all}\;\; P \in \mathcal S'.
\end{equation}
Hence, 
\[
P'\left(A\left(x_1,\ldots, x_m\right)\right) = 0\;\;\mbox{for all}\;\; P' \in \mathcal S',\, x_1, \ldots, x_m \in \F
\]
by Lemma~\ref{lem:ManyPolyDensityArg} and consequently
\[
\alpha_1 B^{(1)} + \ldots + \alpha_n B^{(m)} = A(\alpha_1, \ldots, \alpha_m) \in \R_n(\F)\quad\mbox{for all}\;\;\alpha_1, \ldots, \alpha_m \in \F
\]
by~\eqref{lem:RNonzeroRAll:eq3}.

\end{proof}

\begin{remark}
Lemma~\ref{lem:RNonzeroRAll} could be considered as an analogue Lemma~2.4 from~\cite{Guterman2025c}, where the set of rank-one rectangular matrices is considered instead of $\R_n(\F).$
\end{remark}

\begin{lemma}[Cf.~{\cite[the excersices for Section~3.2]{Friedberg2002-wu}}]\label{lem:RkSum}If $A, B \in \M_{n}(\F)$, then $\rk(A + B) \le \rk(A) + \rk(B)$.
\end{lemma}

\begin{lemma}\label{lem:SumRankOneEqualScal}Let $n \ge 2$ be an integer. If $A^{(1)}, \ldots, A^{(n-1)} \in \M_{n}(\F)$ are such that  $\rk\left(A^{(i)}\right) \le 1$ for all $1 \le i \le (n-1)$ and $\sum_{i=1}^{n-1} A^{(i)}$ is a scalar matrix, then $\sum_{i=1}^{n-1} A^{(i)} = 0$.
\end{lemma}
\begin{proof}
Let $A = \sum_{i=1}^{n-1} A^{(i)}$. Lemma~\ref{lem:RkSum} and the assumptions of the lemma imply that 
$$\rk\left(A\right) = \rk\left(\sum_{i=1}^{n-1} A^{(i)}\right) \le  \sum_{i=1}^{n-1}\rk\left( A^{(i)}\right) \le (n-1).$$
Therefore, $A = 0$ because otherwise the assumption that $A$ is a scalar matrix implies that $\rk(A) = n$. This concludes the statement of the lemma.
\end{proof}

\begin{corollary}\label{cor:OneWayRankOnePlusScalar}Assume that $n \ge 3$. If $A, B \in \M_{n}(\F)$ and $a, b \in \F$ are such that  $\rk(A), \rk(B) \le 1$ and $A + cI_{n} = B + dI_n,$
then $A = B$ and $c = d.$
\end{corollary}

\section{Orbit invariants of linear maps converting one polynomial to another}
\label{sec:OrbInvariants}

We begin this section by introducing the necessary notation from~\cite{Lang2002} and~\cite{YurkovPolyPres2026} which is used to provide the definitions of necessary matrix invariants.

\begin{definition}[{Cf.~\cite[IV,\S 1, the definition on p.~178]{Lang2002}}]\label{def:onedimder}If $f(X) =
a_n x^n + \ldots + a_0 \in \F[x]$, then the \emph{derivative} of $f$ is a polynomial belonging to $\F[x]$ which is denoted by $f'$ and is defined by
\[
f'(X)= na_n x^{n-1} + \ldots + a_1.
\]
\end{definition}

\begin{definition}[{\cite[Definition~3.13]{YurkovPolyPres2026}}]\label{def:partDer}
Let $P \in \F[x_1,\ldots, x_n]$. Then $\frac{\partial P}{\partial \mathbf v}(\mathbf a) \in \F[x_1,\ldots, x_n; y_1, \ldots, y_n]$ denotes the \emph{directional derivative of $P$ along $\mathbf v$ at $\mathbf a$} and is defined as a coefficient at $t$ in the following equality in\linebreak $\F[x_1,\ldots, x_n; y_1, \ldots, y_n; t]$
\[
P(\mathbf a + t \mathbf v) = P(\mathbf a)  + P_1(\mathbf a, \mathbf v)t + \mbox{(terms of higher degree in $t$)},
\]
where the coordinates of $\mathbf a$ correspond to $x_1,\ldots, x_n$, and  the coordinates of $\mathbf v$ correspond to $y_1,\ldots, y_n$. 
\end{definition}

\begin{lemma}[{\cite[Lemma~3.14]{YurkovPolyPres2026}}]\label{lem:DirDerLin}
If $P \in \F[x_1,\ldots, x_n]$ and $\mathbf a \in \F^n$ are fixed, then the function $\mathbf v \mapsto  \frac{\partial P}{\partial \mathbf v}(\mathbf a)$ is linear.
\end{lemma}

\begin{definition}[{\cite[Definition~3.15]{YurkovPolyPres2026}}] \label{def:PAV}
Let $P \in \F[x_1,\ldots, x_n]$ and $\mathbf a, \mathbf v \in \F^n$ are fixed. Then by $P_{\mathbf a, \mathbf v} \in \F[t]$ we denote a polynomial defined by
\begin{equation}\label{eq:def:PAV}
P_{\mathbf a, \mathbf v} = P(\mathbf a + t \mathbf v).
\end{equation}

\end{definition}

\begin{lemma}[{\cite[Lemma~3.16]{YurkovPolyPres2026}}] \label{lem:partDer}Let $P \in \F[x_1,\ldots, x_n]$ and $\mathbf a, \mathbf v \in \F^n$. Then $\deg(P_{\mathbf a, \mathbf v}) \le \deg(P)$ and 
\begin{equation}\label{def:partDer:eqq1}
\frac{\partial P}{\partial \mathbf v}(\mathbf a) = P'_{\mathbf a, \mathbf v}(0)\;\;\mbox{for all}\;\; \mathbf v \in \F^n.
\end{equation}
\end{lemma}

\begin{definition}[{\cite[Definition~4.1]{YurkovPolyPres2026}}]\label{def:PLFunc} Let $P \in \F[x_1,\ldots, x_n]$ and $\mathbf a \in \F^n$. By $l_{P,\mathbf a} \colon \mathbb F^n \to \mathbb F$ we denote a linear function defined by
\[
l_{P, \mathbf a}(\mathbf v) = \frac{\partial P}{\partial \mathbf v}(\mathbf a)\qquad(\mathbf v \in \F^n).
\] 
\end{definition}

\begin{definition}[{\cite[Definition~4.2]{YurkovPolyPres2026}}]\label{def:LP}
Let $P \in \F[x_1,\ldots, x_n]$. By $\mathcal L_P \subseteq (\F^n)^*$ we denote the vector space spanned by all the linear functions of the form $l_{P, \mathbf a}$ for some $\mathbf a \in \F^n.$ That is,
\[
\mathcal L_P = \Span (\{\mathbf v \mapsto \frac{\partial P}{\partial \mathbf v}(\mathbf a) \mid \mathbf a \in \mathbb F^n\}).
\]
Equivalently, $\mathcal L_{P} \subseteq {\F^n}^*$ spanned by the range of the gradient field of $P \in \F[x_1,\ldots, x_n]$.
\end{definition}

\begin{lemma}[{\cite[Lemma~5.9]{YurkovPolyPres2026}}]\label{thm:surjpres}Let $P \in \F[x_1,\ldots, x_n]$ be a polynomial such that $\deg (P) < |\F|$. If  $\phi, \psi \colon \F^n \to \F^n$ are such that
\begin{equation*}
P(\mathbf{x} + \lambda \mathbf{y}) = P(\phi(\mathbf{x}) + \lambda \psi(\mathbf{y}))\;\;\mbox{for all}\;\;\mathbf{x},\mathbf{y} \in \mathbb \F^n,\; \lambda \in \F,
\end{equation*}
and $\phi$ is surjective, then $l \circ \psi \in \mathcal L_P$ for all $l \in \mathcal L_P.$
\end{lemma}

\begin{corollary}\label{cor:SurjPresComp}Let $P \in \F[x_1,\ldots, x_n]$ be a polynomial such that $\deg (P) < |\F|$. If  $\phi, \psi \colon \F^n \to \F^n$ are such that
\begin{equation*}
P(\mathbf{x} + \lambda \mathbf{y}) = P(\phi(\mathbf{x}) + \lambda \psi(\mathbf{y}))\;\;\mbox{for all}\;\;\mathbf{x},\mathbf{y} \in \mathbb \F^n,\; \lambda \in \F,
\end{equation*}
and $\phi$ is surjective, then $\psi(\mathbf v) \in \left(\FL_P\right)^0$ for all $\mbf v \in \left(\FL_P\right)^0$. 
\end{corollary}

In order to compare $\mathcal L_{P}$ and $\mathcal L_{Q}$ for distinct $P, Q \in \F[x_1,\ldots, x_n]$ we need the following extension of Lemma~4.10 from~\cite{YurkovPolyPres2026}.

\begin{lemma}\label{lem:eqpqd}Let $P \in \F[x_1,\ldots, x_n], Q \in \F[x_1, \ldots, x_m]$, $\mathbf x_1, \mathbf y_1\in \F^n$ and $\mathbf x_2, \mathbf y_2 \in \F^m$. If $\deg(P), \deg(Q) < |\F|$ and
\begin{equation}\label{lem:eqpqd:eqlem1}
P(\mathbf x_1 + \lambda \mathbf y_1) = Q(\mathbf x_2 + \lambda \mathbf y_2)\;\;\mbox{for all}\;\;\lambda \in \F, 
\end{equation}
then
\begin{equation*}\label{lem:eqpqd:eqlem2}
l_{P,\mathbf x_1}(\mathbf y_1) = l_{Q, \mathbf x_2}(\mathbf y_2).
\end{equation*}
\end{lemma}
\begin{proof}Indeed, the equality~\eqref{lem:eqpqd:eqlem1} implies that 
\begin{equation}\label{lem:eqpqd:eq1}
P_{\mathbf x_1, \mathbf y_1}(\lambda) = Q_{\mathbf x_2, \mathbf y_2}(\lambda)\;\;\mbox{for all}\;\;\lambda \in \F,
\end{equation}
where $P_{\mathbf x_1, \mathbf y_1}, Q_{\mathbf x_2, \mathbf y_2} \in \F[t]$ are defined in Definition~\ref{def:PAV}. Since $\deg(P_{\mathbf x_1, \mathbf y_1}) \leq \deg(P)$ and $\deg(Q_{\mathbf x_2, \mathbf y_2}) \leq \deg(Q)$ by Lemma~\ref{def:partDer}, and $\deg(P) < |\F|$ by the statement of the lemma, then $\deg(P_{\mathbf x_1, \mathbf y_1}) < |\F|$ and $\deg(Q_{\mathbf x_2, \mathbf y_2}) < |\F|$. Therefore, from~\eqref{lem:eqpqd:eq1} and Lemma~\ref{DegFGLessEqual} we conclude that
\begin{equation}\label{lem:eqpd:eq2}
P_{\mathbf x_1, \mathbf y_1} = Q_{\mathbf x_2, \mathbf y_2}\;\;\mbox{as elements of}\;\F[t].
\end{equation}
Hence, $P'_{\mathbf x_1, \mathbf y_1} = Q'_{\mathbf x_2, \mathbf y_2}$ and, in particular,
$P'_{\mathbf x_1, \mathbf y_1}(0) = Q'_{\mathbf x_2, \mathbf y_2}(0)$

Therefore,
\[
l_{P,\mathbf x_1}(\mathbf y_1) = P'_{\mathbf x_1, \mathbf y_1}(0) = Q'_{\mathbf x_2, \mathbf y_2}(0) =  l_{Q, \mathbf x_2}(\mathbf y_2)
\]
by Lemma~\ref{def:partDer} and the definitions of $l_{P,\mathbf x_1}$ and $l_{Q,\mathbf x_2}$.
\end{proof}

\begin{lemma}\label{lem:DimLPUnderLinConvInj}Let $P \in \F[x_1,\ldots, x_n], Q \in \F[x_1, \ldots, x_m]$ be such that\linebreak  $\deg(P), \deg(Q) < |\F|$. If there is a injective linear map $T \colon \F^n \to \F^m$ such that 
\begin{equation}\label{lem:DimLPUnderLinConv:eqq1}
P(\mathbf x) = Q(T(\mathbf x))\;\;\mbox{for all}\;\;\mathbf x \in \F^{n},
\end{equation}
 then $\dim(\mathcal L_P) \le \dim(\mathcal L_Q).$
\end{lemma}
\begin{proof}By contradiction. Suppose that 
\begin{equation}\label{lem:DimLPUnderLinConv:eq2}
d = \dim(\mathcal L_P) > \dim(\mathcal L_Q).
\end{equation} Let $\mathbf x_1, \ldots, \mathbf x_d \in \F^n$ be such that the functionals $l_{P, \mathbf x_1}, \ldots l_{P, \mathbf x_d}$ form a basis of $\mathcal L_P$. This implies that 
\begin{equation}\label{lem:DimLPUnderLinConv:eq1}
l_{P, \mathbf x_1}, \ldots l_{P, \mathbf x_d}\;\;\mbox{are linearly independent}.
\end{equation}
From the inequality~\eqref{lem:DimLPUnderLinConv:eq2} we obtain that there is $c_1, \ldots, c_d \in \F$ such that $(c_1,\ldots, c_d)^t \neq \mathbf 0$ and
\begin{equation}\label{lem:DimLPUnderLinConv:eq4}
\sum_{i=1}^d c_i  l_{Q, T(\mathbf x_i)} = 0 
\end{equation}

The equality~\eqref{lem:DimLPUnderLinConv:eqq1} implies that
\begin{equation*}
P(\mathbf x_i + \lambda \mathbf y) = Q(T(\mathbf x_i) + \lambda T(\mathbf y))\;\;\mbox{for all}\;\; 1 \le i \le d, \; \mathbf y \in \F^n, \;\lambda \in \F.
\end{equation*}
Hence,
\begin{equation}\label{lem:DimLPUnderLinConv:eq3}
l_{P,\mathbf x_i}(\mathbf y) = l_{Q, T(\mathbf x_i)}(T(\mathbf  y))\;\;\mbox{for all}\;\; 1 \le i \le d, \; \mathbf y \in \F^n
\end{equation}
by Lemma~\ref{lem:eqpqd}. Therefore,
\begin{equation*}
\sum_{i=1}^d c_i  l_{Q, T(\mathbf x_i)}(T(\mathbf y)) \overset{\eqref{lem:DimLPUnderLinConv:eq3}}{=\joinrel=\joinrel=} \sum_{i=1}^d c_i  l_{P, \mathbf x_i}(\mathbf y) \overset{\eqref{lem:DimLPUnderLinConv:eq4}}{=\joinrel=\joinrel=} 0 \;\;\mbox{for all}\;\;  \mathbf y \in \F^n
\end{equation*}
which contradicts with~\eqref{lem:DimLPUnderLinConv:eq1}.
\end{proof}

\begin{corollary}\label{lem:DimLPUnderLinConvBij}
Let $P, Q \in \F[x_1,\ldots, x_n]$ be such that $\deg(P), \deg(Q) < |\F|$. If there is a bijective linear map $T \colon \F^n \to \F^n$ such that $P(\mathbf x) = Q(T(\mathbf x))$ for all $\mathbf x \in \F^{n}$, then $\dim(\mathcal L_P) = \dim(\mathcal L_Q).$
\end{corollary}

For introducing more invariants, we need the following variation of the notions of the partial derivative and the polynomial $P_{\mathbf a, \mathbf v}$  from Definition~\ref{def:partDer} and Definition~\ref{def:PAV}, correspondingly.

\begin{definition}\label{def:PolyPointInv}
Let $P \in \F[x_1,\ldots, x_n]$ and $\mathbf x \in \F^n$ is fixed. Then by $P_{(\mathbf x)} \in \F[a_1,\ldots, a_n, \lambda]$ we denote a polynomial defined by
\[
P_{(\mathbf x)}(\mathbf a, \lambda) = P(\mathbf a + \lambda \mathbf x)
\] 
and by $P_{(\mathbf x)}^{(0)}, \ldots, P_{(\mathbf x)}^{\left(\deg_\lambda\left(P^{(\mathbf x)}\right)\right)} \in \F[a_1,\ldots, a_n]$ we denote coefficients of $\lambda$ in the expansion
\[
P_{(\mathbf x)}(\mathbf a, \lambda) = P_{(\mathbf x)}^{(0)}(\mathbf a) + \lambda P_{(\mathbf x)}^{(1)}(\mathbf a) + \ldots + \lambda^{\deg_\lambda \left(P_{(\mathbf x)}\right)} P_{(\mathbf x)}^{\left(\deg_\lambda\left(P^{(\mathbf x)}\right)\right)}(\mathbf a)
\]
of $P^{(\mathbf x)}$ by the powers of $\lambda$ (in the expansion above $P_{(\mathbf x)}$ is considered as an element of $\F[a_1,\ldots, a_n, \lambda]$).
\end{definition}

\begin{lemma}\label{lem:DegsPPart}Let $P \in \F[x_1,\ldots, x_n]$ and $\mathbf x \in \F^n$ is fixed. Then 
\[\deg \left(P_{(\mathbf x)}^{(i)}\right) \le \deg_\lambda  \left(P_{(\mathbf x)}\right) \le \deg (P)\;\;\mbox{for all}\;\;1 \le i \le \deg_\lambda  \left(P_{(\mathbf x)}\right).
\]
\end{lemma}

\begin{proof}
Indeed,
\begin{multline*}
\deg \left(P_{(\mathbf x)}^{(i)}\right) \overset{\mbox{\scriptsize By the definition of $P_{(\mathbf x)}^{(i)}$}}{=\joinrel=\joinrel=\joinrel=\joinrel=\joinrel=\joinrel=\joinrel=\joinrel=\joinrel=\joinrel=\joinrel=\joinrel=\joinrel=} \deg \left(P_{(\mathbf x)}\right) - i\\
 \le \deg_\lambda  \left(P_{(\mathbf x)}\right)  \le \deg (P^{(\mathbf x)}) \overset{\mbox{\scriptsize By the definition of $P_{(\mathbf x)}$}}{=\joinrel=\joinrel=\joinrel=\joinrel=\joinrel=\joinrel=\joinrel=\joinrel=\joinrel=\joinrel=\joinrel=\joinrel=\joinrel=} \deg (P)
\end{multline*}
for all $1 \le i \le \deg_\lambda  \left(P_{(\mathbf x)}\right).$
\end{proof}

\begin{lemma}\label{lem:IfPolyEqThenCompEq}
Let $P, P' \in \F[x_1,\ldots, x_n],$  $\mathbf x, \mathbf x' \in \F^n$ be such that\linebreak $\deg(P), \deg(P') < |\F|$ and $P_{(\mathbf x)}(\mathbf a, \lambda) = {P'}_{(\mathbf x')}(\mathbf a, \lambda)$ for all $\mathbf a \in \F^n, \lambda \in \F$. Then 
\begin{equation}\label{lem:IfPolyEqThenCompEq:eqq11}
\deg_\lambda\left(P_{(\mathbf x)}\right) = \deg_\lambda\left(P'_{(\mathbf x')}\right)
\end{equation}
and 
\begin{equation}\label{lem:IfPolyEqThenCompEq:eqq1}
P_{(\mathbf x)}^{(i)}=  {P'}_{(\mathbf x')}^{(i)}\;\;\mbox{for all}\;\; 1 \le i \le \deg_\lambda\left(P_{(\mathbf x)}\right). 
\end{equation}
\end{lemma}
\begin{proof}From Lemma~\ref{lem:DegsPPart} and the assumptions of the lemma we conclude that
\begin{equation*}
\deg_\lambda \left(P_{(\mathbf x)}\right) < |\F|\;\;\mbox{and}\;\; \deg_{\lambda}\left(P'_{(\mathbf x')}\right) < |\F|.
\end{equation*}
This allows us to apply the Corollary~\ref{DegFGLessEqual} to $P_{(\mathbf x)}$ and $P'_{(\mathbf x')}$ and infer that $P_{(\mathbf x)}$ and $P'_{(\mathbf x')}$ are equal as elements of $\F[a_1,\ldots, a_n, \lambda]$. Therefore, their corresponding coefficients are equal to each other. Hence, the equalities~\eqref{lem:IfPolyEqThenCompEq:eqq11} and~\eqref{lem:IfPolyEqThenCompEq:eqq1} hold.
\end{proof}

\begin{lemma}\label{lem:DegAndLPAreReallyInvariantForRad}Let $P \in \F[x_1,\ldots, x_n]$ be such that $\deg(P)< |\F|$, and $\mathbf y, \mathbf y'$ be such that $\mathbf y = \mathbf y' \pmod{\rad(P)}$. Then
\begin{equation}\label{lem:DegAndLPAreReallyInvariantForRad:eqq1111}
\deg_\lambda\left(P_{(\mathbf y)}\right) = \deg_\lambda\left(P_{(\mathbf y')}\right)
\end{equation}
and
\begin{equation}\label{lem:DegAndLPAreReallyInvariantForRad:eqq1112}
\dim\left(\mathcal L_{P_{(\mathbf y)}^{(i)}}\right) = \dim\left(\mathcal L_{P_{(\mathbf y')}^{(i)}}\right)\;\;\mbox{for all}\;\; 1 \le i \le \deg_\lambda\left(P_{(\mathbf y)}\right).
\end{equation}
\end{lemma}
\begin{proof}Let $\mathbf z = \mathbf y - \mathbf y'$ and $P' \in \F[x_1, \ldots, x_n]$ be defined by $P(\mathbf x) = P(\mathbf x + \mathbf z)$. This definition implies that $\deg(P') = \deg(P)$. On the one hand,
\begin{equation*}
P' = P
\end{equation*}
because $\mathbf z \in \rad(P)$ and $\deg(P)< |\F|$. Hence,
\begin{equation}\label{lem:DegAndLPAreReallyInvariantForRad:eq1}
P_{(\mathbf y')} = {P'}_{(\mathbf y')}\;\;\mbox{and}\;\;P_{(\mathbf y')}^{(i)} = {P'}_{(\mathbf y')}^{(i)}\qquad\mbox{for all}\;\; 1 \le i \le \deg_\lambda\left(P_{(\mathbf y)}\right).
\end{equation}
On the other hand,
\begin{multline*}
{P'}_{(\mathbf y')}(\mathbf a, \lambda) = {P'}(\mathbf a + \lambda \mathbf y') = P(\mathbf a + \mathbf z + \lambda \mathbf y')\\
= P(\mathbf a + \mathbf z + \lambda (\mathbf y - \mathbf z)) = P(\mathbf a + \lambda \mathbf y + (1 - \lambda) \mathbf z)\\
\overset{\scriptsize \mathbf z \in \rad(P)}{=\joinrel=\joinrel=\joinrel=\joinrel=\joinrel=} P(\mathbf a + \lambda \mathbf y) = {P}_{(\mathbf y)}(\mathbf a, \lambda).
\end{multline*}
By applying Lemma~\ref{lem:IfPolyEqThenCompEq} to $P, \mathbf y$ and $P', \mathbf y'$, correspondingly, we obtain that
 \begin{equation}\label{lem:DegAndLPAreReallyInvariantForRad:eqq11}
\deg_\lambda\left(P_{(\mathbf y)}\right) = \deg_\lambda\left(P'_{(\mathbf y')}\right)
\end{equation}
and 
\begin{equation}\label{lem:DegAndLPAreReallyInvariantForRad:eqq1}
P_{(\mathbf y)}^{(i)}=  {P'}_{(\mathbf y')}^{(i)}\;\;\mbox{for all}\;\; 1 \le i \le \deg_\lambda\left(P_{(\mathbf y)}\right). 
\end{equation}
Finally, by combining~\eqref{lem:DegAndLPAreReallyInvariantForRad:eqq11} and~\eqref{lem:DegAndLPAreReallyInvariantForRad:eqq1} with~\eqref{lem:DegAndLPAreReallyInvariantForRad:eq1} we obtain~\eqref{lem:DegAndLPAreReallyInvariantForRad:eqq1111} and~\eqref{lem:DegAndLPAreReallyInvariantForRad:eqq1112}, respectively. 
\end{proof}

\begin{lemma}\label{lem:TPreservesPThenTPreservesPI}Let $P \in \F[x_1,\ldots, x_n],$ $Q \in \F[x_1,\ldots, x_m],$ $\mathbf x \in \F^n,$ $\mathbf y \in \F^m,$ $T \colon \F^n \to \F^m$ be such that $\deg(P), \deg(Q) < |\F|$, $T$ is a linear map sending $\mathbf x$ to $\mathbf y$ and satisfying 
\begin{equation}\label{lem:TPreservesPThenTPreservesPI:eqq1}
P(\mathbf z) = Q(T(\mathbf z))\;\;\mbox{for all}\;\;\mathbf z \in \F^n.
\end{equation}
Then
\begin{equation}\label{lem:TPreservesPThenTPreservesPI:eqq2}
P_{(\mathbf x)}^{(i)}(\mathbf z) =  Q_{(\mathbf y)}^{(i)}(T(\mathbf z))\;\;\mbox{for all}\;\;\mathbf z \in \F^n,\; 1 \le i \le \deg_\lambda\left(P_{(\mathbf x)}\right). 
\end{equation}
\end{lemma}
\begin{proof}
Since the map $T$ is linear, then the composition $Q' = Q \circ T$ is a polynomial (i.e., element of $\F[x_1, \ldots, x_n]$) and $\deg (Q') \le \deg (Q)$. Together with the initial assumption that $\deg(P), \deg(Q) < |\F|$ and the equality~\eqref{lem:TPreservesPThenTPreservesPI:eqq1} this allows us to apply Lemma~\ref{lem:IfPolyEqThenCompEq} to $P$ and $P'$ and obtain that
\begin{equation}\label{TPreservesPThenTPreservesPI:eq1}
P_{(\mathbf x)}^{(i)}=  {Q'}_{(\mathbf y)}^{(i)}\;\;\mbox{for all}\;\; 1 \le i \le \deg_\lambda\left(P_{(\mathbf x)}\right). 
\end{equation}
In addition, the equality  $Q' = Q \circ T$ infers the equalities 
\begin{equation}\label{TPreservesPThenTPreservesPI:eq2}
{Q'}_{(\mathbf y)}^{(i)} = {Q}_{(\mathbf y)}^{(i)} \circ T\;\;\mbox{for all}\;\; 1 \le i \le \deg_\lambda\left(Q_{(\mathbf x)}\right).
\end{equation} 
Then, by substituting~\eqref{TPreservesPThenTPreservesPI:eq2} in~\eqref{TPreservesPThenTPreservesPI:eq1} and evaluating it in $\mathbf z \in \F^n$, we obtain~\eqref{lem:TPreservesPThenTPreservesPI:eqq2}.
\end{proof}

\begin{lemma}\label{lem:DegAndLPAreInvariant}Let $P \in \F[x_1,\ldots, x_n],$ $Q \in \F[x_1,\ldots, x_m],$ $\mathbf x \in \F^n,$ $\mathbf y \in \F^m,$ $T \colon \F^n \to \F^m$ be such that $\deg(P), \deg(Q) < |\F|$, $T$ is a linear map sending $\mathbf x$ to $\mathbf y$ and satisfying 
\begin{equation}\label{lem:DegAndLPAreInvariant:eqq1}
P(\mathbf z) = Q(T(\mathbf z))\;\;\mbox{for all}\;\;\mathbf z \in \F^n.
\end{equation}
Then
\begin{equation}\label{lem:DegAndLPAreInvariant:eqq2}
\deg_\lambda\left(P_{(\mathbf x)}\right) \leq \deg_\lambda\left(Q_{(\mathbf y)}\right)
\end{equation}
and
\begin{equation}\label{lem:DegAndLPAreInvariant:eqq3}
\dim\left(\mathcal L_{P_{(\mathbf x)}^{(i)}}\right) \leq \dim\left(\mathcal L_{Q_{(\mathbf y)}^{(i)}}\right)\;\;\mbox{for all}\;\; 1 \le i \le \deg_\lambda\left(P_{(\mathbf x)}\right).
\end{equation}
\end{lemma}
\begin{proof}Note that by applying Lemma~\ref{lem:TPreservesPThenTPreservesPI} to $P, Q, \mbf x, \mbf y$ and $T$ we obtain that the equality~\eqref{lem:TPreservesPThenTPreservesPI:eqq2} holds.

Let us prove the inequality~\eqref{lem:DegAndLPAreInvariant:eqq2}. For this let $j \ge \deg_\lambda\left(Q_{(\mathbf y)}\right)$ be nonnegative integer. Then $Q_{(\mathbf y)}^{(j)} = 0$ and consequently
\begin{equation*}
P_{(\mathbf x)}^{(j)}(\mathbf a) \overset{\eqref{lem:TPreservesPThenTPreservesPI:eqq2}}{=\joinrel=\joinrel=} Q_{(\mathbf y)}^{(j)}(T(\mathbf a)) = 0\;\;\mbox{for all}\;\;\mathbf a \in \F^n.
\end{equation*}
Hence, $P_{(\mathbf x)}^{(j)} = 0$ by Corollary~\ref{DegFLessZero}. Therefore, $\deg_\lambda\left(P_{(\mathbf x)}\right) \le \deg_\lambda\left(Q_{(\mathbf y)}\right)$ as required.

In order to prove the inequality~\eqref{lem:DegAndLPAreInvariant:eqq3}, note Lemma~\eqref{lem:DegsPPart} together with the assumption $\deg(P), \deg(Q) < |\F|$ implies that
\begin{equation}\label{lem:DegAndLPAreInvariant:eq3}
\deg \left(P_{(\mathbf x)}^{(i)}\right) < |\F|\;\;\mbox{and}\;\;\deg \left(Q_{(\mathbf y)}^{(i)}\right)\;\;\mbox{for all}\;\;1 \le i \le \deg_\lambda  \left(P_{(\mathbf x)}\right) 
\end{equation}

Relying on the equality~\eqref{lem:TPreservesPThenTPreservesPI:eqq2} and the inequalities in~\eqref{lem:DegAndLPAreInvariant:eq3}, we apply\linebreak Lemma~\ref{lem:DimLPUnderLinConvInj} to $P_{(\mathbf x)}^{(i)}$, $Q_{(\mathbf y)}^{(i)}$ and $T$ for every $1 \le i \le \deg_\lambda\left(P_{(\mathbf x)}\right)$ and obtain the inequality~\eqref{lem:DegAndLPAreInvariant:eqq3}.
\end{proof}

\begin{corollary}\label{cor:DegAndLPAreReallyInvariant}Let $P \in \F[x_1,\ldots, x_n],$ $Q \in \F[x_1,\ldots, x_m],$ $\mathbf x \in \F^n,$ $\mathbf y \in \F^m,$ $T \colon \F^n \to \F^m$ be such that $\deg(P), \deg(Q) < |\F|$, $T$ is a bijective linear map sending $\mathbf x$ to $\mathbf y$ and satisfying $P(\mathbf z) = Q(T(\mathbf z))$ for all $\mathbf z \in \F^n$.  Then
\[
\deg_\lambda\left(P_{(\mathbf x)}\right) = \deg_\lambda\left(Q_{(\mathbf y)}\right)
\]
and
\[
\dim\left(\mathcal L_{P_{(\mathbf x)}^{(i)}}\right) = \dim\left(\mathcal L_{Q_{(\mathbf y)}^{(i)}}\right)\;\;\mbox{for all}\;\; 1 \le i \le \deg_\lambda\left(P_{(\mathbf x)}\right).
\]
\end{corollary}

\begin{lemma}\label{lem:DegAndPartPreservance}Let $P \in \F[x_1,\ldots, x_n],$ $\mathbf x \in \F^n,$ $\mathbf y \in \F^n,$ $T \colon \F^n \to \F^n$ be such that $\deg(P) < |\F|$, $T$ is a bijective linear map such that 
\[
T(\mathbf x) = \mathbf y \pmod{\rad(P)}
\] and $P(\mathbf z) = P(T(\mathbf z))$ for all $\mathbf z \in \F^n$.  Then
\[
\deg_\lambda\left(P_{(\mathbf x)}\right) = \deg_\lambda\left(P_{(\mathbf y)}\right)
\]
and
\[
\dim\left(\mathcal L_{P_{(\mathbf x)}^{(i)}}\right) = \dim\left(\mathcal L_{P_{(\mathbf y)}^{(i)}}\right)\;\;\mbox{for all}\;\; 1 \le i \le \deg_\lambda\left(P_{(\mathbf x)}\right).
\]
\end{lemma}
\begin{proof}Indeed, if $\mathbf y' = T(\mathbf x)$, then
\[
\deg_\lambda\left(P_{(\mathbf x)}\right) \overset{\mbox{\scriptsize Corollary~\ref{cor:DegAndLPAreReallyInvariant}}}{=\joinrel=\joinrel=\joinrel=\joinrel=\joinrel=\joinrel=\joinrel=} \deg_\lambda\left(P_{(\mathbf x')}\right) \overset{\mbox{\scriptsize Lemma~\ref{lem:IfPolyEqThenCompEq}}}{=\joinrel=\joinrel=\joinrel=\joinrel=\joinrel=\joinrel=}  \deg_\lambda\left(P_{(\mathbf y)}\right)
\]
and
\[
\dim\left(\mathcal L_{P_{(\mathbf x)}^{(i)}}\right) \overset{\mbox{\scriptsize Corollary~\ref{cor:DegAndLPAreReallyInvariant}}}{=\joinrel=\joinrel=\joinrel=\joinrel=\joinrel=\joinrel=\joinrel=} \dim\left(\mathcal L_{P_{(\mathbf x')}^{(i)}}\right) \overset{\mbox{\scriptsize Lemma~\ref{lem:IfPolyEqThenCompEq}}}{=\joinrel=\joinrel=\joinrel=\joinrel=\joinrel=\joinrel=} \dim\left(\mathcal L_{P_{(\mathbf y)}^{(i)}}\right)
\]
for all $1 \le i \le \deg_\lambda\left(P_{(\mathbf x)}\right).$
\end{proof}

As a direct consequence of the previous lemma, we obtain the following two lemmas which we use in the sequel for finding $\deg_\lambda (\D_{(X)}),$ $\deg_\lambda (\B_{(X)}),$ $\dim\left(\mathcal L_{\D_{(X)}^{(i)}}\right)$ and $\dim\left(\mathcal L_{\B_{(X)}^{(i)}}\right).$

\begin{lemma}\label{lem:SimPlusScalSameInvD} Assume that $n \ge 2$ and $|\F| > n(n-1)$. Let $X, Y \in \M_{n}(\F)$ be such that $X$ is similar to $Y + cI_n$ for some $c \in \F$. Then $\deg_\lambda (\D_{(X)}) = \deg_\lambda (\D_{(Y)})$ and $\dim\left(\mathcal L_{\D_{(X)}^{(i)}}\right) = \dim\left(\mathcal L_{\D_{(Y)}^{(i)}}\right)$ for all $1 \le \deg_\lambda (\D_{(X)}).$ 
\end{lemma}
\begin{proof}Let $T \colon \M_{n}(\F) \to \M_{n}(\F)$ be a linear map corresponding to a similarity transformation sending $X$ to $Y + cI_n$. This map is bijective as a similarity and preserves $\D$ by the property~\ref{enum:DiscPropertyS} of the matrix discriminant. In addition, $Y + cI_n = Y \pmod{\rad(\D)}$ because $cI_n \in \rad(\D)$ by the property~\ref{enum:DiscPropertyC} of the matrix discriminant. From this we conclude that the required equalities hold by Lemma~\ref{lem:DegAndPartPreservance} applied to $X, Y, \D$ and $T$.
\end{proof}

\begin{lemma}\label{lem:SimPlusScalSameInvD2} Assume that $n \ge 2$ and $|\F| > \frac{n(n-1)}{2}$ and $\Char(\F) = 2$. Let $X, Y \in \M_{n}(\F)$ be such that $Y$ is similar to $X + cI_n$ for some $c \in \F$. Then $\deg_\lambda \left(\B_{(X)}\right) = \deg_\lambda \left(\B_{(Y)}\right)$ and $\dim\left(\mathcal L_{\B_{(X)}^{(i)}}\right) = \dim\left(\mathcal L_{\B_{(Y)}^{(i)}}\right)$ for all $1 \le \deg_\lambda \left(\B_{(X)}\right).$ 
\end{lemma}
\begin{proof}
The proof of this lemma is similar to the proof of Lemma~\ref{lem:SimPlusScalSameInvD}. We only apply  the property~\ref{enum:D2PropertyS} of $\B$ instead of the property~\ref{enum:DiscPropertyS} of the matrix discriminant and the property~\ref{enum:D2PropertyC} of $\B$ instead of the property~\ref{enum:DiscPropertyC} of the matrix discriminant.
\end{proof}

\section{Matrices satisfying $\deg_\lambda \left(\D_{(X)}\right) \le 2n - 2$}
\label{sec:MatsSmallDegDesc}
In this section we provide a description of $n\times n$ matrices $X$ satisfying the inequality $\deg_\lambda \left(\D_{(X)}\right) \le 2n - 2$, where we assume that $n \ge 3.$

\begin{lemma}\label{lem:RankOneScalImpliesLowDeg}Assume that $n \ge 1$ and $|\F| > n(n-1)$. Let $X \in \M_{n}(\F)$. If $X \in \R_n(\F)$,

then
\begin{equation}\label{lem:RankOneScalImpliesLowDeg:eqq1}
\deg_{\lambda} \left(\D_{(X)}\right) \le 2n - 2.
\end{equation}
\end{lemma}

\begin{proof}We will prove sequentially that the statement of the lemma holds in the following cases
\begin{enumerate}[label=Case (\alph*), ref=(\alph*), leftmargin=*]
\item\label{lem:RankOneScalImpliesLowDeg:part1} $X = (r E_{1\,1} + s E_{1\,2})$, where $r, s \in \F$;
\item\label{lem:RankOneScalImpliesLowDeg:part3} $X$ is an arbitrary element of $\R_n(\F)$.
\end{enumerate}

\paragraph{Case~\ref{lem:RankOneScalImpliesLowDeg:part1}} Assume that $X = (r E_{1\,1} + s E_{1\,2})$, where $s, t \in \F$. Let $A \in \M_{n}(\F)$. Consider $p_{A + \lambda X}(t)$:
\[
p_{A + \lambda X}(t) = \left|
\begin{matrix}
t -  a_{1\,1} - \lambda r  & - a_{1\,2} - \lambda s & \ldots & - a_{1\,n} \\
- a_{2\,1}     & t -  a_{2\,2}        & \ldots &  - a_{2\,n} \\
\vdots          & \vdots                & \ddots & \vdots\\
- a_{n\,1}     & -  a_{n\,2}        & \ldots & t - a_{n\,n}
\end{matrix}
\right| \in \F[\lambda, t; a_{1\,1}, \ldots, a_{n\,n}].
\]
Expanding this determinant along the first row we obtain that
\[
p_{A + \lambda X}(t) = P + (k\lambda  + l)Q
\]
for some $P, Q \in \F[t]$ and $k, l \in \F$ with $\deg_t (Q) = n-1$. Hence
\[
p_{A + \lambda X}(x) = x^{n} + (G_{n-1}\lambda  + H_{n-1})x^{n-1} + (G_{n-2}\lambda  + H_{n-2})x^{n-2} + \ldots + (G_{0}\lambda + H_{0}),
\]
for some $G_0, \ldots, G_{n-1}, H_{0}, \ldots, H_{n-1} \in \F.$

Consider then the polynomial $\D_{(X)}(A) \in \F[\lambda]$ expressed as a resultant:
\begin{equation}\label{lem:RankOneScalImpliesLowDeg:eq10}
\D_{(X)}(A) = \Res_{n\, (n-1)} (p_{A + \lambda X}(t),p_{A + \lambda X}'(t)) = \det (R),
\end{equation}
where
\[
R(A) = \begin{pmatrix}
1 & \cdots & G_1\lambda + H_1 & G_{0}\lambda + H_{0} & \multicolumn{1}{c}{\cdots} & 0 \\
\vdots & \ddots & \vdots & \vdots & \ddots & \vdots\\
0 & \multicolumn{1}{c}{\cdots} & 1 & G_{n-1}\lambda + H_{n-1} & \cdots & G_{0}\lambda + H_{0}\\
n & \multicolumn{1}{c}{\cdots} & 2G_{2}\lambda + 2H_{2} & G_{1}\lambda + H_{1} & \multicolumn{1}{c}{\cdots} & 0 \\

\vdots & \ddots & \vdots & \vdots & \ddots & \vdots\\
0 & \multicolumn{1}{c}{\cdots} & n & (n-1)(G_{n-1}\lambda + H_{n-1}) & \multicolumn{1}{c}{\cdots} & G_{1}\lambda + H_{1}
\end{pmatrix}.
\]

Note that $R(A) \in \M_{2n-1}\left(\F[\lambda]\right)$ and the first column of $R(A)$ does not contain $\lambda,$ while every entry in the remaining $2n-2$ rows contains at most the first power of $\lambda$. Hence
\[
\deg_{\lambda}\left(\prod_{1 \le i \le n} R(A)_{i\,\sigma(i)}\right) \le 2n - 1 - 1 = 2n - 2
\]
for all permutations $\sigma.$ Therefore,

\begin{equation}\label{lem:RankOneScalImpliesLowDeg:eq11}
\deg_\lambda (\det (R(A))) \le 2n - 2.
\end{equation}
Thus, 
\begin{equation}\label{lem:RankOneScalImpliesLowDeg:eq10002}
\deg_\lambda \left(\D_{(X)}(A)\right) \overset{\eqref{lem:RankOneScalImpliesLowDeg:eq10}}{=\joinrel=\joinrel=} \deg_\lambda \left(\det (R(A))\right) \overset{\eqref{lem:RankOneScalImpliesLowDeg:eq11}}{\le} 2n - 2.
\end{equation}
Since the equality~\eqref{lem:RankOneScalImpliesLowDeg:eq10002} holds for all $A \in \M_n(\F)$, and $\deg_\lambda \left(\D_{(X)}\right) < |\F|$ by the initial assumption of the lemma and Lemma~\ref{lem:DegsPPart}, then $\deg_\lambda \left(\D_{(X)}\right) \le 2n - 2$ by Lemma~\ref{DegFLessZero}.

\paragraph{Case~\ref{lem:RankOneScalImpliesLowDeg:part3}} Assume that $X \in \R_n(\F)$. This means that  $X = J + cI_n$ for some $c \in \F$ and a matrix $J \in \M_{n}(\F)$ with $\rk(J) \le 1$. If $\rk(J) = 0$, then $J = O_n$ and consequently $\D(A + \lambda X) = \D(A)$ for all $A \in \M_{n}(\F)$ by the property~\ref{enum:DiscPropertyC} of the matrix discriminant. Hence,
\[
\deg_{\lambda} (\D(A + \lambda X)) = 0 \le 2n-2\;\;\mbox{for all}\;\;A \in \M_{n}(\F)
\]
in this case and consequently $\deg_\lambda \left(\D_{(X)}\right) \le 2n - 2$ by Lemma~\ref{DegFLessZero}.

Suppose now that $\rk(J) = 1.$  Hence, $J$ is similar to the linear combination $Y = r E_{1\,1} + s E_{1\,2}$ for some $r, s \in \F$ by Lemma~\ref{lem:RK1Sim}. From the part~\ref{lem:RankOneScalImpliesLowDeg:part1} of the lemma we conclude that $\deg_{\lambda}\left(\D_{(Y)}\right) \le 2n - 2$. Then, by applying Lemma~\ref{lem:SimPlusScalSameInvD} to $X$ and $Y$ we obtain that $\deg_{\lambda}\left(\D_{(X)}\right) \le 2n - 2.$
\end{proof}

We  also use the lower bound $\deg_\lambda (\D_{(X)}) \ge 1$ for non-scalar elements of $\R_n(\F)$. To establish this bound, we need one supplementary lemma.

\begin{lemma}\label{lem:E12PlusStrictUpDegBound}Assume that $n \ge 3$ and $|\F| > n(n-1)$. If $C \in \M_{(n-2)}(\F)$ is strictly upper triangular and $B = \begin{psmallmatrix}0 & 1\\ 0 & 0\end{psmallmatrix}\oplus C$, then $\deg_\lambda \left(\D_{\left(B\right) }\right)) > 0$.
\end{lemma}
\begin{proof}The conditions of the lemma mean that
\[
B = \left[
\begin{array}{ c | c }
    B_0 & * \\
    \hline
    0 & B'
  \end{array}\right],
\]
where $B_0 = \begin{psmallmatrix}0 & 1 & 0\\ 0 & 0 & 0\\ 0 & 0 & 0\end{psmallmatrix}$ and $B' \in \M_{(n-3)}(\F)$ is strictly upper triangular. Let $A \in \M_{n}(\F)$ be defined by $A = A_0 \oplus A'$, where  $A_0 = \begin{psmallmatrix} 1 & 0 & 0\\  1 & 0 & 0\\ 0 & 0 & 0\end{psmallmatrix}, A' = \diag(a_4,\ldots, a_n)$ and $a_4,\ldots a_n$ are all distinct and nonzero (since we assume that $|\F| > n(n-1)$, such $a_4, \ldots, a_n$ indeed exist). Then
\begin{multline}\label{lem:NonScalarElementofRnLowBound:eq10002}
\D_{(X)}(A) = \D(A + \lambda X)\\
\overset{\eqref{lem:DiscDirSum:eqq1}}{=\joinrel=\joinrel=} \D(A_0 + \lambda X_0)  \prod_{4 \le k \le n} p_{(A_0 + \lambda X_0)}^2 (a_k) \prod_{4 \le i < j \le n} (a_i - a_j)^2\\
\lambda \left(4\lambda^2 + \lambda\right)\cdot\prod_{4\le  k \le n} \left(a_k^3-a_k^2 - \lambda a_k\right)^2 \cdot \prod_{4\le i < j \le n} (a_i - a_j)^2
\end{multline}
Since  $a_4, \ldots, a_n$ are pairwise distinct and nonzero, then from~\eqref{lem:NonScalarElementofRnLowBound:eq10002} we conclude that $\deg_\lambda \left(\D_{(X)}(A)\right) \ge 2 + 2(n-3) > 0.$ Therefore, $\deg_\lambda \left(D_{(X)}\right) > 0$.
\end{proof}

\begin{lemma}\label{lem:NonScalarElementofRnLowBound}Assume that $n \ge 3$ and $|\F| > n(n-1)$. If $X \in \left(\R_n(\F) \setminus \langle I_n \rangle\right),$ then $\deg_\lambda (\D_{(X)}) > 0$.
\end{lemma}
\begin{proof}We will prove sequentially that the statement of the lemma holds in the following cases
\begin{enumerate}[label=Case (\alph*), ref=(\alph*), leftmargin=*]
\item\label{lem:NonScalarElementofRnLowBound:part1} $X = (c E_{1\,1})$, where $0 \neq c \in \F $;
\item\label{lem:NonScalarElementofRnLowBound:part2} $X = (E_{1\,2})$;
\item\label{lem:NonScalarElementofRnLowBound:part3} $X$ is an arbitrary element of $\R_n(\F) \setminus \langle I_n \rangle$.
\end{enumerate}

\paragraph{Case~\ref{lem:NonScalarElementofRnLowBound:part1}} Let $A \in \M_{n}(\F)$ be defined by $A = \diag(0, a_2 \ldots, a_n)$, where $a_2,\ldots a_n$ are pairwise distinct (since we assume that $|\F| > n(n-1)$, such $a_2, \ldots, a_n$ indeed exist). Then
\begin{multline}\label{lem:NonScalarElementofRnLowBound:eq10001}
\D_{(X)}(A) = \D(A + \lambda X) = \D\left(\diag\left(\lambda c, a_2, \ldots, a_n\right)\right)\\
 \overset{\mbox{\scriptsize Property~\ref{enum:DiscPropertyDiag}}}{=\joinrel=\joinrel=\joinrel=\joinrel=\joinrel=\joinrel=\joinrel=\joinrel=\joinrel=} \prod_{2 \le k  \le n} (c\lambda - a_k)^2 \times \prod_{2\le i < j \le n} (a_i - a_j)^2.
\end{multline}
Since $c \neq 0$ and $a_2, \ldots, a_n$ are pairwise distinct, then from~\eqref{lem:NonScalarElementofRnLowBound:eq10001} we conclude that $\deg_\lambda \left(\D_{(X)}(A)\right) = 2(n-1) > 0.$ Therefore, $\deg_\lambda \left(D_{(X)}\right) > 0$.

\paragraph{Case~\ref{lem:NonScalarElementofRnLowBound:part2}}  In this case the matrix $X$ satisfies the conditions of Lemma~\ref{lem:E12PlusStrictUpDegBound} and consequently $\deg_\lambda \left(D_{(X)}\right) > 0$.

\paragraph{Case~\ref{lem:NonScalarElementofRnLowBound:part3}}  The condition $X \in \left(\R_n(\F) \setminus \langle I_n \rangle\right)$ implies that  $X = J + cI_n$ for some $c \in \F$ and a matrix $J \in \M_{n}(\F)$ with $\rk(J) = 1$. 

Now there are two possibilities for $J$: either $J$ is nilpotent or $J$ is non-nilpotent. Using Lemma~\ref{lem:RK1Sim} it could be reformulated as follows: either $J$ is similar to $E_{1\,2}$ or $J$ is similar to $cE_{1\,1}$ for some $0 \neq c \in \F$. The first case is considered in the part~\ref{lem:NonScalarElementofRnLowBound:part2} of the lemma, while the second case is considered in  part~\ref{lem:NonScalarElementofRnLowBound:part1} of it. In the both cases we obtain that $\deg_\lambda \left(D_{(J)}\right) > 0$ by applying Lemma~\ref{lem:SimPlusScalSameInvD} to $J$ and $E_{1\,2}$ or $cE_{1\,1}$, correspondingly. Therefore $\deg_\lambda \left(D_{(X)}\right) > 0$ by Lemma~\ref{lem:SimPlusScalSameInvD} applied to $J$ and $X$.
\end{proof}

\begin{lemma}\label{lem:LowDegImpliesLowRankOrSeveralCasesForAlgebraicallyClosed}Assume that $\F$ is an algebraically closed field. If $X \in \M_{n}(\F)$ is such that 
\begin{equation}\label{lem:LowDegImpliesLowRankOrSeveralCasesForAlgebraicallyClosed:eqq1}
\deg_{\lambda} \left(\D_{(X)}\right) \le 2n - 2,
\end{equation}

then $X$ satisfies one of the following conditions:
\begin{enumerate}[label=C\arabic*., ref=C\arabic*]
\item\label{lem:LowDegImpliesLowRankOrSeveralCasesForAlgebraicallyClosed:cond1}\label{lem:LowDegImpliesLowRankOrSeveralCasesForAlgebraicallyClosed:condfst} $X \in \R_n(\F)$.
\item\label{lem:LowDegImpliesLowRankOrSeveralCasesForAlgebraicallyClosed:cond2} $n = 3$ and $X$ is similar to $\begin{psmallmatrix}0 & 1 & 0\\0 & 0 & 1\\ 0 & 0 & 0\end{psmallmatrix} + cI_n$ for some $c \in \F$.
\item\label{lem:LowDegImpliesLowRankOrSeveralCasesForAlgebraicallyClosed:cond3} $n = 4$ and $X$ is similar to $\begin{psmallmatrix}0 & 1 & 0 & 0\\0 & 0 & 0 & 0\\ 0 & 0 & 0 & 1\\ 0 & 0 & 0 & 0\end{psmallmatrix}  + cI_n$ for some $c \in \F$.
    
\item\label{lem:LowDegImpliesLowRankOrSeveralCasesForAlgebraicallyClosed:cond4}\label{lem:LowDegImpliesLowRankOrSeveralCasesForAlgebraicallyClosed:condlst} $n = 5$, $\Char(\F) = 2$ and  $X$ is similar to $\begin{psmallmatrix}0 & 1 & 0 & 0 & 0\\0 & 0 & 0 & 0 & 0\\ 0 & 0 & 0 & 1 & 0\\ 0 & 0 & 0 & 0 & 0\\ 0 & 0 & 0 & 0 & 0\end{psmallmatrix}  + cI_n$ for some $c \in \F$.
\end{enumerate}
\end{lemma}
\begin{proof}We will prove sequentially that the statement of the lemma holds in the following cases
\begin{enumerate}[label=Case (\alph*), ref=(\alph*), leftmargin=*]
\item\label{lem:LowDegImpliesLowRankOrSeveralCasesForAlgebraicallyClosed:part1} For all $X$ in the Jordan normal form satisfying the condition~\eqref{lem:LowDegImpliesLowRankOrSeveralCasesForAlgebraicallyClosed:eqq1} such that the eigenvalue of $X$ of maximal multiplicity is equal to zero.
\item\label{lem:LowDegImpliesLowRankOrSeveralCasesForAlgebraicallyClosed:part2} For all $X$ satisfying the condition~\eqref{lem:LowDegImpliesLowRankOrSeveralCasesForAlgebraicallyClosed:eqq1}.
\end{enumerate}

\paragraph{Case~\ref{lem:LowDegImpliesLowRankOrSeveralCasesForAlgebraicallyClosed:part1}} By contradiction. Let $X \in \M_{n}(\F)$ be a matrix in the Jordan normal form satisfying the condition~\eqref{lem:LowDegImpliesLowRankOrSeveralCasesForAlgebraicallyClosed:eqq1} such that the eigenvalue of $X$ of maximal multiplicity is equal to zero and $X$ does not satisfy any of the conditions~\ref{lem:LowDegImpliesLowRankOrSeveralCasesForAlgebraicallyClosed:condfst}--\ref{lem:LowDegImpliesLowRankOrSeveralCasesForAlgebraicallyClosed:condlst}. The following subcases are possible:
\begin{enumerate}[{label=\Roman*., ref=\Roman*}]
\item\label{lem:LowDegImpliesLowRankOrSeveralCasesForAlgebraicallyClosed:part11} $X$ has at least three nonequal eigenvalues.
\item\label{lem:LowDegImpliesLowRankOrSeveralCasesForAlgebraicallyClosed:part12} $X$ has only two nonequal eigenvalues.

\item\label{lem:LowDegImpliesLowRankOrSeveralCasesForAlgebraicallyClosed:part13} All the eigenvalues of $X$ are equal to each other.

\end{enumerate}
Let us consider them separately.

\paragraph{Subcase~\ref{lem:LowDegImpliesLowRankOrSeveralCasesForAlgebraicallyClosed:part11}. $X$ has at least three nonequal eigenvalues.} Let us choose a matrix $A = \diag(a_1,\ldots a_n)$, where $a_1, \ldots, a_n$ are all distinct (recall that $\F$ is algebraically closed). Then
\begin{multline*}
\D(A + \lambda X) = {}\\
{} = \D(\diag(a_1 + \lambda\xi_1, \ldots, a_n + \lambda\xi_n)) = \phantom{(a_i - a_j + \lambda(\xi_i - \xi_j))^2}\\
{} = \prod_{i < j} (a_i - a_j + \lambda(\xi_i - \xi_j))^2 = {}\\
{} = (-1)^{\frac{n(n-1)}{2}} \prod_{1 \le i \ne j \le n} (a_i - a_j + \lambda(\xi_i - \xi_j)).
\end{multline*}
Hence,
\begin{equation}\label{lem:LowDegImpliesLowRankOrSeveralCasesForAlgebraicallyClosed:eq3}
\deg_{\lambda}(\D(A + \lambda X)) = \#\{(i,j)\mid \xi_i \neq \xi_j\}
\end{equation}

Since $X$ has at least three nonequal eigenvalues, then there are $1 \le i_1, i_2, i_3 \le n$ such that $\xi_{i_1}, \xi_{i_2}$ and $\xi_{i_3}$ are pairwise distinct. Let $a = \#\{1 \le j \le n\mid \xi_j = \xi_1\},$ $b = \#\{1 \le j \le n \mid \xi_j = \xi_2\}$ and $c = m - a - b$. Note that $c \ge \#\{1 \le j \le n\mid \xi_j = \xi_3\}$. Therefore $a, b, c \ge 1$ and consequently
\[
abc > (a-1)(b-1)(c-1).
\]
By expanding the both sides of the inequality we obtain that
\[
0 > a + b + c - ab - bc - ac - 1
\]
which implies that
\[
ab + bc + ac > a + b + c - 1 \overset{\scriptsize\mbox{the definition of $c$}}{=\joinrel=\joinrel=\joinrel=\joinrel=\joinrel=\joinrel=\joinrel=\joinrel=\joinrel=} n - 1.
\]
Hence,
\begin{equation}\label{lem:LowDegImpliesLowRankOrSeveralCasesForAlgebraicallyClosed:eq1}
2ab + 2bc + 2ac > 2n - 2.
\end{equation}
In addition,
\begin{multline}\label{lem:LowDegImpliesLowRankOrSeveralCasesForAlgebraicallyClosed:eq2}
\#\{(i,j)\mid \xi_i \neq \xi_j\} \ge 2\#\left(\{(i,j)\mid \xi_i = \xi_1\;\mbox{and}\; \xi_j = \xi_2\}\right)\\
 + 2\#\left(\{(i,j)\mid \xi_i = \xi_1\;\mbox{and}\; \xi_j \neq \xi_1, \xi_2\}\right)\\
+ 2\#\left(\{(i,j)\mid \xi_i = \xi_2\;\mbox{and}\; \xi_i \neq \xi_1, \xi_2\}\right)\\
 = 2ab + 2bc + 2ac.
\end{multline}
By aligning the inequalities~\eqref{lem:LowDegImpliesLowRankOrSeveralCasesForAlgebraicallyClosed:eq1} and~\eqref{lem:LowDegImpliesLowRankOrSeveralCasesForAlgebraicallyClosed:eq2} together we conclude that
\[
\deg_{\lambda}(\D(A + \lambda X)) \overset{\eqref{lem:LowDegImpliesLowRankOrSeveralCasesForAlgebraicallyClosed:eq3}}{=\joinrel=\joinrel=}\#\{(i,j)\mid \xi_i \neq \xi_j\} > 2n - 2
\]
which contradicts with~\eqref{lem:LowDegImpliesLowRankOrSeveralCasesForAlgebraicallyClosed:eqq1}.

\paragraph{Subcase~\ref{lem:LowDegImpliesLowRankOrSeveralCasesForAlgebraicallyClosed:part12}. $X$ has only two nonequal eigenvalues} As we assume, the eigenvalue $\xi_2$ of maximal multiplicity is equal to zero. Let $m$ be multiplicity of $\xi_2 = 0$. Then the multiplicity of the second eigenvalue $\xi_1 \neq 0$ is equal to $n - m$. Note that these assumptions imply that
\begin{equation}\label{lem:LowDegImpliesLowRankOrSeveralCasesForAlgebraicallyClosed:eq4}
1 \le (n-m) \le m \le n.
\end{equation}
Hence, if $m = 1$, then $n - m  = 1$, which implies that $n = 2$ and $X$ has the form $\begin{psmallmatrix}* & 0\\ 0 & 0\end{psmallmatrix}$ and consequently rank one. This contradicts with the condition~\ref{lem:LowDegImpliesLowRankOrSeveralCasesForAlgebraicallyClosed:condfst}.

Suppose that $1 < m < n - 1$. Then
\begin{multline*}
2m(n-m) - (2n - 2) = 2(n(m-1) - (m^2 - 1)) = {}\\
{} = 2((n - 1 - m)(m-1)) > 0.
\end{multline*}
Hence,
\begin{equation}\label{lem:LowDegImpliesLowRankOrSeveralCasesForAlgebraicallyClosed:eq5}
2m(n-m) > (2n - 2).
\end{equation}
Similarly to the subcase~\ref{lem:LowDegImpliesLowRankOrSeveralCasesForAlgebraicallyClosed:part11}, let us choose a matrix $A = \diag(a_1,\ldots a_n)$, where $a_1, \ldots, a_n$ are all distinct. Then
\begin{multline*}
\deg_{\lambda}(\D(A + \lambda X))\\
 \overset{\eqref{lem:LowDegImpliesLowRankOrSeveralCasesForAlgebraicallyClosed:eq3}}{=\joinrel=\joinrel=} \#\{(i,j)\mid \xi_i \neq \xi_j\} = 2\#\{(i,j)\mid \xi_i = \xi_1\;\;\mbox{and}\;\; \xi_j = \xi_2\}\\
  = 2m(n - m) \overset{\eqref{lem:LowDegImpliesLowRankOrSeveralCasesForAlgebraicallyClosed:eq5}}{>} 2n - 2
\end{multline*}
which contradicts with the condition~\eqref{lem:LowDegImpliesLowRankOrSeveralCasesForAlgebraicallyClosed:eqq1}. Therefore, $m \ge n - 1$. Together with~\eqref{lem:LowDegImpliesLowRankOrSeveralCasesForAlgebraicallyClosed:eq4} this implies that $m = n-1$. Thus, the only possible case is that $1 < m = n - 1$. This implies in particular that $n \ge 3$. Let us consider two possible subsubcases.  

\begin{enumerate}[{label=\ref{lem:LowDegImpliesLowRankOrSeveralCasesForAlgebraicallyClosed:part12}.\arabic*., ref=\ref{lem:LowDegImpliesLowRankOrSeveralCasesForAlgebraicallyClosed:part12}.\arabic*}]
\item\label{lem:LowDegImpliesLowRankOrSeveralCasesForAlgebraicallyClosed:part121} every Jordan cell of $X$ corresponding to $\xi_2 = 0$ has size 1.
\item\label{lem:LowDegImpliesLowRankOrSeveralCasesForAlgebraicallyClosed:part122} there is a Jordan cell of $X$ corresponding to $\xi_2 = 0$ of size greater than 2.
\end{enumerate}

\paragraph{Subsubcase~\ref{lem:LowDegImpliesLowRankOrSeveralCasesForAlgebraicallyClosed:part121}} In this case $X$ has the form $\begin{psmallmatrix}* & 0 & \cdots & 0\\ 0 & 0 & \cdots & 0\\ \vdots & \vdots & \ddots & \vdots\\ 0 & 0 & \cdots & 0\end{psmallmatrix}$ and consequently has rank one which means that  $X \in \R_{n}(\F)$. This contradicts with the condition~\ref{lem:LowDegImpliesLowRankOrSeveralCasesForAlgebraicallyClosed:condfst}.

\paragraph{Subsubcase~\ref{lem:LowDegImpliesLowRankOrSeveralCasesForAlgebraicallyClosed:part122}} In this case $X = \left[\begin{array}{ c | c }
    X_0 & * \\
    \hline
    0 & X'
  \end{array}\right]
$
with $X_0 = \begin{psmallmatrix} \xi_1 & 0 & 0\\ 0 & 0 & 1\\ 0 & 0 & 0\end{psmallmatrix}$ and $X' \in \M_{n-3}(\F)$ is a strictly upper triangular matrix. Let us choose a matrix $A = A_0 \oplus A'$, where  $A_0 = \begin{psmallmatrix} 0 & 0 & 0\\ 0 & 0 & 0\\ 1 & 0 & 0\end{psmallmatrix}, A' = \diag(a_4,\ldots, a_n)$ and $a_4,\ldots a_n$ are all distinct. Then
\begin{multline*}
\D_{(X)}(A) = \D(A + \lambda X)\\
\overset{\eqref{lem:DiscDirSum:eqq1}}{=\joinrel=\joinrel=} \D(A_0 + \lambda X_0)  \prod_{4 \le k \le n} p_{(A_0 + \lambda X_0)}^2 (a_k) \prod_{4 \le i < j \le n} (a_i - a_j)^2\\
= (\lambda^5 \xi_1^4 + \lambda^4 \xi_1^2 + \lambda^3) \cdot \prod_{4 \le k \le n} \left(a_k^3-\lambda \xi_1 a_k^2 - \lambda a_k + \lambda^2\xi_1\right)^2\cdot \prod_{4\le i < j \le n} (a_i - a_j)^2.
\end{multline*}
Therefore, $\deg_{\lambda} (\D_{(X)}(A)) = 5 + 2(2(n-3)) + 0 = 4n - 7.$ Note that $4n - 7 \ge 2n - 1$  when $n \ge 3$. Thus, we obtain the contradiction with~\eqref{lem:LowDegImpliesLowRankOrSeveralCasesForAlgebraicallyClosed:eqq1}.

\paragraph{Subcase~\ref{lem:LowDegImpliesLowRankOrSeveralCasesForAlgebraicallyClosed:part13}. All the eigenvalues of $X$ are equal to each other} It follows from our assumptions that all eigenvalues of $X$ are equal to zero. The following subsubcases are possible.

\begin{enumerate}[{label=\ref{lem:LowDegImpliesLowRankOrSeveralCasesForAlgebraicallyClosed:part13}.\arabic*., ref=\ref{lem:LowDegImpliesLowRankOrSeveralCasesForAlgebraicallyClosed:part13}.\arabic*}]
\item\label{lem:LowDegImpliesLowRankOrSeveralCasesForAlgebraicallyClosed:part131} $X$ has a Jordan block of size at least 3.
\item\label{lem:LowDegImpliesLowRankOrSeveralCasesForAlgebraicallyClosed:part132} The size of each Jordan block of $X$ is at most 2.
\end{enumerate}

\paragraph{Subsubcase~\ref{lem:LowDegImpliesLowRankOrSeveralCasesForAlgebraicallyClosed:part131}. $X$ has a Jordan block of size at least 3}  Note that in this subsubcase $n \ge 4$ because otherwise $X$ satisfies the condition~\ref{lem:LowDegImpliesLowRankOrSeveralCasesForAlgebraicallyClosed:cond2}. In this subcase
$X = \left[\begin{array}{ c | c }
    X_0 & * \\
    \hline
    0 & X'
  \end{array}\right]
$
with $X_0 = \begin{psmallmatrix} 0 & 1 & 0\\ 0 & 0 & 1\\ 0 & 0 & 0\end{psmallmatrix}$ and strictly upper triangular $X' \in \M_{n-3}(\F).$ Let us choose a matrix $A = A_0 \oplus A'$, where  $A_0 = \begin{psmallmatrix} 0 & 0 & 0\\ 1 & 0 & 0\\ 1 & 0 & 0\end{psmallmatrix}, A' = \diag(a_4,\ldots, a_n)$ and $a_4,\ldots a_n$ are all distinct. Then
\begin{multline*}
\D_{(X)}(A) = \D(A + \lambda X)\\
 \overset{\eqref{lem:DiscDirSum:eqq1}}{=\joinrel=\joinrel=} \D(A_0 + \lambda X_0)  \prod_{4 \le k \le n} p_{(A_0 + \lambda X_0)}^2 (a_k) \prod_{4 \le i < j \le n} (a_i - a_j)^2\\
= (4\lambda^3 - 27\lambda^4) \cdot \prod_{4 \le k \le n} \left(a_k^3-\lambda a_k-\lambda^2\right)^2\cdot \prod_{4\le i < j \le n} (a_i - a_j)^2.
\end{multline*}

Therefore,
\begin{equation*}
\deg_{\lambda} \left(\D_{(X)}(A)\right)= \begin{cases}
4 + 2(2(n-3)) = 4n - 8, & \Char(\F) \neq 3;\\
3 + 2(2(n-3)) = 4n - 9, & \Char(\F) = 3
\end{cases}.
\end{equation*}
Note that if $n \ge 4$ then $4n - 8 \ge 2n - 1$ and $4n - 9 \ge 2n - 1$. The inequality $n \ge 4$ holds because otherwise $X$ satisfies the condition~\ref{lem:LowDegImpliesLowRankOrSeveralCasesForAlgebraicallyClosed:cond2}. Thus, $\deg_{\lambda} \left(\D_{(X)}(A)\right) \ge 2n - 1$ for chosen $A$ which contradicts with~\eqref{lem:LowDegImpliesLowRankOrSeveralCasesForAlgebraicallyClosed:eqq1}.

\paragraph{Subsubcase~\ref{lem:LowDegImpliesLowRankOrSeveralCasesForAlgebraicallyClosed:part132}. The size of each Jordan block of $X$ is at most 2} Since we assume that $X \not \in \R_{n}(\F)$, then $X$ is not a nilpotent matrix of rank 1. Hence, $X$ has at least two such Jordan blocks which implies that $n \ge 4$. 

Thus, $X = X_0 \oplus X'$, where $X_0 = \begin{psmallmatrix} 0 & 1\\ 0 & 0\end{psmallmatrix} \oplus \begin{psmallmatrix} 0 & 1\\ 0 & 0\end{psmallmatrix}$ and $X' \in \M_{n-4}(\F).$ Let us choose a matrix $A = A_0 \oplus A'$, where $A_0 = \begin{psmallmatrix} 1 & 0 & 0 & 0\\ 0 & 0 & 1 & 0\\ 0 & 0 & 0 & 0\\ 1 & 0 & 0 & 0 \end{psmallmatrix}, A' = \diag(a_5,\ldots, a_n)$ and $a_5,\ldots a_n$ are all distinct. Then
\begin{multline*}
\D_{(X)}(A) = \D(A + \lambda X)\\
 \overset{\eqref{lem:DiscDirSum:eqq1}}{=\joinrel=\joinrel=} \D(A_0 + \lambda X_0)  \prod_{5 \le k \le n} p_{(A_0 + \lambda X_0)}^2 (a_k) \prod_{5 \le i < j \le n} (a_i - a_j)^2\\
 = \lambda^4(27+256\lambda^2) \cdot \prod_{5 \le k \le n} \left(a_k^4-a_k^3-\lambda^2\right)^2\cdot \prod_{5\le i < j \le n} (a_i - a_j)^2.
 \end{multline*}

Hence,
\begin{equation*}
\deg_{\lambda} \left(\D_{(X)}(A)\right) =
\begin{cases}
6 + 2(2(n-4)) = 4n - 10, & \Char(\F) \neq 2;\\
4 + 2(2(n-4)) = 4n - 12, & \Char(\F) = 2
\end{cases}.
\end{equation*}
Note that $4n - 10 \ge 2n - 1$ when $n \ge 5$ and $4n - 12 \ge 2n - 1$ when $n \ge 6.$ The inequality $n \ge 5$ holds because otherwise $X$ satisfies the condition~\ref{lem:LowDegImpliesLowRankOrSeveralCasesForAlgebraicallyClosed:cond3}. If $n \ge 5$, then $\Char(\F) \neq 2$ because otherwise $X$ satisfies the condition~\ref{lem:LowDegImpliesLowRankOrSeveralCasesForAlgebraicallyClosed:cond4}. Thus, $\deg_{\lambda} \left(\D_{(X)}(A)\right) \ge 2n - 1$ for chosen $A$ which contradicts with~\eqref{lem:LowDegImpliesLowRankOrSeveralCasesForAlgebraicallyClosed:eqq1}.
\end{proof}

\section{The necessary explicit values and estimates of $\deg_\lambda \left(\D_{(X)}\right)$,\linebreak $\dim \left(\FL_{\D^{(i)}_{(X)}}\right)$, $\deg_\lambda \left(\B_{(X)}\right)$ and $\dim \left(\FL_{\B^{(i)}_{(X)}}\right)$}

\label{sec:InvariantsOfSmallMatrices}

The results of the previous section imply that for $n\ge 6$ the matrix $X \in \M_n(\F)$ belongs to $\R_n(\F)$ if and only if $\deg_\lambda \left(\D_{(X)}\right) \le 2n - 2$, while for $n \le 5$ there several exceptions exist (see Lemma~\ref{lem:LowDegImpliesLowRankOrSeveralCasesForAlgebraicallyClosed}). Thus, in order to distinguish these exceptional matrices from the elements of $\R_n(\F)$ we need to calculate more invariants in these cases. The summary table of all calculations made in this section is provided in Table~\ref{tbl:compTab} at the end of this section.

Throughout this section, if $\Char(\F) = 2$ is assumed, then the ``-'' signs are omitted and/or replaced by ``+'' signs. 

\begin{lemma}\label{lem:DistinctN3CH2RK1NNIL}Assume that $\Char (\F) = 2.$ If $B = \begin{psmallmatrix}1 & 0 & 0\\ 0 & 0 & 0\\ 0 & 0 & 0\end{psmallmatrix}$, then 
\[\deg_{\lambda} \left(\B_{(B)}\right) = 2\;\;\mbox{and}\;\;\dim(\mathcal L_{\B_{(B)}^{(1)}}) = 4.\]
 
\end{lemma}
\begin{proof}Let $p_{A}(x) = x^3 + b_2x^2 + b_1x + b_0$. Then
\begin{multline*}
p_{A + \lambda B}(t) = \begin{vmatrix}x + a_{1\,1} + \lambda & a_{1\,2} & a_{1\,3}\\ a_{2\,1} & x + a_{2\,2} & a_{2\,3}\\ a_{3\,1} & a_{3\,2} & x + a_{3\,3} \end{vmatrix}\\
 = x^3 + (b_2 + \lambda)x^2 + (b_1 + \lambda a_{3\,3} + \lambda a_{2\,2})x + b_0 + \lambda a_{2\,2}a_{3\,3} + \lambda a_{2\,3}a_{3\,2}. 
\end{multline*}
Hence, 
\begin{multline}\label{lem:DistinctN3CH2RK1NNIL:eq1}
\B_{(B)} = \B(p_{A + \lambda B}(x))\\
 \overset{\eqref{lem:D2PolyN3Char2:eqq1}}{=\joinrel=\joinrel=} (b_0 + \lambda a_{2\,2}a_{3\,3} + \lambda a_{2\,3}a_{3\,2}) + (b_1 + \lambda a_{3\,3} + \lambda a_{2\,2})(b_2 + \lambda).
\end{multline}
Therefore, by inspecting the right-hand side of~\eqref{lem:DistinctN3CH2RK1NNIL:eq1} we obtain that\linebreak $\deg_{\lambda} \left(\B_{(B)}\right) = 2$ and 

\begin{multline*}
\B_{(B)}^{(1)} \overset{\eqref{lem:DistinctN3CH2RK1NNIL:eq1}}{=\joinrel=\joinrel=} a_{2\,2} a_{3\,3} + a_{2\,3} a_{3\,2} + b_1  + (a_{3\,3} + a_{2\,2})b_2
\overset{\eqref{eqqq:formulaChPoly}}{=\joinrel=\joinrel=} a_{2\,2} a_{3\,3} + a_{2\,3} a_{3\,2}\\
+ (a_{1\,2} a_{2\,1} + a_{1\,1} a_{2\,2} + a_{1\,3} a_{3\,1} + a_{2\,3} a_{3\,2} + a_{1\,1} a_{3\,3} + a_{2\,2} a_{3\,3})\\
+ (a_{3\,3} + a_{2\,2})(a_{1\,1} + a_{2\,2} + a_{3\,3})\\
= a_{1\,2} a_{2\,1} + a_{2\,2}^2  + a_{1\,3} a_{3\,1} + a_{3\,3}^2.
\end{multline*}

In order to find $\mathcal L_{\B_{(B)}^{(1)}}$ and its dimension, let us denote $\B_{(B)}^{(1)}$ by $P$ and find the partial derivatives $\frac{\partial P}{\partial E_{1\,1}}, \ldots, \frac{\partial P}{\partial E_{3\,3}}.$
\begin{equation}\label{lem:DistinctN3CH2RK1NNIL:eq2001}
\begin{pmatrix}
\frac{\partial P}{\partial E_{1\,1}} & \frac{\partial P}{\partial E_{1\,2}} & \frac{\partial P}{\partial E_{1\,3}}\\
\frac{\partial P}{\partial E_{2\,1}} & \frac{\partial P}{\partial E_{2\,2}} & \frac{\partial P}{\partial E_{2\,3}}\\
\frac{\partial P}{\partial E_{3\,1}} & \frac{\partial P}{\partial E_{3\,2}} & \frac{\partial P}{\partial E_{3\,3}}\\
\end{pmatrix}
= 
\begin{pmatrix}
0 &  a_{2\,1} & a_{3\,1}\\
a_{1\,2} & 0 & 0\\
a_{1\,3} & 0 & 0
\end{pmatrix}.
\end{equation}

Since the function  $\mathbf v \to \frac{\partial P}{\partial \mathbf v}$ is linear by Lemma~\ref{lem:DirDerLin}, the equality~\eqref{lem:DistinctN3CH2RK1NNIL:eq2} implies that 
\begin{equation}\label{lem:DistinctN3CH2RK1NNIL:eq2}
\dim (\FL_P) \le 4.
\end{equation}

In addition, the matrices $A_1 = \begin{psmallmatrix}0 & 1 & 0\\ 0 & 0 & 0\\ 0 & 0 & 0\end{psmallmatrix}, A_2 = \begin{psmallmatrix}0 & 0 & 1\\0 & 0 & 0\\ 0 & 0 & 0\end{psmallmatrix}, A_3 = \begin{psmallmatrix}0 & 0 & 0\\ 1 & 0 & 0\\ 0 & 0 & 0\end{psmallmatrix}$ and $A_4 = \begin{psmallmatrix}0 & 0 & 0\\ 0 & 0 & 0\\ 1 & 0 & 0\end{psmallmatrix}$ yield the following linear functionals
\[
l_{P, A_1}(Z) = z_{2\,1}, \quad l_{P, A_2}(Z) = z_{3\,1}, \quad l_{P, A_3}(Z) = z_{1\,2}, \quad l_{P, A_4}(Z) = z_{2\,1}
\]
which belong to $\FL_P$ and are linearly independent. From this we conclude that
\begin{equation}\label{lem:DistinctN3CH2RK1NNIL:eq3}
\dim (\FL_P) \ge 4.
\end{equation}
Finally, the inequalities~\eqref{lem:DistinctN3CH2RK1NNIL:eq2} and~\eqref{lem:DistinctN3CH2RK1NNIL:eq3} together imply that
\[
\dim(\FL_P) = \dim \left(\FL_{\B_{(B)}^{(1)}}\right)= 4.
\]
\end{proof}

\begin{lemma}\label{lem:DistinctN3CH2RK2}Assume that $\Char (\F) = 2.$ If $B = \begin{psmallmatrix}0 & 1 & 0\\ 0 & 0 & 1\\ 0 & 0 & 0\end{psmallmatrix}$, then 
\[\deg_{\lambda} \left(\B_{(B)}\right) = 2\;\;\mbox{and}\;\;\dim(\mathcal L_{\B_{(B)}^{(1)}}) = 6.\]
\end{lemma}

\begin{proof}Let $p_{A}(t) = x^3 + b_2x^2 + b_1x + b_0$. Then
\begin{multline*}
p_{A + \lambda B}(t) = \begin{vmatrix}t + a_{1\,1} & a_{1\,2} + \lambda & a_{1\,3}\\ a_{2\,1} & x + a_{2\,2} & a_{2\,3} + \lambda \\ a_{3\,1} & a_{3\,2} & x + a_{3\,3} \end{vmatrix}\\
 = x^3 + b_2x^2  + (b_1 + \lambda (a_{2\,1} + a_{3\,2}))x\phantom{XXXXX}\\
  + b_0 + \lambda^2 a_{3\,1}  + \lambda (a_{2\,1} a_{3\,3} + a_{1\,1} a_{3\,2} + a_{2\,3}a_{3\,1} + a_{1\,2}a_{3\,1}).
\end{multline*}
Hence, 
\begin{multline}\label{lem:DistinctN3CH2RK2:eq1001}
\B_{(B)} = \B(p_{A + \lambda B}(t))\\
  \overset{\eqref{lem:D2PolyN3Char2:eqq1}}{=\joinrel=\joinrel=}  (b_0 + \lambda^2 a_{3\,1}  + \lambda (a_{2\,1} a_{3\,3} + a_{1\,1} a_{3\,2} + a_{2\,3}a_{3\,1} + a_{1\,2}a_{3\,1}))\\
   +(b_1 + \lambda (a_{2\,1} + a_{3\,2}))(b_2).
\end{multline}
Therefore, by inspecting the right-hand side of~\eqref{lem:DistinctN3CH2RK2:eq1001} we obtain that\linebreak $\deg_{\lambda} \left(\B_{(B)}\right) = 2$ and 
\begin{multline*}
\B_{(B)}^{(1)} \overset{\eqref{lem:DistinctN3CH2RK2:eq1001}}{=\joinrel=\joinrel=} (a_{2\,1} a_{3\,3} + a_{1\,1} a_{3\,2} + a_{2\,3}a_{3\,1} + a_{1\,2}a_{3\,1}) + (a_{2\,1} + a_{3\,2})b_2\\
  \overset{\eqref{eqqq:formulaChPoly}}{=\joinrel=\joinrel=} (a_{2\,1} a_{3\,3} + a_{1\,1} a_{3\,2} + a_{2\,3}a_{3\,1} + a_{1\,2}a_{3\,1})\phantom{XXXX}\\
  \phantom{XXXXXXXX} + (a_{2\,1} + a_{3\,2})(a_{1\,1} + a_{2\,2} + a_{3\,3}) \\
 = a_{1\,1} a_{2\,1} + a_{2\,1} a_{2\,2} + a_{1\,2} a_{3\,1} + a_{2\,3} a_{3\,1} + a_{2\,2} a_{3\,2} + 
 a_{3\,2} a_{3\,3}.
\end{multline*}
In order to find $\mathcal L_{\B_{(B)}^{(1)}}$ and its dimension, let us denote $\B_{(B)}^{(1)}$ by $P$ and find the partial derivatives $\frac{\partial P}{\partial E_{1\,1}}, \ldots, \frac{\partial P}{\partial E_{3\,3}}.$
\begin{equation}\label{lem:DistinctN3CH2RK2:eq1}
\begin{pmatrix}
\frac{\partial P}{\partial E_{1\,1}} & \frac{\partial P}{\partial E_{1\,2}} & \frac{\partial P}{\partial E_{1\,3}}\\
\frac{\partial P}{\partial E_{2\,1}} & \frac{\partial P}{\partial E_{2\,2}} & \frac{\partial P}{\partial E_{2\,3}}\\
\frac{\partial P}{\partial E_{3\,1}} & \frac{\partial P}{\partial E_{3\,2}} & \frac{\partial P}{\partial E_{3\,3}}\\
\end{pmatrix}
= 
\begin{pmatrix}
a_{2\,1} &  a_{3\,1} & 0\\
a_{1\,1} & a_{3\,2} & a_{3\,1}\\
a_{1\,2} + a_{2\,3} & a_{2\,2} + a_{3\,3} &  a_{3\,2}
\end{pmatrix}.
\end{equation}
On the one hand, from~\eqref{lem:DistinctN3CH2RK2:eq1} we conclude that 
\[l_{P, A}(Z_1) = l_{P, A}(Z_2) = l_{P, A}(Z_3) = 0\;\;\mbox{for all}\;\;A \in \M_{3}(\F),\]
where 
\[Z_1 = \begin{psmallmatrix}0 & 0 & 1\\ 0 & 0 & 0\\ 0 & 0 & 0\end{psmallmatrix}, Z_2 = \begin{psmallmatrix}0 & 1 & 0\\ 0 & 0 & 1\\ 0 & 0 & 0\end{psmallmatrix}, Z_3 = \begin{psmallmatrix}0 & 0 & 0\\ 0 & 1 & 0\\ 0 & 0 & 1\end{psmallmatrix}.\]
Since $Z_1, Z_2, Z_3$ are linearly independent, then 
\begin{equation}\label{lem:DistinctN3CH2RK2:eq2}
\dim (\mathcal L_P) \le 9 - 3 = 6.
\end{equation}

On the other hand, the matrices $A_1 = \begin{psmallmatrix}0 & 0 & 0\\ 1 & 0 & 0\\ 0 & 0 & 0\end{psmallmatrix},$ $A_2 = \begin{psmallmatrix}0 & 0 & 0\\0 & 0 & 0\\ 1 & 0 & 0\end{psmallmatrix},$ $A_3 = \begin{psmallmatrix}1 & 0 & 0\\ 0 & 0 & 0\\ 0 & 0 & 0\end{psmallmatrix},$ $A_4 = \begin{psmallmatrix}0 & 0 & 0\\ 0 & 0 & 0\\ 0 & 1 & 0\end{psmallmatrix},$ $A_5 = \begin{psmallmatrix}0 & 1 & 0\\ 0 & 0 & 0\\ 0 & 0 & 0\end{psmallmatrix}$ and $A_6 = \begin{psmallmatrix}0 & 0 & 0\\ 0 & 1 & 0\\ 0 & 0 & 0\end{psmallmatrix}$ we yield the following linear functionals
\[
\begin{array}{ccc}
l_{P, A_1}(Z) = z_{1\,1}, & l_{P, A_2}(Z) =  (z_{1\,2} + z_{2\,3}), & l_{P, A_3}(Z) =  z_{2\,1},\\
l_{P, A_4}(Z) = (z_{2\,2} + z_{3\,3}), & l_{P, A_5}(Z) =  z_{3\,1}, & l_{P, A_6}(Z) =  z_{3\,2}
\end{array}
\]
which belong to $\mathcal L_P$ and are linearly independent. Hence,
\begin{equation}\label{lem:DistinctN3CH2RK2:eq3}
\dim (\mathcal L_P) \ge 6.
\end{equation}
Finally, the inequalities~\eqref{lem:DistinctN3CH2RK2:eq2} and~\eqref{lem:DistinctN3CH2RK2:eq3} together imply that 
\[
\dim(\FL_P) = \dim \left(\FL_{\B_{(B)}^{(1)}}\right)= 6.
\]

\end{proof}

\begin{lemma}\label{lem:DistinctN3CH3RK1NNIL}Assume that $\Char (\F) = 3.$ If $B = \begin{psmallmatrix}1 & 0 & 0\\ 0 & 0 & 0\\ 0 & 0 & 0\end{psmallmatrix},$ then $\deg_\lambda \left(\D_{(B)}\right) = 4$.
\end{lemma}
\begin{proof}Let $p_{A}(t) = x^3 + b_2x^2 + b_1x + b_0$. Then
\begin{multline*}
p_{A + \lambda X}(x) = \begin{vmatrix}x - a_{1\,1} - \lambda & -a_{1\,2} & - a_{1\,3}\\ -a_{2\,1} & x - a_{2\,2} & -a_{2\,3}\\ -a_{3\,1} & -a_{3\,2} & x - a_{3\,3} \end{vmatrix}\\
 = x^3 + (b_2 - \lambda)x^2 + (b_1 + \lambda a_{3\,3} + \lambda a_{2\,2})x + b_0 - \lambda a_{2\,2}a_{3\,3} + \lambda a_{2\,3}a_{3\,2}. 
\end{multline*}
Hence, 
\begin{multline}\label{lem:DistinctN3CH3RK1NNIL:eq1}
\D_{(X)} = \D(p_{A + \lambda X}(x))\\
 \overset{\eqref{lem:DPolyN3Char3:eqq1}}{=\joinrel=\joinrel=}  2(b_1 + \lambda a_{3\,3} + \lambda a_{2\,2})^3 + (b_1 + \lambda a_{3\,3} + \lambda a_{2\,2})^2(b_2 - \lambda)^2\\ + 2(b_0 - \lambda a_{2\,2}a_{3\,3} + \lambda a_{2\,3}a_{3\,2})(b_2 - \lambda)^3.
\end{multline}
Therefore, by inspecting the right-hand side of~\eqref{lem:DistinctN3CH3RK1NNIL:eq1} we obtain that \linebreak $\deg_\lambda \left(\D_{(B)}\right) = 4$. 

\end{proof}

\begin{lemma}\label{lem:DistinctN3CH3RK2}Assume that $\Char (\F) = 3.$ If $B = \begin{psmallmatrix}0 & 1 & 0\\ 0 & 0 & 1\\ 0 & 0 & 0\end{psmallmatrix}$, then $\deg_\lambda \left(\D_{(B)}\right) = 3$. 
\end{lemma}

\begin{proof}Let $p_{A}(t) = x^3 + b_2x^2 + b_1x + b_0$. Then
\begin{multline*}
p_{A + \lambda B}(x) = \begin{vmatrix}x - a_{1\,1} & -a_{1\,2} - \lambda & - a_{1\,3}\\ -a_{2\,1} & x - a_{2\,2} & -a_{2\,3} -\lambda \\ -a_{3\,1} & -a_{3\,2} & x - a_{3\,3} \end{vmatrix}\\
 = x^3 + b_2x^2 + (b_1 - \lambda (a_{2\,1} + a_{3\,2}))x\phantom{XXXX}\\
  + b_0 - \lambda^2 a_{3\,1}  + \lambda (a_{2\,1} a_{3\,3} + a_{1\,1} a_{3\,2} - a_{2\,3}a_{3\,1} - a_{1\,2}a_{3\,1}). 
\end{multline*}
Hence, 
\begin{multline}\label{lem:DistinctN3CH3RK2:eqAAAA1}
\D_{(B)}  = \D(p_{A + \lambda B}(x))\\
  \overset{\eqref{lem:DPolyN3Char3:eqq1}}{=\joinrel=\joinrel=}  2(b_1 - \lambda (a_{2\,1} + a_{3\,2}))^3 + (b_1 - \lambda (a_{2\,1} + a_{3\,2}))^2b_2^2\\
+ 2(b_0 - \lambda^2 a_{3\,1}  + \lambda (a_{2\,1} a_{3\,3} + a_{1\,1} a_{3\,2} - a_{2\,3}a_{3\,1} - a_{1\,2}a_{3\,1}))b_2^3.
\end{multline}
Therefore, by inspecting the right-hand side of~\eqref{lem:DistinctN3CH3RK2:eqAAAA1} we obtain that\linebreak $\deg_\lambda \left(\D_{(B)}\right) = 3$.

\end{proof}

\begin{lemma}\label{lem:DistinctN3CHN23RK1NNIL}Assume that $\Char (\F) \neq 2, 3.$ If $B = \begin{psmallmatrix}1 & 0 & 0\\ 0 & 0 & 0\\ 0 & 0 & 0\end{psmallmatrix}$, then
\[\deg_{\lambda} \left(\D_{(B)}\right) = 4\;\;\mbox{and}\;\;\dim(\mathcal L_{\D_{(B)}^{(4)}}) = 3.\]

\end{lemma}
\begin{proof}Let $p_{A}(x) = x^3 + b_2x^2 + b_1x + b_0$. Then
\begin{multline*}
p_{A + \lambda B}(x) = \begin{vmatrix}x - a_{1\,1} - \lambda & -a_{1\,2} & - a_{1\,3}\\ -a_{2\,1} & x - a_{2\,2} & -a_{2\,3}\\ -a_{3\,1} & -a_{3\,2} & x - a_{3\,3} \end{vmatrix}\\
 = x^3 + (b_2 - \lambda)x^2 + (b_1 + \lambda a_{3\,3} + \lambda a_{2\,2})x + b_0 - \lambda a_{2\,2}a_{3\,3} + \lambda a_{2\,3}a_{3\,2}.
\end{multline*}
Hence, 
\begin{multline}\label{lem:DistinctN3CHN23RK1NNIL:eq1}
\D_{(B)} = \D(p_{A + \lambda B}(t))\\
 \overset{\eqref{lem:DPolyN3:eqq1}}{=\joinrel=\joinrel=} -27(b_0 + \lambda a_{2\,2}a_{3\,3} - \lambda a_{2\,3}a_{3\,2})^2 - 4(b_1 - \lambda a_{3\,3} - \lambda a_{2\,2})^3\\
  + 18(b_0 + \lambda a_{2\,2}a_{3\,3} - \lambda a_{2\,3}a_{3\,2})(b_1 - \lambda a_{3\,3} - \lambda a_{2\,2})(b_2 + \lambda)\\
   + (b_1 - \lambda a_{3\,3} - \lambda a_{2\,2})^2(b_2 + \lambda)^2 - 4(b_0 + \lambda a_{2\,2}a_{3\,3} - \lambda a_{2\,3}a_{3\,2})(b_2 + \lambda)^3.
\end{multline}
Therefore, by inspecting the right-hand side of~\eqref{lem:DistinctN3CHN23RK1NNIL:eq1} we obtain that \linebreak $\deg_\lambda \left(\D_{(B)}\right) = 4$ and
\begin{equation*}
\D_{(B)}^{(4)} = P =  (a_{3\,3} + a_{2\,2})^2 - 4(a_{2\,2}a_{3\,3} - a_{2\,3}a_{3\,2}) = (a_{3\,3} - a_{2\,2})^2 +4a_{2\,3}a_{3\,2}.
\end{equation*}

In order to find $\dim(\mathcal L_{P})$, let us calculate $\frac{\partial P}{\partial E_{1\,1}}, \ldots, \frac{\partial P}{\partial E_{3\,3}}.$
\begin{equation}\label{lem:DistinctN3CHN23RK1NNIL:eq1001}
\begin{pmatrix}
\frac{\partial P}{\partial E_{1\,1}} & \frac{\partial P}{\partial E_{1\,2}} & \frac{\partial P}{\partial E_{1\,3}}\\
\frac{\partial P}{\partial E_{2\,1}} & \frac{\partial P}{\partial E_{2\,2}} & \frac{\partial P}{\partial E_{2\,3}}\\
\frac{\partial P}{\partial E_{3\,1}} & \frac{\partial P}{\partial E_{3\,2}} & \frac{\partial P}{\partial E_{3\,3}}\\
\end{pmatrix}
= 
\begin{pmatrix}
0 &  0 & 0\\
0 & 2a_{2\,2} - 2a_{3\,3} & 4a_{3\,2}\\
0 & 4a_{2\,3} & -2a_{2\,2} + 2a_{3\,3}
\end{pmatrix}.
\end{equation}

On the one hand, from~\eqref{lem:DistinctN3CHN23RK1NNIL:eq1001} we conclude that 
\[l_{P, A}(Z) = 0\;\;\mbox{for all}\;\;A \in \M_{3}(\F),\]
whenever $Z$ has the form $\begin{psmallmatrix}z_1 & z_2 & z_3\\ z_4 & z_5 & 0\\ z_6 & 0 & z_5\end{psmallmatrix}$ for some $z_{i} \in \F,$ $1 \le i \le 6$. Hence, $\dim \left(\left(\FL_P\right)^0\right) \ge 6$
and  the basic properties of the vector spaces and their annihilators infer that
\begin{equation}\label{lem:DistinctN3CHN23RK1NNIL:eq2}
\dim (\mathcal L_P) \le 9 - 6 = 3.
\end{equation}

On the other hand, the matrices $A_1 = \begin{psmallmatrix}0 & 0 & 0\\0 & 1 & 0\\ 0 & 0 & 0\end{psmallmatrix},$ $A_2 = \begin{psmallmatrix}0 & 0 & 0\\0 & 0 & 1\\ 0 & 0 & 0\end{psmallmatrix}$ and\linebreak $A_3 = \begin{psmallmatrix}0 & 0 & 0\\0 & 0 & 0\\ 0 & 1 & 0\end{psmallmatrix}$ yield the following linear functionals
\[
l_{P, A_1}(Z) = 2z_{2\,2} - 2z_{3\,3},\quad l_{P, A_2}(Z) = 4z_{3\,2}, \quad l_{P, A_3}(Z) = 4z_{2\,3},
\]
which belong to $\mathcal L_P$ and are linearly independent. Hence,
\begin{equation}\label{lem:DistinctN3CHN23RK1NNIL:eq3}
\dim (\mathcal L_P) \ge 3.
\end{equation}

Finally, the inequalities~\eqref{lem:DistinctN3CHN23RK1NNIL:eq2} and~\eqref{lem:DistinctN3CHN23RK1NNIL:eq3} together imply that\linebreak $\dim\left(\FL_{\D_{(B)}^{(4)}}\right) = \dim(\mathcal L_P) = 3$.
\end{proof}

\begin{lemma}\label{lem:DistinctN3CHN23RK2}Assume that $\Char (\F) \neq 2, 3.$ If $B = \begin{psmallmatrix}0 & 1 & 0\\ 0 & 0 & 1\\ 0 & 0 & 0\end{psmallmatrix}$, then 
\[\deg_\lambda\left(\D_{(B)}\right) = 4\;\;\mbox{and}\;\;\dim(\mathcal L_{\D_{(B)}^{(4)}}) = 1.\]
\end{lemma}

\begin{proof}Let $p_{A}(x) = x^3 + b_2x^2 + b_1x + b_0$. Then
\begin{multline*}
p_{A + \lambda B}(x) = \begin{vmatrix}x - a_{1\,1} & -a_{1\,2} - \lambda & - a_{1\,3}\\ -a_{2\,1} & x - a_{2\,2} & -a_{2\,3} -\lambda \\ -a_{3\,1} & -a_{3\,2} & x - a_{3\,3} \end{vmatrix}\\
  = x^3 + b_2x^2 + (b_1 - \lambda (a_{2\,1} + a_{3\,2}))x\phantom{XXXX}\\
  + b_0 - \lambda^2 a_{3\,1}  + \lambda (a_{2\,1} a_{3\,3} + a_{1\,1} a_{3\,2} - a_{2\,3}a_{3\,1} - a_{1\,2}a_{3\,1}). 
\end{multline*}
Hence, 
\begin{multline*}
\D_{(B)} \overset{\eqref{lem:DPolyN3:eqq1}}{=\joinrel=\joinrel=} -27( b_0 - \lambda^2 a_{3\,1}  + \lambda (a_{2\,1} a_{3\,3} + a_{1\,1} a_{3\,2} - a_{2\,3}a_{3\,1} - a_{1\,2}a_{3\,1}))^2\\
   - 4 (b_1 - \lambda (a_{2\,1} + a_{3\,2}))^3\\
   + 18(b_0 - \lambda^2 a_{3\,1}  + \lambda (a_{2\,1} a_{3\,3} + a_{1\,1} a_{3\,2} - a_{2\,3}a_{3\,1} - a_{1\,2}a_{3\,1}))(b_1 - \lambda (a_{2\,1} + a_{3\,2}))b_2\\
   + (b_1 - \lambda (a_{2\,1} + a_{3\,2}))^2b_2^2 - 4(b_0 - \lambda^2 a_{3\,1}  + \lambda (a_{2\,1} a_{3\,3} + a_{1\,1} a_{3\,2} - a_{2\,3}a_{3\,1} - a_{1\,2}a_{3\,1}))b_2^3.
\end{multline*}
Therefore, by inspecting the right-hand side of~\eqref{lem:DistinctN3CHN23RK1NNIL:eq1} we obtain that \linebreak $\deg_\lambda \left(\D_{(X)}\right) = 4$ and
\begin{equation*}
\D_{(B)}^{(4)} = P =  -27a_{3\,1}^2.
\end{equation*}

In order to find $\dim(\mathcal L_{P})$, let us calculate $\frac{\partial P}{\partial E_{1\,1}}, \ldots, \frac{\partial P}{\partial E_{3\,3}}.$
\begin{equation}\label{lem:DistinctN3CHN23RK2:eq1}
\begin{pmatrix}
\frac{\partial P}{\partial E_{1\,1}} & \frac{\partial P}{\partial E_{1\,2}} & \frac{\partial P}{\partial E_{1\,3}}\\
\frac{\partial P}{\partial E_{2\,1}} & \frac{\partial P}{\partial E_{2\,2}} & \frac{\partial P}{\partial E_{2\,3}}\\
\frac{\partial P}{\partial E_{3\,1}} & \frac{\partial P}{\partial E_{3\,2}} & \frac{\partial P}{\partial E_{3\,3}}\\
\end{pmatrix}
= 
\begin{pmatrix}
0 & 0 & 0\\
0 & 0 & 0\\
-54a_{3\,1} & 0 & 0
\end{pmatrix}.
\end{equation}

On the one hand, from~\eqref{lem:DistinctN3CHN23RK2:eq1} we conclude that 
\[l_{P, A}(Z) = 0\;\;\mbox{for all}\;\;A \in \M_{3}(\F),\]
whenever $Z$ has the form $\begin{psmallmatrix}z_1 & z_2 & z_3\\ z_4 & z_5 & z_6\\ 0 & z_7 & z_8\end{psmallmatrix}$ for some $z_{i} \in \F,$ $1 \le i \le 8$. Hence, $\dim \left(\left(\FL_P\right)^0\right) \ge 8$
and  the basic properties of the vector spaces and their annihilators infer that
\begin{equation}\label{lem:DistinctN3CH3RK2:eq2}
\dim (\mathcal L_P) \le 9 - 8 = 1.
\end{equation}

On the other hand, the matrix $A_1 = \begin{psmallmatrix}0 & 0 & 0\\0 & 0 & 0\\ 1 & 0 & 0\end{psmallmatrix}$ yields the  linear functional
\[
l_{P, A_1}(Z) = -54z_{3\,1}
\]
which belongs to $\mathcal L_P$ and nonzero. Hence,
\begin{equation}\label{lem:DistinctN3CH3RK2:eq3}
\dim (\mathcal L_P) \ge 1.
\end{equation}

Finally, the inequalities~\eqref{lem:DistinctN3CH3RK2:eq2} and~\eqref{lem:DistinctN3CH3RK2:eq3} together imply that $$\dim\left(\FL_{\D_{(B)}^{(4)}}\right) = \dim(\mathcal L_P) = 1$$.
\end{proof}

\begin{lemma}\label{lem:DistinctN4CHN2RK1NNIL}Assume that $\Char (\F) \neq 2$ and $|\F| > 4\cdot (4-1) = 2$.  If $B = \begin{psmallmatrix}1 & 0 & 0 & 0\\ 0 & 0 & 0 & 0\\ 0 & 0 & 0 & 0\\ 0 & 0 & 0 & 0\end{psmallmatrix}$, then $\deg_{\lambda}\left(\D_{(B)}\right) = 6$ and $\dim\left(\FL_{\D_{(B)}^{(6)}}\right) \ge 5$.
\end{lemma}

\begin{proof}Let $p_{A}(x) = x^4 + a_3x^3 + a_2x^2 + a_1x + a_0$ and
\begin{equation}\label{lem:DistinctN4CHN2RK1NNIL:eq1111}
 p_{A(1|1)}(x) = x^3 + a'_2x^2 + a'_1x + a'_0.
\end{equation}
Using the multilinearity of the determinant along the first column we obtain
\begin{multline}\label{lem:DistinctN4CHN2RK1NNIL:eq11}
p_{A + \lambda B}(x)\\ = \begin{vmatrix}x - a_{1\,1} - \lambda & -a_{1\,2} & -a_{1\,3} & - a_{1\,4}\\ -a_{2\,1} & x - a_{2\,2} & -a_{2\,3} & -a_{2\,4}\\ -a_{3\,1} & -a_{3\,2} & x - a_{3\,3} & -a_{3\,4}\\ - a_{4\,1} & -a_{4\,2} & -a_{4\,3} & x - a_{4\,4} \end{vmatrix}
 = p_{A}(x) - \lambda  p_{A(1|1)}(x)\\
 = x^4 + (a_3 - \lambda)x^3 + (a_2 - \lambda a'_2)x^2 + (a_1 - \lambda a'_1)x + (a_0 - \lambda a'_0). 
\end{multline}
Hence, 

\begin{equation}\label{lem:DistinctN4CHN2RK1NNIL:eq1}
\D_{(B)}(A) = \D(p_{A + \lambda B}(x))= \det(R(A)),
\end{equation}
where $R(A)$ is defined by
\begin{equation}\label{lem:DistinctN4CHN2RK1NNIL:eq2}
R(A) = \begin{psmallmatrix} 
1 & (a_3 - \lambda) & (a_2 - \lambda a'_2) & (a_1 - \lambda a'_1) & (b_0 - \lambda a'_0) & 0 & 0 \\
0 & 1 & (a_3 - \lambda) & (a_2 - \lambda a'_2) & (a_1 - \lambda a'_1) & (b_0 - \lambda a'_0) & 0  \\
0 & 0 & 1 & (a_3 - \lambda) & (a_2 - \lambda a'_2) & (a_1 - \lambda a'_1) & (b_0 - \lambda a'_0)  \\
4 & 3(a_3 - \lambda) & 2(a_2 - \lambda a'_2) & (a_1 - \lambda a'_1) & 0 & 0 & 0\\
0 & 4 & 3(a_3 - \lambda) & 2(a_2 - \lambda a'_2) & (a_1 - \lambda a'_1) & 0 & 0 \\
0 & 0 & 4 & 3(a_3 - \lambda) & 2(a_2 - \lambda a'_2) & (a_1 - \lambda a'_1) & 0 \\
0 & 0 & 0 & 4 & 3(a_3 - \lambda) & 2(a_2 - \lambda a'_2) & (a_1 - \lambda a'_1)\\
\end{psmallmatrix}
\end{equation}
Let us find $P = \D_{(B)}^{(6)}$ by considering the coefficient of $\lambda^6$ in~\eqref{lem:DistinctN4CHN2RK1NNIL:eq2}. Since the first column of the matrix $R$ does not contain $\lambda$ and the other entries of $R$ contain only the first power of $\lambda$ and $R(A)$ has 7 rows and columns, $P$ could be found from the equality
\begin{equation}\label{lem:DistinctN4CHN2RK1NNIL:eq3}
\D_{(B)}^{(6)} = P = \det (M_P),
\end{equation}
where
\begin{equation}\label{lem:DistinctN4CHN2RK1NNIL:eq4}
M_P(A) = 
\begin{pmatrix} 
1 & -1 & -a'_2 & - a'_1 & -a'_0 & 0 & 0 \\
0 & 0 & -1 & -a'_2 & -a'_1 & -a'_0 & 0  \\
0 & 0 & 0 & -1 & -a'_2 & -a'_1 &  -a'_0 \\
4 & -3 & -2a'_2 & - a'_1 & 0 & 0 & 0\\
0 & 0 & -3 & -2a'_2 & -a'_1 & 0 & 0 \\
0 & 0 & 0 & -3 & -2a'_2 & -a'_1 & 0 \\
0 & 0 & 0 & 0 & -3 & -2a'_2 & -a'_1\\
\end{pmatrix}
\end{equation}
whose 2th--7th columns are obtained from $R(A)$ by taking the coefficients of $\lambda$ in the corresponding entries of $R(A)$, and the first column is equal to the first column of $R(A)$. 

Let $A_1 = \begin{psmallmatrix}0 & 0 & 0 & 0\\ 0 & 1 & 0 & 0\\ 0 & 0 & 0 & 1\\ 0 & 0 & 0 & 0\end{psmallmatrix}$. In order to find the formula for the linear functional $l_{P, A_1}$, consider $Z \in \M_4(\F)$. According to Definition~\ref{def:PLFunc} and Definition~\ref{def:partDer}, the value of $l_{P, A_1}(Z)$ is equal to the coefficient of $\kappa$ in the expansion of $P(A_1 + \kappa Z)$ by powers of $\kappa$. Accordingly, $P$ is expressed through $a'_0, a'_1$ and $a'_2$ by~\eqref{lem:DistinctN4CHN2RK1NNIL:eq3} and~\eqref{lem:DistinctN4CHN2RK1NNIL:eq4}. Relying on the definition of $a'_0, a'_1, a'_2$ in~\eqref{lem:DistinctN4CHN2RK1NNIL:eq1111} and the formula~\eqref{eqqq:formulaChPoly}, we obtain that

\begin{equation}\label{lem:DistinctN4CHN2RK1NNIL:eq5}
a'_0 = -\det ((A_1 + \kappa Z)(1|1)) = \kappa x_{4\,3} + \mbox{(terms of higher degree in $\kappa$)},
\end{equation}

\begin{multline}\label{lem:DistinctN4CHN2RK1NNIL:eq6}
a'_1 = \mbox{sum of all principal $2\times 2$ minors of $(A_1 + \kappa Z)(1|1)$}\\
= \kappa (x_{3\,3} -  x_{4\,3} + x_{4\,4}) + \mbox{(terms of higher degree in $\kappa$)}
\end{multline}
and
$$a'_2 = -\tr ((A_1 + \kappa Z)(1|1)) = -1 -\kappa(x_{2\,2} + x_{3\,3} + x_{4\,4}).$$

The equalities~\eqref{lem:DistinctN4CHN2RK1NNIL:eq5} and~\eqref{lem:DistinctN4CHN2RK1NNIL:eq6} imply that all the entries of the last column $M_P(A_1 + \kappa Z)$  are divisible by the first power of $\kappa$. Therefore,  the coefficient of $\kappa$ in  $P(A_1 + \kappa Z)$ is equal to $\det(M'_R)$, where
\begin{equation*}
M'_P(A_1 + \kappa Z)= \begin{pmatrix} 
1 & -1 & 1 & 0 & 0 & 0 & 0 \\
0 & 0 & -1 & 1 & 0 & 0 & 0  \\
0 & 0 & 0 & -1 & 1 & 0 & -z_{4\,3}\\
4 & -3 & 2 & 0 & 0 & 0 & 0\\
0 & 0 & -3 & 2 & 0 & 0 & 0 \\
0 & 0 & 0 & -3 & 2 & 0 & 0 \\
0 & 0 & 0 & 0 & -3 & 2 & -(z_{3\,3} -  z_{4\,3} + z_{4\,4})\\
\end{pmatrix}
\end{equation*}
which is obtained from $M_P(A_1 + \kappa Z)$ by taking the coefficient of $\kappa$ from the entries of its last column and the constant term from the other entries of $M_P(A_1 + \kappa Z)$. Thus,
\begin{equation}\label{lem:DistinctN4CHN2RK1NNIL:eq7}
l_{P, A_{1}}(Z) = \det M'_P(A_1 + \kappa Z) = 4x_{4\,3}\;\;\mbox{for all}\;\;X \in \M_{4}(\F).
\end{equation}

Consider the bijective linear maps $T', T_{(2 3)}$ and $T_{(2 4)}$ on $\M_{4}(\F)$ defined by
\[
T'(Y) = Y^t,\;\; T_{(2 3)}\left((y_{i\,j})\right) = (y_{\sigma_1(i)\,\sigma_1(j)}),\;\;T_{(2 4)}\left((x_{i\,j})\right) = (x_{\sigma_2(i)\,\sigma_2(j)})
\]
for all $Y = (y_{i\,j}) \in \M_{4}(\F),$ where $\sigma_1 = (2 3)$ and $\sigma_2 = (2 4)$ are permutations. Performing the same permutation $\sigma$ on the matrix rows and columns is represented by the similarity map $Y \mapsto M_{\sigma}^{-1}XM_{\sigma}$, where $M_\sigma = \sum_{i = 1}^{n} E_{i\,\sigma(i)}$. Therefore, $T_{(2 3)}$ and $T_{(2 4)}$ preserve $\D$. The same is true for $T'$ by the because the $p_{Y} = p_{Y^t}$ for every $Y \in \M_{4}(\F)$.  In addition, $T', T_{(2 3)}$ and $T_{(2 4)}$ send $B$ to $B$. Hence, by applying Lemma~\ref{lem:TPreservesPThenTPreservesPI} to these maps and $B$, we obtain that $T', T_{(2 3)}$ and $T_{(2 4)}$ preserve $P = \D_{(B)}^{(6)}.$ Using this fact and Lemma~\ref{thm:surjpres} applied to $P, l_{P, A_{1}}$ and $T$ we conclude that the linear functionals
\[
l' = l_{P, A_{1}} \circ T',\;\;\, l_{(2 3)} = l_{P, A_{1}} \circ T_{(2 3)},\;\;\, l_{(2 4)} = l_{P, A_{1}} \circ T_{(2 4)},\;\;\, l'' = l_{P, A_{1}} \circ T_{(2 4)} \circ T' 
\]
belong to $\FL_P$. Since
\[
\begin{aligned}
l'(Z) = 4\left(T'(Z)\right)_{4\,3} = 4z_{3\,4}, &\phantom{XXXX} l_{(2 3)}(Z) = 4\left(T_{(2 3)}(Z)\right)_{4\,3} = 4z_{4\,2},\\
l_{(2 4)}(Z) = 4\left(T_{(2 4)}(Z)\right)_{4\,3} = 4z_{2\,3}, &\phantom{XXXX} l''(Z) = 4\left(T'(Z)\right)_{2\,3} = 4z_{3\,2}
\end{aligned}
\]
for all $Z \in M_{4}(\F)$, then $l, l', l_{(2 3)}, l_{(2 4)}$ and $l''$ are linearly independent. Therefore, $\dim\left(\FL_P\right) \ge 5$.
\end{proof}
\begin{remark}The reader can notice that the assumption $|\F| > n(n-1) = \deg(\D)$ for $n = 4$ is added in the statement of the above lemma. This assumption arises from Lemma~\ref{lem:TPreservesPThenTPreservesPI} and Lemma~\ref{thm:surjpres} which are used to generate new elements of $\FL_{\D_{(B)}^{(6)}}$ from a given one. However, the direct calculation show that statement of Lemma~\ref{lem:DistinctN3CHN23RK2} remains true without this assumption. In addition, it is possible to modify the formulations and proofs of Lemma~\ref{lem:TPreservesPThenTPreservesPI} and Lemma~\ref{thm:surjpres} in such a way that they would not restrict the cardinality of $\F$ but would remain applicable in the context of Lemma~\ref{lem:DistinctN3CHN23RK2}. But we find it unnecessary to do that here because the assumption $|\F| > n(n-1)$ appears in further steps of the proof of the main theorem of this paper.

For the same reason the restriction on the cardinality of $\F$ appears in statements of some other lemmas in this section.
\end{remark}

\begin{lemma}\label{lem:DistinctN4CHN2RK2}Assume that $\Char (\F) \neq 2.$ If $B = \begin{psmallmatrix}0 & 1 & 0 & 0\\ 0 & 0 & 0 & 0\\ 0 & 0 & 0 & 1\\ 0 & 0 & 0 & 0\end{psmallmatrix}$, then $\deg_{\lambda} \left(\D_{(B)}\right) = 6$ and  $\dim\left(\FL_{\D_{(X)}}\right) \leq 4$.
\end{lemma}

\begin{proof}Let $p_{A}(x), p'(x), p''(x), a''' \in \F[x; a_{1\,1}, \ldots, a_{1\,4}, a_{2\,1}, \ldots, a_{4\,4}]$ be defined by
\begin{equation}\label{lem:DistinctN4CHN2RK2:eq101}
p_{A}(x) = \det ((xI_4 - A),
\end{equation}
\begin{equation}\label{lem:DistinctN4CHN2RK2:eq102}
p'(x) = \det ((xI_4 - A)(1|2)),
\end{equation}
\begin{equation}\label{lem:DistinctN4CHN2RK2:eq103}
p''(x) = \det ((xI_4 - A)(3|4))
\end{equation}
and
\begin{equation}\label{lem:DistinctN4CHN2RK2:eq104}
a''' = \det((xI_4 - A)(1,3|2,4)) = \begin{vmatrix}-a_{2\,1} & - a_{2\,3}\\ -a_{4\,1} & - a_{4\,3}\end{vmatrix} \in \F.
\end{equation}
Note that the definitions of $p'$ and $p''$ imply that $\deg_x(p') = \deg_x(p'') = 2$. Hence, there are $a'_0, \ldots, a'_2, a''_0, \ldots, a''_2 \in \F$ such that
\begin{equation}\label{lem:DistinctN4CHN2RK2:eq105}
p'(t)  = a'_2x^2 + a'_1x + a'_0\;\;\mbox{and}\;\;p''(t)  = a''_2x^2 + a''_1x + a''_0.
\end{equation}

Using the multilinearity of the determinant along the second and the fourth column
\begin{multline}\label{lem:DistinctN4CHN2RK2:eq1}
p_{A + \lambda B}(x) = \begin{vmatrix}x - a_{1\,1}  & -a_{1\,2}- \lambda & -a_{1\,3} & - a_{1\,4}\\ -a_{2\,1} & x - a_{2\,2} & -a_{2\,3} & -a_{2\,4}\\ -a_{3\,1} & -a_{3\,2} & x - a_{3\,3} & -a_{3\,4} - \lambda  \\ - a_{4\,1} & -a_{4\,2} & -a_{4\,3} & x - a_{4\,4} \end{vmatrix}\\
 = p_{A}(x) + \lambda (p'(x) + p''(x)) + \lambda^2 \det\left((xI_4 - A)(1,3|2,4)\right) \\
 = x^4 + a_3x^3 + (a_2 + \lambda (a'_2 + a''_2))x^2 + (a_1 + \lambda (a'_1 + a''_1))x  + a''_0 + \lambda(a_0 + a'_0) + \lambda^2 a'''\\
 = x^4 + Q_3(A)x^3 + Q_2(A)x^2 + Q_1(A)x + Q_0, 
\end{multline}
where $Q_0, \ldots, Q_3, \in \F[\lambda, a_{1\,1}, \ldots, a_{1\, 4}, a_{2\,1}, \ldots, a_{4\,4}]$ are defined by
\begin{multline}\label{lem:DistinctN4CHN2RK2:eq3}
Q_3 = a_3, \quad Q_2(A) = a_2 + \lambda (a'_2 + a''_2), \quad Q_1 = (a_1 + \lambda (a'_1 + a''_1))\\
Q_0 =  a''_0 + \lambda(a_0 + a'_0) + \lambda^2 a'''.
\end{multline}
The equality~\eqref{lem:DistinctN4CHN2RK2:eq3} implies that
\begin{equation}\label{lem:DistinctN4CHN2RK2:eq2}
\deg_{\lambda}(Q_3) = 0, \quad \deg_{\lambda}(Q_2) = 1, \quad \deg_{\lambda}(Q_1) = 1, \quad \deg_{\lambda}(Q_0) = 2.
\end{equation}
Using the formula~\eqref{lem:DiscN4FormulaCHN2:eqq1} we obtain that
\begin{multline}\label{lem:DistinctN4CHN2RK2:eq1AAA1}
\D_{(B)}(A) = \D(p_{A + \lambda B}(x)) = 256 Q_0^3 - 27 Q_1^4 + 144 Q_0 Q_1^2 Q_2 - 128 Q_0^2 Q_2^2\\
 - 4 Q_1^2 Q_2^3 + 16 Q_0 Q_2^4 - 192 Q_0^2 Q_1 Q_3 + 18 Q_1^3 Q_2 Q_3\\
  - 80 Q_0 Q_1 Q_2^2 Q_3  - 6 Q_0 Q_1^2 Q_3^2 + 144 Q_0^2 Q_2 Q_3^2\\
  + Q_1^2 Q_2^2 Q_3^2 - 4 Q_0 Q_2^3 Q_3^2 - 4 Q_1^3 Q_3^3\\
   + 18 Q_0 Q_1 Q_2 Q_3^3  - 27 Q_0^2 Q_3^4.
\end{multline}

Let us find the maximal degree of $\lambda$ contained in each term of~\eqref{lem:DistinctN4CHN2RK2:eq1AAA1} using~\eqref{lem:DistinctN4CHN2RK2:eq2}. We obtain that
\begin{multline}\label{lem:DistinctN4CHN2RK2:eq11001}
\deg_{\lambda}(Q_0^3) = 6,\qquad \deg_{\lambda}(Q_1^4) = 4,\qquad \deg_{\lambda}(Q_0 Q_1^2 Q_2) = 5,\\
\deg_{\lambda}(Q_0^2 Q_2^2) = 6, \qquad \deg_{\lambda}(Q_1^2 Q_2^3) = 5,\qquad \deg_{\lambda}(Q_0 Q_2^4) = 6,\\
\deg_{\lambda}(Q_0^2 Q_1 Q_3) = 5 ,\qquad \deg_{\lambda}(Q_1^3 Q_2 Q_3) = 4, \qquad \deg_{\lambda}(Q_0 Q_1 Q_2^2 Q_3) = 5,\\
\deg_{\lambda}(Q_0 Q_1^2 Q_3^2) = 4,\qquad \deg_{\lambda}( Q_0^2 Q_2 Q_3^2) = 5,\qquad \deg_{\lambda}(Q_1^2 Q_2^2 Q_3^2) = 4,\\
\deg_{\lambda}(Q_0 Q_2^3 Q_3^2) = 5,\qquad \deg_{\lambda}(Q_1^3 Q_3^3) = 4,\qquad \deg_{\lambda}(Q_0 Q_1 Q_2 Q_3^3) = 5,\\
\deg_{\lambda}(Q_0^2 Q_3^4) = 4.
\end{multline}
By inspecting~\eqref{lem:DistinctN4CHN2RK2:eq11001}, we conclude that $\D_{(B)}^{(6)} = P$  is equal to the coefficient of $\lambda^6$ in
\begin{equation*}
256 Q_0^3 - 128 Q_0^2 Q_2^2 + 16 Q_0 Q_2^4.
\end{equation*}
According to~\eqref{lem:DistinctN4CHN2RK2:eq3}, 
\begin{gather*}
\mbox{the coefficient of $\lambda^6$ in $Q_0^3$ is equal to}\;\; \left(a'''\right)^3 ,\\
\mbox{the coefficient of $\lambda^6$ in $Q_0^2 Q_2^2$ is equal to}\;\; \left(a'''\right)^3(a'_2 + a''_2)^2,\\
\mbox{the coefficient of $\lambda^6$ in $ Q_0 Q_2^4$ is equal to}\;\; \left(a'''\right)(a'_2 + a''_2)^4.
\end{gather*}
Therefore,
\begin{multline}\label{lem:DistinctN4CHN2RK2:eq1001}
P = 256\left(a'''\right)^3 - 128\left(a'''\right)^3(a'_2 + a''_2)^2 + 16a'''(a'_2 + a''_2)^4\\
= 16a'''(16\left(a'''\right)^2 - 8\left(a'''\right)^2(a'_2 + a''_2)^2 + (a'_2 + a''_2)^4)\\
= 16a'''(4a''' - (a'_2 + a''_2)^2)^2.
\end{multline}
Using the definitions of $a''', a'_2$ and $a''_2$ we obtain that
\begin{equation*}
a'_2 \overset{\eqref{lem:DistinctN4CHN2RK2:eq102}}{=\joinrel=\joinrel=} -a_{2\,1}, \qquad a''_2 \overset{\eqref{lem:DistinctN4CHN2RK2:eq103}}{=\joinrel=\joinrel=} -a_{4\,3}\\
\qquad \mbox{and}\qquad a''' \overset{\eqref{lem:DistinctN4CHN2RK2:eq104}}{=\joinrel=\joinrel=} \begin{vmatrix}a_{2\,1} & a_{2\,3}\\ a_{4\,1} & a_{4\,3}\end{vmatrix}.
\end{equation*}
Thus,
\begin{multline*}
P \overset{\eqref{lem:DistinctN4CHN2RK2:eq1001}}{=\joinrel=\joinrel=} 16a'''(4a''' - (a'_2 + a''_2)^2)^2\\
= 16(a_{2\,1}a_{4\,3} - a_{2\,3}a_{4\,1})(4a_{2\,1}a_{4\,3} - 4a_{2\,3}a_{4\,1} - (a_{2\,1} + a_{4\,3})^2)^2.
\end{multline*}
This formula implies that $P$ does not depend on $a_{1\,1}, \ldots, a_{1\,4}$, $a_{2\,2},$ $a_{2\,4},$\linebreak $a_{3\,1}, \ldots, a_{3\,4}, a_{4\,2}$ and $a_{4\,4}$ and consequently $E_{1\,1}, \ldots, E_{1\,4},$ $E_{2\,2}, \ldots, E_{2\,4},$\linebreak $E_{3\,1}, \ldots, E_{3\,4}, E_{4\,2}$ and $E_{4\,4}$ belong to $\left(\FL_P\right)^0$. Hence, $\dim \left(\left(\FL_P\right)^0\right) \ge 12$
and  the basic properties of the vector spaces and their annihilators infer that
\begin{equation*}
\dim(\mathcal L_P) = 16 - \dim \left(\left(\FL_P\right)^0\right) \le 16 - 12 = 4.
\end{equation*}

\end{proof}

\begin{lemma}\label{lem:DistinctN4CH2RK1NNIL}Assume that $\Char (\F) = 2.$ If $B = \begin{psmallmatrix}1 & 0 & 0 & 0\\ 0 & 0 & 0 & 0\\ 0 & 0 & 0 & 0\\ 0 & 0 & 0 & 0\end{psmallmatrix}$, then\linebreak $\deg_\lambda \left(\B_{(B)}\right) = 3$.
\end{lemma}

\begin{proof}Let $p_{A}(x) = x^4 + a_3x^3 + a_2x^2 + a_1x + a_0$, $p_{A(1|1)}(x) = x^3 + a'_2x^2 + a'_1x + a'_0$. Using the multilinearity of the determinant along the first column we obtain
\begin{multline*}
p_{A + \lambda B}(x)\\ = \begin{vmatrix}x - a_{1\,1} - \lambda & -a_{1\,2} & -a_{1\,3} & - a_{1\,4}\\ -a_{2\,1} & x - a_{2\,2} & -a_{2\,3} & -a_{2\,4}\\ -a_{3\,1} & -a_{3\,2} & x - a_{3\,3} & -a_{3\,4}\\ - a_{4\,1} & -a_{4\,2} & -a_{4\,3} & x - a_{4\,4} \end{vmatrix}
 = p_{A}(x) - \lambda  p_{A(1|1)}(x)\\
 = x^4 + (a_3 - \lambda)x^3 + (a_2 - \lambda a'_2)x^2 + (a_1 - \lambda a'_1)x + (a_0 - \lambda a'_0). 
\end{multline*}
Hence, 
\begin{multline}\label{lem:DistinctN4CH2RK1NNIL:eq10001}
\B_{(B)} = \B(p_{A + \lambda X}(t))\\
\overset{\eqref{lem:Disc2N4FormulaCH2:eqq1}}{=\joinrel=\joinrel=} (a_1 - \lambda a'_1)^2 + (a_1 - \lambda a'_1)(a_2 - \lambda a'_2) (a_3 - \lambda) + (a_0 - \lambda a'_0)(a_3 - \lambda)^2.
\end{multline}
By inspecting the right-hand side of~\eqref{lem:DistinctN4CH2RK1NNIL:eq10001} we obtain that $\deg_\lambda \left(\B_{(B)}\right) \le 3$ and 
\begin{multline}\label{lem:DistinctN4CH2RK1NNIL:eq10002}
\B_{(B)}^{(3)}(A) = a'_1a'_2 + a'_0\\
 = (\mbox{the sum of all principal $2\times 2$ minors of $A(1|1)$})\tr (A(1|1))\\
  + \det(A(1|1))
\end{multline}
Hence, if $A_0 = E_{2\,3} + E_{3\,4} + E_{4\,2}$, then
\begin{equation}
\B_{(B)}^{(3)}\left(A_0\right) = (\mbox{the sum of all principal $2\times 2$ minors of $A(1|1)$}) \cdot 0 + 1 = 1, 
\end{equation}
which infers that $\B_{(B)}^{(3)}$ is indeed nonzero polynomial and $\deg_\lambda \left(\B_{(B)}\right) = 3.$

\end{proof}

\begin{lemma}\label{lem:DistinctN4CH2RK2}Assume that $\Char (\F) = 2.$ If $B = \begin{psmallmatrix}0 & 1 & 0 & 0\\ 0 & 0 & 0 & 0\\ 0 & 0 & 0 & 1\\ 0 & 0 & 0 & 0\end{psmallmatrix}$, then\linebreak $0 < \deg_\lambda \left(\B_{(B)}\right) \le 2$.
\end{lemma}
\begin{proof}From Lemma~\ref{lem:E12PlusStrictUpDegBound} we obtain that $0 < \deg_\lambda \left(\D_{(B)}\right).$ Therefore, $0 < \deg_\lambda \left(\B_{(B)}\right)$ because $\D = \B^2$ by the definition of $\B$. Let us show that\linebreak $\deg_\lambda \left(\B_{(B)}\right) \le 2$.

For this, let $p_{A}(x), p'(x), p''(x), a''' \in \F[a_{1\,1}, \ldots, a_{1\,4}, a_{2\,1}, \ldots, a_{4\,4}]$ be defined by
\begin{equation}\label{lem:DistinctN4CH2RK2:eq101}
p_{A}(x) = \det (xI_4 - A),
\end{equation}
\begin{equation}\label{lem:DistinctN4CH2RK2:eq102}
p'(x) = \det \left((xI_4 - A)(1|2)\right),
\end{equation}
\begin{equation}\label{lem:DistinctN4CH2RK2:eq103}
p''(x) = \det \left((xI_4 - A)(3|4)\right)
\end{equation}
and
\begin{equation}\label{lem:DistinctN4CH2RK2:eq104}
a''' = \det((xI_4 - A)(1,3|2,4)) = \begin{vmatrix}-a_{2\,1} & - a_{2\,3}\\ -a_{4\,1} & - a_{4\,3}\end{vmatrix} \in \F.
\end{equation}
Note that the definitions of $p'$ and $p''$ imply that $\deg_t(p') = \deg_t(p'') = 2$. Hence, there is $a'_0, \ldots, a'_2, a''_0, \ldots, a''_2 \in \F$ such that
\begin{equation}\label{lem:DistinctN4CH2RK2:eq105}
p'(x)  = a'_2x^2 + a'_1x + a'_0\;\;\mbox{and}\;\;p''(x)  = a''_2x^2 + a''_1x + a''_0.
\end{equation}

Using the multilinearity of the determinant along the second and the fourth column
\begin{multline}\label{lem:DistinctN4CH2RK2:eq1}
p_{A + \lambda B}(x) = \begin{vmatrix}x - a_{1\,1}  & -a_{1\,2}- \lambda & -a_{1\,3} & - a_{1\,4}\\ -a_{2\,1} & x - a_{2\,2} & -a_{2\,3} & -a_{2\,4}\\ -a_{3\,1} & -a_{3\,2} & x - a_{3\,3} & -a_{3\,4} - \lambda  \\ - a_{4\,1} & -a_{4\,2} & -a_{4\,3} & x - a_{4\,4} \end{vmatrix}\\
 = p_{A}(x) + \lambda (p'(x) + p''(x)) + \lambda^2 \det\left((xI_4 - A)(1,3|2,4)\right) \\
 = x^4 + a_3x^3 + (a_2 + \lambda (a'_2 + a''_2))x^2 + (a_1 + \lambda (a'_1 + a''_1))x  + a''_0 + \lambda(a_0 + a'_0) + \lambda^2 a'''\\
 = x^4 + Q_3(A)x^3 + Q_2(A)x^2 + Q_1(A)x + Q_0, 
\end{multline}
where $Q_0, \ldots, Q_3, \in \F[\lambda, a_{1\,1}, \ldots, a_{1\, 4}, a_{2\,1}, \ldots, a_{4\,4}]$ are defined by
\begin{multline}\label{lem:DistinctN4CH2RK2:eq3}
Q_3 = a_3, \quad Q_2(A) = a_2 + \lambda (a'_2 + a''_2), \quad Q_1 = (a_1 + \lambda (a'_1 + a''_1))\\
Q_0 =  a''_0 + \lambda(a_0 + a'_0) + \lambda^2 a'''.
\end{multline}
The equalities in~\eqref{lem:DistinctN4CH2RK2:eq3} imply that
\begin{equation}\label{lem:DistinctN4CH2RK2:eq2}
\deg_{\lambda}(Q_3) = 0, \quad \deg_{\lambda}(Q_2) = 1, \quad \deg_{\lambda}(Q_1) = 1, \quad \deg_{\lambda}(Q_0) = 2.
\end{equation}

Using the formula~\eqref{lem:Disc2N4FormulaCH2:eqq1} we obtain that
\begin{equation}\label{lem:DistinctN4CH2RK2:eq111}
\B_{(B)} = \B(p_{A + \lambda B}(t)) = Q_1^2  + Q_1 Q_2 Q_3 + Q_0 Q_3^2
\end{equation}
According to~\eqref{lem:DistinctN4CH2RK2:eq2},
\begin{equation*}
\deg_{\lambda}(Q_1^2) = 2,\qquad \deg_{\lambda}(Q_1 Q_2 Q_3) = 2\qquad\mbox{and}\;\; \deg_{\lambda}(Q_0 Q_3^2) = 2.
\end{equation*}
Therefore, from~\eqref{lem:DistinctN4CH2RK2:eq111} we conclude that $\deg_\lambda \left(\B_{(B)}\right) \le 2$.

\end{proof}

\begin{lemma}\label{lem:DistinctN5CH2RK1}Assume that $\Char (\F) = 2$ and $|\F| > \frac{5 \cdot 4}{2} = 10$. If\linebreak $B = \begin{psmallmatrix}1 & 0 & 0 & 0 & 0\\ 0 & 0 & 0 & 0 & 0\\ 0 & 0 & 0 & 0 & 0\\ 0 & 0 & 0 & 0 & 0\\ 0 & 0 & 0 & 0 & 0\end{psmallmatrix}$, then $\deg_{\lambda} (\B_{(B)} = 4$ and  $\dim\left(\FL_{\B_{(B)}^{(4)}}\right) \ge 15$.
\end{lemma}
\begin{proof}Let $p_{A}(x) = x^5 + a_4x^4 + a_3x^3 + a_2x^2 + a_1x + a_0$ and
\begin{equation}\label{lem:DistinctN5CH2RK1:eq1}
p_{A(1|1)}(x) = x^4 + a'_3x^3 + a'_2x^2 + a'_1x + a'_0.
\end{equation}

Using the multilinearity of the determinant along the first column we obtain
\begin{multline}\label{lem:DistinctN5CH2RK1:eq11}
p_{A + \lambda B}(x) 
 = p_{A}(x) - \lambda  p_{A(1|1)}(x)\\
 = x^5 + Q_4(A)x^4 + Q_3(A)x^3 + Q_2(A)x^2 + Q_1(A)x + Q_0, 
\end{multline}
where $Q_0, \ldots, Q_4, \in \F[\lambda, a_{1\,1}, \ldots, a_{1\, 5}, a_{2\,1}, \ldots, a_{5\,5}]$ are defined by
\begin{multline}\label{lem:DistinctN5CH2RK1:eq3}
Q_4 = a_4 + \lambda, \quad Q_3 = a_3 + \lambda a'_3 \quad Q_2 = a_2 + \lambda a'_2 , \quad Q_1 = a_1 + \lambda a'_1\\
Q_0 =  a_0 + \lambda  a'_0.
\end{multline}
The equalities in~\eqref{lem:DistinctN5CH2RK1:eq3} imply that
\begin{equation}\label{lem:DistinctN5CH2RK1:eq2}
\deg_{\lambda}(Q_4) = \deg_{\lambda}(Q_3) = \deg_{\lambda}(Q_2) =  \deg_{\lambda}(Q_1) =  \deg_{\lambda}(Q_0) = \deg_{\lambda}(Q_3) = 1.
\end{equation}
Using the formula~\eqref{lem:Disc2N5FormulaCH2:eqq1} we obtain that
\begin{multline}\label{lem:DistinctN5CH2RK1:eq111}
\B_{(B)} = \B(p_{A + \lambda B}(x))\\
 = Q_0^2 + Q_1 Q_2^2 + Q_0 Q_2 Q_3 + Q_1 Q_2 Q_3 Q_4 + Q_0 Q_3^2 Q_4 + Q_1^2 Q_4^2.
\end{multline}
Let us find the maximal degree of $\lambda$ contained in each term of~\eqref{lem:DistinctN5CH2RK1:eq111}. The equalities in~\eqref{lem:DistinctN5CH2RK1:eq2} infer that
\begin{multline*}
\deg_{\lambda}(Q_0^2) = 2,\qquad \deg_{\lambda}( Q_1 Q_2^2) = 3,\qquad \deg_{\lambda}(Q_0 Q_2 Q_3) = 3,\\
\deg_{\lambda}(Q_1 Q_2 Q_3 Q_4) = 4,\qquad \deg_{\lambda}( Q_0 Q_3^2 Q_4) = 4,\qquad \deg_{\lambda}(Q_1^2 Q_4^2) = 4.
\end{multline*}
This implies that 
\begin{equation}\label{lem:DistinctN5CH2RK1:eq60001}
\deg_{\lambda} \left(\B_{(B)}\right) \le 4.
\end{equation}
In addition, $\B_{(B)}^{(4)} = P$  is equal to the sum of the coefficients of $\lambda^4$ in the terms $Q_1 Q_2 Q_3 Q_4$, $Q_0 Q_3^2 Q_4$ and $Q_1^2 Q_4^2$ on the right-hand side of~\eqref{lem:DistinctN5CH2RK1:eq111}. According to~\eqref{lem:DistinctN5CH2RK1:eq3}, 
\begin{gather*}
\mbox{the coefficient of $\lambda^4$ in $Q_1 Q_2 Q_3 Q_4$ is equal to}\;\; a'_1a'_2a'_3 ,\\ 
\mbox{the coefficient of $\lambda^4$ in $Q_0 Q_3^2 Q_4$ is equal to}\;\; a'_0(a'_3)^2,\\
\mbox{the coefficient of $\lambda^4$ in $Q_1^2 Q_4^2$ is equal to}\;\; (a'_1)^2.
\end{gather*}
Hence,
\begin{equation}\label{lem:DistinctN5CH2RK1:eq1666}
P = a'_1a'_2a'_3 + a'_0(a'_3)^2 + (a'_1)^2. 
\end{equation}
From~\eqref{lem:DistinctN5CH2RK1:eq1} and~\eqref{eqqq:formulaChPoly} we obtain the following expressions for $a'_0, a'_1, a'_2$ and $a'_3$:
\begin{gather}\label{lem:DistinctN5CH2RK1:eq7}
a'_0 = \det(A(1|1)), \quad a'_1 = \sum_{1 \le i_1 < i_2 < i_3 \le 4} \det\left((A(1|1))[i_1, i_2, i_3|i_1, i_2, i_3]\right)\\
a'_2 = \sum_{1 \le i_1 < i_2 \le 4} \det\left((A(1|1))[i_1, i_2|i_1, i_2]\right), \qquad a'_3 = \tr(A(1|1)).
\end{gather}

Let $A_1 = \begin{psmallmatrix}0 & 0 & 0 & 0 & 0\\ 0 & 0 & 0 & 0 & 0\\ 0 & 0 & 1 & 0 & 0\\ 0 & 0 & 1 & 0 & 0\\ 0& 0 & 0 & 0 & 1\end{psmallmatrix}$. If $Z \in \M_5(\F)$, then  $l_{P, A_1}(Z)$ is equal to the coefficient of $\kappa$ in the expression $P(A_1 + \kappa Z) \overset{\eqref{lem:DistinctN5CH2RK1:eq1666}}{=\joinrel=\joinrel=}  a'_1a'_2a'_3 + a'_0(a'_3)^2 + (a'_1)^2,$ where $a'_0, a'_1, a'_2$ could be found by~\eqref{lem:DistinctN5CH2RK1:eq7}. We have the following equalities
\begin{equation}\label{lem:DistinctN5CH2RK1:eq8}
a'_0 = \det((A_1 + \kappa X)(1|1)) = \kappa \cdot z_{2\,3}  + \mbox{(terms of higher degree in $\kappa$)},
\end{equation}
\begin{multline}\label{lem:DistinctN5CH2RK1:eq9}
a'_1 = \sum_{1 \le i_1 < i_2 < i_3 \le 4} \det\left(((A_1 + \kappa Z)(1|1))[i_1, i_2, i_3|i_1, i_2, i_3]\right)\\
= \kappa (z_{2\,3} + z_{2\,4} + z_{4\,3}) + \mbox{(terms of higher degree in $\kappa$)},
\end{multline}
\begin{multline}\label{lem:DistinctN5CH2RK1:eq10}
a'_2 = \sum_{1 \le i_1 < i_2 \le 4} \det\left(((A_1 + \kappa Z)(1|1))[i_1, i_2|i_1, i_2]\right)\\
=\kappa (z_{2\,2} + z_{2\,4} + z_{3\,3} + z_{4\,3} + z_{4\,4}) + \mbox{(terms of higher degree in $\kappa$)},
\end{multline}
and
\begin{multline}\label{lem:DistinctN5CH2RK1:eq11AA11}
a'_3 = \tr((A_1 + \kappa Z)(1|1))\\
= 1 + \kappa (z_{2\,2} + z_{3\,3} +  z_{4\,4} + z_{5\,5}) + \mbox{(terms of higher degree in $\kappa$)}.
\end{multline}
By substituting~\eqref{lem:DistinctN5CH2RK1:eq8}--\eqref{lem:DistinctN5CH2RK1:eq11AA11} into~\eqref{lem:DistinctN5CH2RK1:eq1666} we obtain that

\begin{equation*}
P(A + \kappa X) = 0 +  \kappa z_{2\,3} \cdot 1^2 + 0  + \mbox{(terms of higher degree in $\kappa$)}.
\end{equation*}
Hence,
\begin{equation}\label{lem:DistinctN5CH2RK1:eq1111}
l_{P, A_{1}}(Z) = z_{2\,3}\;\;\mbox{for all}\;\;Z \in \M_5(\F).
\end{equation}
Let $\sigma_1 = \id, \sigma_2, \ldots, \sigma_{12} \in S_5$ be permutations such that $\sigma_{i}(1) = 1$ for all $1 \le i \le $ and $$(\sigma_1(2),\sigma_1(3)) = (2,3),\; (\sigma_2(2),\sigma_2(3)) = (2, 4),\; \ldots,\; (\sigma_{12}(2),\sigma_{12}(3)) = (5,4),$$
 correspondingly. In other words, $(\sigma_i(2),\sigma_i(3))$ ranges over all possible ordered pairs of distinct elements of $\{2, 3, 4, 5\}$ when $i$ ranges from $1$ to $12$. Since $S_5$ contains all the bijections of $\{1, \ldots, 5\}$, such permutations indeed exist. 
 
Let $T^{(2)},\ldots, T^{(12)}$ be bijective linear maps on $\M_{5}(\F)$ defined by
\[
T^{(i)}(Y) = (y_{\sigma_k(i)\,\sigma_k(j)})\;\;\mbox{for all}\;\;Y = (y_{i\,j}) \in \M_{5}(\F)\;\;\mbox{and}\;\;1 \le k \le 12.
\]
As in Lemma~\ref{lem:DistinctN4CHN2RK1NNIL}, these linear maps preserve $\D$ and send $B$ to $B$. \linebreak By Lemma~\ref{lem:PresDEqPresD2}, they preserve $\B$ as well. Hence, by applying Lemma~\ref{lem:TPreservesPThenTPreservesPI} to these maps and $B$, we obtain that $T^{(1)},\ldots, T^{(12)}$  preserve $P = \B_{(B)}^{(4)}.$ Using this fact and Lemma~\ref{thm:surjpres} applied to $P, l_{P, A_{1}}$ and $T_i$ we conclude that the linear functionals
\[
l^{(1)} = l_{P, A_{1}} \circ T^{(1)} = l_{P, A_{1}},\quad l^{(2)} = l_{P, A_{1}} \circ T^{(2)}, \quad \ldots\;,\quad l^{(12)} = l_{P, A_{1}} \circ T^{(12)}
\]
belong to $\FL_P$. In addition,
\begin{equation}\label{lem:DistinctN5CH2RK1:eq100001}
\begin{gathered}
l^{(1)}(Z) = \left(T^{(1)}(Z)\right)_{2\,3} = z_{\sigma_1(2)\,\sigma_1(3)} = z_{2\,3},\\
l^{(2)}(Z) = \left(T^{(2)}(Z)\right)_{2\,3} = z_{\sigma_2(2)\,\sigma_2(3)} = z_{2\,4},\\
\vdots\\
l^{(12)}(Z) = \left(T^{(12)}(Z)\right)_{2\,3} = z_{\sigma_{12}(2)\,\sigma_{12}(3)} = z_{4\,5},
\end{gathered}
\end{equation}
for all $Z \in M_{5}(\F)$.

Now let $M' = I_5 + E_{2\,3}$ and $T'$ be a bijective linear map on $\M_{5}(\F)$ defined by
\[
T'(Z) = M'ZM'\;\;\mbox{for all}\;\;Z \in \M_5(\F).
\]
Since $M'M' = I_5$, then $T'$ is a matrix similarity. Hence, $T'$ preserves $\D$. In addition, the direct calculation shows that $T'(B) = B$.  Therefore, $T'$  preserves $P = \B_{(B)}^{(4)}$ for the same reason as for $T^{(1)}, \ldots, T^{(12)}$. Similarly as above we obtain that  $l' = l_{P, A_{1}} \circ T' \in \FL_P$ and
\begin{equation}\label{lem:DistinctN5CH2RK1:eq100002}
l'(Z) = \left(T'(Z)\right)_{2\,3} = z_{2\,2} + z_{2\,3} + z_{3\,2} + z_{3\,3}\;\;\mbox{for all}\;\;Z \in \M_{5}(\F).
\end{equation}

Finally, let $l'' = l' \circ T^{(2)}$ and $l''' = l' \circ T^{(3)}$. As above, we conclude that $l''$ and $l'''$ belong to $\FL_P$. Also the equality
\begin{multline}\label{lem:DistinctN5CH2RK1:eq100003}
l''(Z) = \left(T^{(2)}(Z)\right)_{2\,2} + \left(T^{(2)}(Z)\right)_{2\,3} + \left(T^{(2)}(Z)\right)_{3\,2} + \left(T^{(2)}(Z)\right)_{3\,3}\\
= z_{\sigma_2(2)\,\sigma_2(2)} + z_{\sigma_2(2)\,\sigma_2(3)} + z_{\sigma_2(3)\,\sigma_2(2)} + z_{\sigma_2(3)\,\sigma_2(3)}\\
= z_{2\,2} + z_{2\,4} + z_{4\,2} + z_{4\,4}
\end{multline}
and the equality
\begin{multline}\label{lem:DistinctN5CH2RK1:eq100004}
l'(Z) = \left(T^{(3)}(Z)\right)_{2\,2} + \left(T^{(3)}(Z)\right)_{2\,3} + \left(T^{(3)}(Z)\right)_{3\,2} + \left(T^{(2)}(Z)\right)_{3\,3}\\
= z_{\sigma_3(2)\,\sigma_3(2)} + z_{\sigma_3(2)\,\sigma_3(3)} + z_{\sigma_3(3)\,\sigma_3(2)} + z_{\sigma_3(3)\,\sigma_3(3)}\\
= z_{2\,2} + z_{2\,5} + z_{5\,2} + z_{5\,5}
\end{multline}
hold for all $Z \in \M_{5}(\F).$ By considering the equalities in~\eqref{lem:DistinctN5CH2RK1:eq100001},~\eqref{lem:DistinctN5CH2RK1:eq100002} and~\eqref{lem:DistinctN5CH2RK1:eq100003} we conclude that $l', l'', l''', l^{(1)}, \ldots, l^{(12)}$ are linearly independent. Therefore, $\dim(\mathcal L_P) \ge 15.$ This also implies that $P = \B_{(B)}^{(4)}$ is nonzero and consequently $\deg_{\lambda} \left(\B_{(B)}\right) \ge 4$. Together with~\eqref{lem:DistinctN5CH2RK1:eq60001} this implies that $\deg_{\lambda} \left(\B_{(B)}\right) = 4$.

\end{proof}

\begin{lemma}\label{lem:DistinctN5CH2RK2}Assume that $\Char (\F) = 2$ and $|\F| > \frac{5 \cdot 4}{2} = 10$. \linebreak If $B = \begin{psmallmatrix}0 & 1 & 0 & 0 & 0\\ 0 & 0 & 0 & 0 & 0\\ 0 & 0 & 0 & 1 & 0\\ 0 & 0 & 0 & 0 & 0\\ 0 & 0 & 0 & 0 & 0\end{psmallmatrix}$, then $0 < \deg_\lambda \left(\B_{(B)}\right) \le 4$ and  $\dim(\mathcal L_{\B_{(B)}^{(4)}}) \le 12$.
\end{lemma}
\begin{proof}From Lemma~\ref{lem:E12PlusStrictUpDegBound} we obtain that $0 < \deg_\lambda \left(\D_{(B)}\right).$ Therefore, $0 < \deg_\lambda \left(\B_{(B)}\right)$ because $\D = \B^2$ by the definition of $\B$. Let us show that\linebreak $\deg_\lambda \left(\B_{(B)}\right) \le 4$ and $\dim(\mathcal L_{\B_{(B)}^{(4)}}) \le 12$.

Let $p'(x), p''(x), p''' \in \F[a_{1\,1}, \ldots, a_{1\,5}, a_{2\,1}, \ldots, a_{5\,5}]$ be defined by

\begin{equation}\label{lem:DistinctN5CH2RK2:eq102}
p'(x) = p_{A(1|2)}(x) = a'_3x^3+ a'_2x^2 + a'_1x + a'_0,
\end{equation}
\begin{equation}\label{lem:DistinctN5CH2RK2:eq103}
p''(x) =  p_{A(3|4)}(x) =  a''_3x^3 + a''_2x^2 + a''_1x + a''_0
\end{equation}
and
\begin{equation}\label{lem:DistinctN5CH2RK2:eq104}
p'''(x) = p_{A(1,3|2,4)}(x) = a'''_1x + a'''_0. 
\end{equation}

Using the multilinearity of the determinant along the second and the fourth column we obtain that
\begin{multline}\label{lem:DistinctN5CH2RK2:eq1}
p_{A + \lambda B}(x)
 = p_{A}(x) + \lambda (p'(x) + p''(x)) + \lambda^2 p'''(x) \\ 
 = x^5 + Q_4x^4 + Q_3x^3 + Q_2x^2 + Q_1x + Q_0, 
\end{multline}
where $Q_0, \ldots, Q_3, \in \F[\lambda, a_{1\,1}, \ldots, a_{1\, 5}, a_{2\,1}, \ldots, a_{5\,5}]$ are defined by
\begin{multline}\label{lem:DistinctN5CH2RK2:eq3}
Q_4 = \tr(A), \quad Q_3 = a_3 + \lambda (a'_3 + a''_3), \quad Q_2 = (a_2 + \lambda (a'_2 + a''_2))\\
Q_1 =  a_0 + \lambda(a'_1 + a''_1) + \lambda^2 a'''_0, \qquad Q_0 =  a_0 + \lambda(a'_0 + a''_0) + \lambda^2 a'''_0.
\end{multline}
The equalities~\eqref{lem:DistinctN5CH2RK2:eq3} imply that
\begin{equation}\label{lem:DistinctN5CH2RK2:eq2}
\deg_{\lambda}(Q_4) = 0, \quad \deg_{\lambda}(Q_3) =  \deg_{\lambda}(Q_2) = 1, \quad \deg_{\lambda}(Q_1)  = \deg_{\lambda}(Q_0) = 2.
\end{equation}
Using the formula~\eqref{lem:Disc2N5FormulaCH2:eqq1} we obtain that
\begin{multline}\label{lem:DistinctN5CH2RK2:eq111}
\B_{(B)} = \B(p_{A + \lambda B}(t))\\
 = Q_0^2 + Q_1 Q_2^2 + Q_0 Q_2 Q_3 + Q_1 Q_2 Q_3 Q_4 + Q_0 Q_3^2 Q_4 + Q_1^2 Q_4^2.
\end{multline}
From~\eqref{lem:DistinctN5CH2RK2:eq2} we conclude that
\begin{multline*}
\deg_{\lambda}(Q_0^2) = 4,\qquad \deg_{\lambda}( Q_1 Q_2^2) = 4,\qquad \deg_{\lambda}(Q_0 Q_2 Q_3) = 4,\\
\deg_{\lambda}(Q_1 Q_2 Q_3 Q_4) = 4,\qquad \deg_{\lambda}( Q_0 Q_3^2 Q_4) = 4,\qquad \deg_{\lambda}(Q_1^2 Q_4^2) = 4.
\end{multline*}
This implies that $\deg_\lambda \left(\B_{(B)}\right) \le 4$ and $\B_{(B)}^{(4)} = P$  is equal to the sum of the coefficients of $\lambda^4$ in all the terms  in the right-hand side of~\eqref{lem:DistinctN5CH2RK1:eq111}. According to~\eqref{lem:DistinctN5CH2RK2:eq3}, 
\begin{gather*}
\mbox{the coefficient of $\lambda^4$ in $Q_0^2$ is equal to}\;\; (a'''_0)^2 ,\\ 
\mbox{the coefficient of $\lambda^4$ in $Q_1 Q_2^2$ is equal to}\;\;(a'''_1)(a'_2 + a''_2)^2 ,\\ 
\mbox{the coefficient of $\lambda^4$ in $Q_0 Q_2 Q_3$ is equal to}\;\; a'''_0(a'_2 + a''_2)(a'_3 + a''_3) ,\\ 
\mbox{the coefficient of $\lambda^4$ in $Q_1 Q_2 Q_3 Q_4$ is equal to}\;\; a'''_1(a'_2 + a''_2)(a'_3 + a''_3)\tr(A) ,\\ 
\mbox{the coefficient of $\lambda^4$ in $Q_0 Q_3^2 Q_4$ is equal to}\;\; a'''_0(a'_3 + a''_3)^2\tr(A),\\
\mbox{the coefficient of $\lambda^4$ in $Q_1^2 Q_4^2$ is equal to}\;\; (a'''_1)^2(\tr(A))^2.
\end{gather*}
Hence,
\begin{multline}\label{lem:DistinctN5CH2RK1:eq6}
P = \left[(a'''_0)^2\right] + \left[(a'''_1)(a'_2 + a''_2)^2\right] + \left[a'''_0(a'_2 + a''_2)(a'_3 + a''_3)\right]\\
 + \left[a'''_1(a'_2 + a''_2)(a'_3 + a''_3)\tr(A)\right]  + \left[a'''_0(a'_3 + a''_3)^2\tr(A)\right] + \left[(a'''_1)^2(\tr(A))^2\right]. 
\end{multline}

Let $Z \in \M_{5}(\F)$. Using the chain rule and the fact that $\Char(\F) = 2$ we obtain that
\begin{multline*}
\frac{\partial P}{\partial Z}
= \left[0\right] + \left[\frac{\partial a'''_1}{\partial Z}(a'_2 + a''_2)^2\right]\\
 + \left[\frac{\partial a'''_0}{\partial Z}(a'_2 + a''_2)(a'_3 + a''_3) + \frac{\partial (a'_2 + a''_2)}{\partial Z}a'''_0(a'_3 + a''_3) + \frac{\partial (a'_3 + a''_3) }{\partial Z}a'''_0(a'_2 + a''_2)\right.\\
 \left. + \frac{\partial  a'''_1}{\partial Z}(a'_2 + a''_2)(a'_3 + a''_3)\tr(A)+ \frac{\partial (a'_2 + a''_2)}{\partial Z}a'''_1(a'_3 + a''_3)\tr(A)\right]\\
  + \left[\frac{\partial (a'_3 + a''_3)}{\partial Z}a'''_1(a'_2 + a''_2)\tr(A)+ \frac{\partial a_4}{\partial Z}a'''_1(a'_2 + a''_2)(a'_3 + a''_3)\right]\\
   + \left[\frac{\partial a'''_0}{\partial Z}(a'_3 + a''_3)^2\tr(A)+ \frac{\partial a_4}{\partial Z}a'''_0(a'_3 + a''_3)^2\right] + \left[0\right]\\
   =  \frac{\partial \tr(A)}{\partial Z}\left(a'''_1(a'_2 + a''_2)(a'_3 + a''_3) + a'''_0(a'_3 + a''_3)^2\right)\\
   + \frac{\partial (a'_2 + a''_2)}{\partial Z}\left(a'''_0(a'_3 + a''_3) + a'''_1(a'_3 + a''_3)\tr(A)\right)\\
   + \frac{\partial (a'_3 + a''_3) }{\partial Z}\left(a'''_0(a'_2 + a''_2) + a'''_1(a'_2 + a''_2)\tr(A)\right)\\
   + \frac{\partial a'''_0}{\partial Z}\left((a'_2 + a''_2)(a'_3 + a''_3) +  (a'_3 + a''_3)^2\tr(A)\right)\\
   + \frac{\partial a'''_1}{\partial Z}\left((a'_2 + a''_2)^2 + (a'_2 + a''_2)(a'_3 + a''_3)\tr(A)\right).
\end{multline*}
Using the definitions of $a'''_0, a'''_1, a'_1, a'_2, a''_1, a''_2, \in \F$ in~\eqref{lem:DistinctN5CH2RK2:eq102}--\eqref{lem:DistinctN5CH2RK2:eq104}, from~\eqref{eqqq:formulaChPoly} we deduce that
\begin{multline}\label{lem:DistinctN5CH2RK1:eq6001}
a'_2 + a''_2
 = \left(\det (A[2, 3| 1, 3]) + \det (A[2, 4| 1, 4])
  + \det (A[2, 5| 1, 5])\right)\\
   + \left(\det (A[1, 4| 1, 3]) + \det (A[2, 4| 2, 3]) + \det (A[4, 5| 3, 5])\right),
\end{multline}
\begin{equation}\label{lem:DistinctN5CH2RK1:eq6002}
a'_3 + a''_3 = a_{2\,1} + a_{4\,3},
\end{equation}
\begin{equation}\label{lem:DistinctN5CH2RK1:eq6003}
a'''_0 = \det(A[2, 4, 5| 1, 3, 5]),
\end{equation}
and
\begin{equation}\label{lem:DistinctN5CH2RK1:eq6004}
a'''_1 = \det(A[2, 4| 1, 3]).
\end{equation}
The equalities~\eqref{lem:DistinctN5CH2RK1:eq6001}--\eqref{lem:DistinctN5CH2RK1:eq6004} imply that the variables $a_{1\,2},$ $a_{1\,4},$ $a_{1\,5},$ $a_{3\,2},$ $a_{3\, 4},$ $a_{3\,5},$ $a_{5\,1}$ and $a_{5\,4}$ are not contained in $\tr(A), a'_2 + a''_2, a'_3 + a''_3, a'''_0, a'''_1$. Hence
\begin{equation}\label{lem:DistinctN5CH2RK1:eq6005}
E_{1\,2},  E_{1\,4}, E_{1\,5}, E_{3\,2}, E_{3\, 4}, E_{3\,5}, E_{5\,1}, E_{5\,4} \in \left(\FL_P\right)^0.
\end{equation}

Let $\sigma_1, \sigma_2 \in S_5$ be defined by 
$\sigma_1(1) = 2,\, \sigma_1(2) =1,\, \sigma_1(3) = 4 ,\, \sigma_1(4) = 3,\, \sigma_1(5) = 5$
and
$\sigma_2(1) = 3,\, \sigma_2(2) =4,\, \sigma_2(3) = 1 ,\, \sigma_2(4) = 2,\, \sigma_2(5) = 5,$
$M \in \M_{5}(\F)$ be defined by $M = I_5 + B = 
\begin{psmallmatrix}
1 & 1 & 0 & 0 & 0\\
0 & 1 & 0 & 0 & 0\\
0 & 0 & 1 & 1 & 0\\
0 & 0 & 0 & 1 & 0\\
0 & 0 & 0 & 0 & 1
\end{psmallmatrix},$
and $T^{(1)}, T^{(2)}, T^{(3)}\colon \M_{5}(\F) \to \M_{5}(\F)$ be defined by 
\[
T^{(1)}(Y) = (y_{\sigma_1(i)\,\sigma_1(j)})^t,\quad T^{(2)}(Y) = (y_{\sigma_2(i)\,\sigma_2(j)}),\quad  T^{(3)}(Y) = MYM.
\]

Since $M^2 = I_5$ and the matrix transposition preserves $\D$, then as in\linebreak Lemma~\ref{lem:DistinctN5CH2RK1} we obtain that $T^{(1)}, T^{(2)}$ and $T^{(3)}$ are linear maps preserving $\B$. In addition, the direct calculation shows that $T^{(i)}(B) = B$ for all $1 \le i \le 3$. Hence, by applying Lemma~\ref{lem:TPreservesPThenTPreservesPI} to these maps and $B$, we obtain that $T^{(1)}, T^{(2)}$ and $T^{(3)}$  preserve $P = \B_{(B)}^{(4)}.$ Using this fact and Lemma~\ref{cor:SurjPresComp}
applied to $P, T^{(1)}$ and $E_{5\,1}, E_{1\,5}$ we conclude that 
\begin{equation}\label{lem:DistinctN5CH2RK1:eq6006}
E_{2\,5} = T^{(1)}(E_{5\,1}),\;E_{5\,2} = T^{(1)}(E_{1\,5}) \in \left(\FL_P\right)^0
\end{equation}
Then by the similar reasons we obtain that
\begin{equation}\label{lem:DistinctN5CH2RK1:eq6007}
E_{2\,2} + E_{2\,5} + E_{5\,2} + E_{5\,5} = T^{(2)}(E_{2\,5}) \in \left(\FL_P\right)^0.
\end{equation}
Finally,
\begin{equation}\label{lem:DistinctN5CH2RK1:eq6008}
E_{4\,5} = T^{(3)}(E_{2\,5}),\;E_{5\,3} = T^{(3)}(E_{5\,1}) \in \left(\FL_P\right)^0
\end{equation}
Thus, from~\eqref{lem:DistinctN5CH2RK1:eq6005}--\eqref{lem:DistinctN5CH2RK1:eq6008} we obtain 13 linearly independent elements of $\left(\FL_P\right)^0$. Hence, $\dim \left(\left(\FL_P\right)^0\right) \ge 13$
and  the basic properties of the vector spaces and their annihilators infer that
\begin{equation*}
\dim(\mathcal L_P) = 25 - \dim \left(\left(\FL_P\right)^0\right) \le 25 - 13 = 12.
\end{equation*}

\end{proof}

\begin{landscape}
\begin{table}\label{tbl:compTab}
\caption{Invariants of small matrices discussed in Section~\ref{sec:InvariantsOfSmallMatrices}. The gray color highlights the invariants which are used to distinguish matrices in Lemma~\ref{lem:IfPresDThenPresRk1PSc}}
\vspace{20pt}
\begin{tabular}{||c|c|c|c|c|c||}
\hline
\hline
$n$ & $\Char(\F)$ & $X = $ & $\deg_\lambda \left(\D_{(X)}\right)$ or $\deg_\lambda \left(\B_{(X)}\right)$ & $\dim \left(\mathcal L_{\D_{(X)}^{(i)}}\right)$ or $\dim \left(\mathcal L_{\B_{(X)}^{(i)}}\right)$ & Reference\\
\hline
\hline
\multirow{6}{*}{$n = 3$} & \multirow{2}{*}{$\Char(\F) = 2$}  & $E_{1\,1}$  &  $\deg_\lambda \left(\B_{(X)}\right) = 2$ & \cellcolor{lightgray} $\dim \left(\mathcal L_{\B_{(X)}^{(2)}}\right) = 4$ & Lemma~\ref{lem:DistinctN3CH2RK1NNIL}\\ 
& & $E_{1\,2} + E_{2\,3}$  &  $\deg_\lambda \left(\B_{(X)}\right) = 2$ & \cellcolor{lightgray} $\dim \left(\mathcal L_{\B_{(X)}^{(2)}}\right) = 6$ & Lemma~\ref{lem:DistinctN3CH2RK2}\\

 & \multirow{2}{*}{$\Char(\F) = 3$}  & $E_{1\,1}$  & \cellcolor{lightgray} $\deg_{\lambda} \left(\D_{(X)}\right) = 4$ &  & Lemma~\ref{lem:DistinctN3CH3RK1NNIL}\\ 
& & $E_{1\,2} + E_{2\,3}$  &  \cellcolor{lightgray} $\deg_{\lambda} \left(\D_{(X)}\right) = 3$ & & Lemma~\ref{lem:DistinctN3CH3RK2}\\

 & \multirow{2}{*}{$\Char(\F) \neq 2, 3$}  & $E_{1\,1}$  & $\deg_\lambda \left(\D_{(X)}\right) = 4$ & \cellcolor{lightgray} $\dim \left(\mathcal L_{\left(\D\right)_{(X)}^{(4)}}\right) = 3$ & Lemma~\ref{lem:DistinctN3CHN23RK1NNIL}\\ 
& & $E_{1\,2} + E_{2\,3}$  & $\deg_\lambda \left(\D_{(X)}\right) = 4$ & \cellcolor{lightgray}$\dim \left(\mathcal L_{\left(\D\right)_{(X)}^{(4)}}\right) = 1$  & Lemma~\ref{lem:DistinctN3CHN23RK2}\\

\multirow{4}{*}{$n = 4$} & \multirow{2}{*}{$\Char(\F) \neq 2$}  & $E_{1\,1}$  & $\deg_\lambda \left(\D_{(X)}\right) = 6$ & \cellcolor{lightgray} $\dim \left(\mathcal L_{D_{(X)}^{(6)}}\right) = 8$ & Lemma~\ref{lem:DistinctN4CHN2RK1NNIL}\\ 
& & $E_{1\,2} + E_{3\,4}$  & $\deg_\lambda \left(\D_{(X)}\right) = 6$ & \cellcolor{lightgray} $\dim \left(\mathcal L_{D_{(X)}^{(6)}}\right) \le 4$ & Lemma~\ref{lem:DistinctN4CHN2RK2}\\

 & \multirow{2}{*}{$\Char(\F) = 2$}  & $E_{1\,1}$  & \cellcolor{lightgray} $\deg_\lambda \left(\B_{(X)}\right) = 3$ &  & Lemma~\ref{lem:DistinctN4CH2RK1NNIL}\\ 
& & $E_{1\,2} + E_{3\,4}$  & \cellcolor{lightgray} $0 < \deg_\lambda \left(\B_{(X)}\right) \le 2$  &  & Lemma~\ref{lem:DistinctN4CH2RK2}\\

\multirow{2}{*}{$n = 5$} & \multirow{2}{*}{$\Char(\F) = 2$}  & $E_{1\,1}$  & $\deg_\lambda \left(\B_{(X)}\right) = 4$ & \cellcolor{lightgray} $\dim \left(\mathcal L_{\B_{(X)}^{(4)}}\right) \ge 15 $ & Lemma~\ref{lem:DistinctN5CH2RK1}\\ 
& & $E_{1\,2} + E_{3\,4}$  & $0 < \deg_\lambda \left(\B_{(X)}\right) \le 4$ & \cellcolor{lightgray} $\dim \left(\mathcal L_{\B_{(X)}^{(4)}}\right) \le 12$ & Lemma~\ref{lem:DistinctN5CH2RK2}\\
\hline
\hline
\end{tabular}
\end{table}
\end{landscape}

\section{The radical of $\D$ for $n \ge 3$ and the radical of $\B$ for $3 \le n \le 5$}
\label{sec:RadDAndRadB}

\begin{lemma}\label{lem:LinearAdditionPreserverEqualToScalarAlgClos}Assume that $n \ge 3$ and $\F$ is algebraically closed. Let $B \in \M_{n}(\F)$. If
\begin{equation}\label{lem:LinearAdditionPreserverEqualToScalarAlgClos:eqq1}
\D(A + \lambda B) = \D(A)\;\;\mbox{for all}\;\;\lambda \in \F\;\;\mbox{and}\;\;A \in \M_{n}(\F),
\end{equation}
 then $B \in \langle I_n\rangle$.
\end{lemma}
\begin{proof}The condition~\eqref{lem:LinearAdditionPreserverEqualToScalarAlgClos:eqq1} allows to use Lemma~\ref{DegFLessZero} and conclude that\linebreak  $\deg_{\lambda} (\D(A + \lambda B)) \le 0$ for all $A \in \M_{n}(\F).$ By applying Lemma~\ref{DegFLessZero} again we obtain the inequality
\begin{equation}\label{lem:LinearAdditionPreserverEqualToScalarAlgClos:eq1}
\deg_{\lambda} \left(\D_{(B)}\right) \le 0.
\end{equation}
 Hence, $\deg_{\lambda} \left(\D_{(B)}\right) \le 2n - 2$ and Lemma~\ref{lem:LowDegImpliesLowRankOrSeveralCasesForAlgebraicallyClosed} infers that there exist only four possibilities for $B$
\begin{enumerate}[label=C\arabic*., ref=C\arabic*]
\item\label{lem:LinearAdditionPreserverEqualToScalarAlgClos:case1} $B \in \R_n(\F)$.
\item\label{lem:LinearAdditionPreserverEqualToScalarAlgClos:case2} $n = 3$ and $B$ is similar to $\begin{psmallmatrix}0 & 1 & 0\\0 & 0 & 1\\ 0 & 0 & 0\end{psmallmatrix} + cI_n$ for some $c \in \F$.
\item\label{lem:LinearAdditionPreserverEqualToScalarAlgClos:case3} $n = 4$ and $B$ is similar to $\begin{psmallmatrix}0 & 1 & 0 & 0\\0 & 0 & 0 & 0\\ 0 & 0 & 0 & 1\\ 0 & 0 & 0 & 0\end{psmallmatrix}  + cI_n$ for some $c \in \F$.
    
\item\label{lem:LinearAdditionPreserverEqualToScalarAlgClos:case4} $n = 5$, $\Char(\F) = 2$ and  $B$ is similar to $\begin{psmallmatrix}0 & 1 & 0 & 0 & 0\\0 & 0 & 0 & 0 & 0\\ 0 & 0 & 0 & 1 & 0\\ 0 & 0 & 0 & 0 & 0\\ 0 & 0 & 0 & 0 & 0\end{psmallmatrix}  + cI_n$ for some $c \in \F$.
\end{enumerate}
From Table~\ref{tbl:compTab} we conclude that the cases~\ref{lem:LinearAdditionPreserverEqualToScalarAlgClos:case2}--\ref{lem:LinearAdditionPreserverEqualToScalarAlgClos:case4} are incompatible with the inequality~\eqref{lem:LinearAdditionPreserverEqualToScalarAlgClos:eq1}. Hence, $B \in \R_n(\F)$, where Lemma~\ref{lem:NonScalarElementofRnLowBound} implies that $B \in \langle I_n \rangle$.
\end{proof}

\begin{lemma}\label{lem:LinearAdditionPreserverEqualToScalar} The statement of Lemma~\ref{lem:LinearAdditionPreserverEqualToScalarAlgClos} holds for any field $\F$ satisfying $|\F| > n(n-1)$.
\end{lemma}
\begin{proof}Let $B \in M_{n}(\F)$ be a matrix satisfying the equality~\ref{lem:LinearAdditionPreserverEqualToScalarAlgClos:eqq1}. Since $|\F| > n(n-1)$, than as in Lemma~\ref{lem:LinearAdditionPreserverEqualToScalarAlgClos} we obtain that 
\begin{equation}\label{lem:LinearAdditionPreserverEqualToScalar:eq1}
\deg_{\lambda} \left(\D_{(B)}\right) \le 0.
\end{equation}

Let $\overline{\F}$ be an algebraic closure of $\F$. Since $\D_{(B)} \in \F[\lambda; a_{1\,1}, \ldots, a_{n\,n}] \subseteq \overline{\F}[\lambda; a_{1\,1}, \ldots, a_{n\,n}]$, the inequality~\eqref{lem:LinearAdditionPreserverEqualToScalar:eq1} implies that $\deg_{\lambda}\D_{(B)}(\overline{A}) \le 0$ for all $\overline{A} \in \M_{n}(\overline{\F}).$ Therefore, 
\begin{equation}\label{lem:LinearAdditionPreserverEqualToScalar:eq1AA1}
\D_{(B)}(\overline{\lambda}; \overline{A}) = \D(\overline{A} + \overline{\lambda} B) = \mathrm{const}\;\;\mbox{for all}\;\;\overline{\lambda} \in \overline{\F},
\end{equation}
where $\overline{A} \in \M_{n}(\overline{\F})$ is fixed. Since 
\[
\D_{(B)}(0; \overline{A}) = \D(\overline{A} + 0\cdot B) = \D(\overline{A}),
\]
then~\eqref{lem:LinearAdditionPreserverEqualToScalar:eq1AA1} implies that
\begin{equation}\label{lem:LinearAdditionPreserverEqualToScalar:eq2}
\D(\overline{A} + \overline{\lambda} B) = \D(\overline{A})\;\;\mbox{for all}\;\;\overline{\lambda} \in \overline \F\;\;\mbox{and}\;\; \overline A \in \M_{n}(\overline \F).
\end{equation}
Hence, $B = cI_n$ for some $c \in \overline{\F}$ by Lemma~\ref{lem:LinearAdditionPreserverEqualToScalarAlgClos}. In particular $b_{1\,1} = c$. Therefore $c \in \F$ and the statement of the lemma is established.
\end{proof}

\begin{corollary}\label{cor:DRad} 
Assume that $n \ge 3$ and $|\F| > n(n-1)$. Then $\rad (\D) = \langle I_n\rangle$.
\end{corollary}

\begin{corollary}\label{cor:D2Rad}
Assume that $n \ge 3$, $\Char(\F) = 2$ and $|\F| > n(n-1)$. Then $\rad (\B) = \langle I_n\rangle$.
\end{corollary}

\section{Vector spaces of matrices belonging to $\mathfrak R_n(\F)$ and their properties}
\label{sec:VecSpacesInR1}

In this section we introduce the notion of vector space $\mbf L(\mbf v) \subseteq \R_n(\F)$ corresponding to every $0 \neq \mbf v \in \F^n$ (Definition~\ref{def:LMatSpace}) and develop the theory of such spaces which is used in the sequel.

\begin{lemma}\label{lem:DimColSpInE}Let $E \subseteq \F^n$ be a vector space. Define a vector space $W \subseteq \M_n(\F)$ by 
$$
W = \{A \in \M_{n}(\F) \mid \colsp (A) \subseteq E\}.
$$
Then $\dim(W) = n\dim(E)$.
\end{lemma}
\begin{proof}Indeed, if the set $\mathbf v^{(1)}, \ldots, \mathbf v^{(d)}$ forms a basis of $E$ for $d = \dim(E)$, then the set
$$
\begin{psmallmatrix}
\mathbf v^{(1)}_1 & 0 & \cdots & 0\\
\vdots & \vdots & \ddots & \vdots\\
\mathbf v^{(1)}_n & 0 & \cdots & 0
\end{psmallmatrix}, \ldots,
\begin{psmallmatrix}
\mathbf v^{(d)}_1 & 0 & \cdots & 0\\
\vdots & \vdots & \ddots & \vdots\\
\mathbf v^{(d)}_n & 0 & \cdots & 0
\end{psmallmatrix}, \ldots, 
\begin{psmallmatrix}
0 & \cdots & 0 & \mathbf v^{(d)}_1 \\
\vdots & \ddots & \vdots & \vdots\\
0 & \cdots & 0 & \mathbf v^{(d)}_n
\end{psmallmatrix}
$$
form a basis of $W$.
\end{proof}

\begin{lemma}[{Cf.~\cite[Theorem~3]{meshulam1985}}]\label{lem:MaxSpaceBoundedRank}
Suppose $W \subseteq \M_{n}(\F)$ is a subspace of dimension $rn$, such that $\rk(A) \le r$ for all $A \in W$. Then there is a vector space $E \subseteq \F^n$ such that either $W = \{A \in \M_{n}(\F) \mid \colsp (A) \subseteq E\}$, or $W = \{A \in \M_{n}(\F) \mid \rowsp(A) \subseteq E\}$. 
\end{lemma}

\begin{definition}\label{def:LMatSpace}Let $0 \neq \mathbf v \in \F^n$. By $\L(\mathbf v) \subseteq \M_{n}(\F)$ we denote the $n+1$ dimensional vector space of matrices of the form $A + cI_{n}$, where $c \in \F$ and $\colsp(A) \subseteq \langle\mathbf v\rangle$.
\end{definition}

Using Corollary~\ref{cor:OneWayRankOnePlusScalar}, we conclude in the corollary below that if $n \ge 3$, then every element of $\L(\mathbf v)$ has the form described in Definition~\ref{def:LMatSpace}.

\begin{corollary}\label{cor:ExpressionForLElement}
Assume that $n \ge 3$. Let $0 \neq \mathbf v \in \F^n$ and $B \in \M_{n}(\F)$. Then $B \in \L(\mathbf v)$  if and only if  $B = J(\mbf a, \mbf v) + cI_{n}$, where $c \in \F$ and $\mbf a \in \F^n$.
\end{corollary}

\begin{lemma}\label{lem:ABRk2SumNotinR}Assume that $n \ge 3$. If $A, B \in \M_{n}(\F)$ are such that\linebreak $\rk(A), \rk(B) = 1$ and 
\begin{equation}\label{lem:ABRk2SumNotinR:eqq1}
\rk(A + B) = 2,
\end{equation} then there is $\alpha \in \F$ such that $\alpha A + B \not\in \R_n(\F)$.
\end{lemma}
\begin{proof}The proof is divided into two parts as follows.
\begin{enumerate}[label=(\roman*), ref=(\roman*)]
\item\label{lem:ABRk2SumNotinR:part1} We prove the statement of the lemma assuming that 
\begin{equation}\label{lem:ABRk2SumNotinR:eq1}
A = \begin{psmallmatrix}a_1 & \cdots & a_n\\ 0 & \cdots & 0\\ \vdots & \ddots & \vdots\\ 0 & \cdots & 0\end{psmallmatrix}\quad\mbox{and}\quad B = \begin{psmallmatrix}0 & \cdots & 0\\ b_1 & \cdots & b_n\\ \vdots & \ddots & \vdots\\ 0 & \cdots & 0\end{psmallmatrix}.
\end{equation}
\item\label{lem:ABRk2SumNotinR:part2} We prove the statement of the lemma without any additional assumptions on $A$ and $B$.    
\end{enumerate}

\paragraph{\ref{lem:ABRk2SumNotinR:part1}} Assume that $A$ and $B$ satisfy the assumptions of the lemma the condition~\eqref{lem:ABRk2SumNotinR:eq1}. If $n \ge 4$, then $\alpha A + B \not\in \R_n(\F)$ for $\alpha = 1$. Indeed,  otherwise there is  $c \in \F$ is such that 
\begin{equation}\label{lem:ABRk2SumNotinR:eq2}
\rk(A + B - cI_n) = 1,
\end{equation}
which implies that 
\[
0 = \det \left((A + B - cI_n)[3,4|3,4])\right) = c^2.
\]
Hence, $c = 0$ which yields the contradicting equalities~\eqref{lem:ABRk2SumNotinR:eqq1} and~\eqref{lem:ABRk2SumNotinR:eq2}.

Assume now that $n = 3$ and prove the statement of the lemma by the contradiction. Suppose that $A, B \in \M_{3}(\F)$ satisfy the conditions of the lemma and $\alpha A + B \in \R_n(\F)$ for all $\alpha \in F$. Equivalently, for every $\alpha \in \F$ there is $c$ such that 
\begin{equation}\label{lem:ABRk2SumNotinR:eq3}
\rk(\alpha A + B - cI_3) \le 1.
\end{equation}
The equality~\eqref{lem:ABRk2SumNotinR:eqq1} implies that
\begin{equation}\label{lem:ABRk2SumNotinR:eq4}
\rk(\alpha A + B - cI_3) \le 1\quad \Rightarrow \quad c \neq 0. 
\end{equation}

Let $\alpha = \alpha' = 1$ and $c'$ be such that the condition~\eqref{lem:ABRk2SumNotinR:eq3} is satisfied for this $\alpha$. From~\eqref{lem:ABRk2SumNotinR:eq3} we obtain that
\begin{gather}
\det \left(\left((\alpha'A + B - c'I_3)[1,3|1,3]\right)\right)=0 \notag \\ 
(\alpha' a_1 - c')c' = 0.\label{lem:ABRk2SumNotinR:eq5}
\end{gather}
Since $c' \neq 0$ by~\eqref{lem:ABRk2SumNotinR:eq4} and~\eqref{lem:ABRk2SumNotinR:eq3}, then~\eqref{lem:ABRk2SumNotinR:eq5} implies that $a_1 \neq 0$. Therefore, there is $\alpha'' \in \F$ such that 
\begin{equation}\label{lem:ABRk2SumNotinR:eq6}
\alpha''a_1 - b_2 \neq 0
\end{equation}
because $|\F| \ge 2$.

Let $\alpha = \alpha''$ and $c''$ be such that the condition~\eqref{lem:ABRk2SumNotinR:eq3} is satisfied for this $\alpha$. From~\eqref{lem:ABRk2SumNotinR:eq3} we obtain that
\begin{gather}
\left\{\begin{aligned}
\det \left(\left((\alpha''A + B - c''I_3)[1,3|1,3]\right)\right) & =  0\\
\det \left(\left((\alpha''A + B - c''I_3)[2,3|2,3]\right)\right) & =  0
\end{aligned}
\right.\notag\\
\left\{\begin{aligned}
(\alpha'' a_1 - c'')c'' & =  0\\
(b_2 - c'')c'' & =  0
\end{aligned}
\right. \label{lem:ABRk2SumNotinR:eq7}
\end{gather}
By subtracting the second equality in~\eqref{lem:ABRk2SumNotinR:eq7} from the first we obtain that
\begin{equation}\label{lem:ABRk2SumNotinR:eq8}
(\alpha'' a_1 - b_2)c'' =  0.
\end{equation}
Since $c_1 \neq 0$ by~\eqref{lem:ABRk2SumNotinR:eq4} and~\eqref{lem:ABRk2SumNotinR:eq3}, then~\eqref{lem:ABRk2SumNotinR:eq8} implies that $\alpha'' a_1 - b = 0$ which contradicts~\eqref{lem:ABRk2SumNotinR:eq6}. This establishes the part~\ref{lem:ABRk2SumNotinR:part1} of the lemma. 
\paragraph{\ref{lem:ABRk2SumNotinR:part2}} Assume that $A$ and $B$ satisfy the conditions of the lemma. Hence, there are $0 \neq \mbf u \neq \mbf v \in \F^n$ such that $\colsp(A) = \langle \mbf u\rangle$ and  $\colsp(B) = \langle \mbf v\rangle$. Since $0 \neq \mbf u \neq \mbf v$, there is an invertible matrix $M \in \M_{n}(\F)$ such that 
\begin{equation}\label{lem:ABRk2SumNotinR:eq9}
M\mbf u = \mbf e_1\quad \mbox{and}\quad M\mbf v = \mbf e_2.
\end{equation} 
Let $A' = MAM^{-1}, B' = MBM^{-1}$.
Since the similarity transformation preserves the rank of a matrix (see, for example~\cite[Theorem 3.4]{Friedberg2002-wu}), then\linebreak $\rk(A'), \rk(B') = 1$ and  $\rk(A' + B') = 2$. In addition, the equality~\eqref{lem:ABRk2SumNotinR:eq9} implies that $\colsp(A') = \langle \mbf e_1\rangle$ and $\colsp(B') = \langle \mbf e_2\rangle$. Therefore, the part~\eqref{lem:ABRk2SumNotinR:part1} of the lemma is applicable to $A'$ and $B'$, from which we conclude that there is $\alpha \in \F$ such that 
$\alpha A' + B' \not\in \R_{n}(\F).$ This equality infers that
$$\alpha A + B = M^{-1}\left(\alpha A' + B'\right)M \not\in \R_{n}(\F)$$
because $M^{-1}\R_n(\F)M = \R_{n}(\F)$ by Lemma~\ref{lem:RsetProp}. 
\end{proof}

\begin{lemma}\label{lem:SpaceInR1CapR1}Assume that $n \ge 3$. If $W$ is a vector subspace of $W \subseteq \R_n(\F)$ then
\begin{equation}\label{lem:SpaceInR1CapR1:eqq1}
U = \{X \in W \mid \rk(X) \le 1\}
\end{equation}
is a vector subspace of $W$.
\end{lemma}
\begin{proof}Since $n \ge 3$, then $\lambda A \in U$ for all $\lambda \in \F$ and $A \in U$ by Corollary~\ref{cor:OneWayRankOnePlusScalar}. Thus, it is sufficient to show that $A + B \in U$ for all $A, B \in U$. We prove this fact by contradiction. Suppose that $A + B \not\in U$. Hence, $\rk(A + B) = 2$ and consequently $\rk(A), \rk(B) = 1$ by Lemma~\ref{lem:RkSum}. Therefore, by Lemma~\ref{lem:ABRk2SumNotinR}, there exists $\alpha \in \F $ such that $\alpha A + B \not\in \R_n(\F)$ which contradicts the inclusion $(\alpha A + B) \in W \subseteq \R_n(\F)$.
\end{proof}

\begin{lemma}\label{lem:RepAsRkOnePScalarNP1EqToL}Assume that $n \ge 3$. Let $W$ be a subspace of $\M_n(\F)$ such that $\dim(W) = n + 1$, $I_n \in W$ and $W \subseteq \R_n(\F)$.
Then there is $0 \neq \mathbf v \in \F^n$ such that either $W = \L(\mathbf v)$, or $ W = \L^t(\mathbf v)$.
\end{lemma}
\begin{proof}Let $U$ is defined as in~\eqref{lem:SpaceInR1CapR1:eqq1}.  Lemma~\ref{lem:SpaceInR1CapR1} implies that $U$ is a vector space. Since $n \ge 3$, then $I_n \not\in U$ and consequently $U \subsetneq W$. Hence, $\dim(U) < \dim(W)$. In addition, from the inclusions $W \subseteq \R_n(\F)$ and $I_n \in W$ we conclude that $W \subseteq (U + \langle I_n \rangle)$. Therefore, 
\[
\dim(U) = \dim(W) - 1 = (n + 1) - 1 = n
\]
and
\begin{equation}\label{lem:RepAsRkOnePScalarNP1EqToL:eq1}
W = (U + \langle I_n \rangle).
\end{equation}
It follows from Lemma~\ref{lem:MaxSpaceBoundedRank} that there is  $0 \neq \mathbf v \in \F^n$ such that either $U = \{A \in \M_{n}(\F) \mid \colsp (A) \subseteq \langle v \rangle \}$, or $U = \{A \in \M_{n}(\F) \mid \rowsp(A) \subseteq \langle v \rangle \}$. In the first case using~\eqref{lem:RepAsRkOnePScalarNP1EqToL:eq1} we obtain that $W = \L(\mathbf v)$, while in the second $ W = \L^t(\mathbf v)$.
\end{proof}

\begin{lemma}\label{lem:CapSame}Assume that $n \ge 2$.  If $0 \neq \mathbf v, \mathbf v' \in \F^n$ are linearly independent, then $\L(\mathbf v) \cap \L(\mathbf v') = \langle I_n\rangle$. 
\end{lemma}
\begin{proof}The inclusion $\langle I_n\rangle \subseteq \L(\mathbf v) \cap \L(\mathbf v')$ clearly holds. In order to prove the inverse inclusion, let $X \in \L(\mathbf v) \cap \L(\mathbf v')$. Corollary~\ref{cor:ExpressionForLElement} implies that there are $\mbf a, \mbf a' \in \F^n$ and $c, c' \in \F$ such that
\begin{equation}\label{lem:CapSame:eq1}
X = J(\mbf a, \mbf v) + cI_n = J(\mbf a', \mbf v') + c'I_n.
\end{equation} 
By applying Corollary~\eqref{cor:OneWayRankOnePlusScalar} to~\eqref{lem:CapSame:eq1} we obtain that $J(\mbf a, \mbf v) = J(\mbf a', \mbf v').$ Since we assume that $\mbf v, \mbf v'$ are linearly independent, then 
\begin{equation}\label{lem:CapSame:eq2}
J(\mbf a, \mbf v) = J(\mbf a', \mbf v') = O_n
\end{equation}
by the property~\ref{prop:JABUNiqExpr}. Hence,
\[
X \overset{\eqref{lem:CapSame:eq1}}{=\joinrel=} J(\mbf a, \mbf v) + cI_n \overset{\eqref{lem:CapSame:eq2}}{=\joinrel=} O_n + cI_n = cI_n.
\]
Therefore, the required equality is established.
\end{proof}

\begin{corollary}\label{cor:DimCapSame}Assume that $n \ge 2$.  If $0 \neq \mathbf v, \mathbf v' \in \F^n$ are linearly independent, then $\dim\left(\L(\mathbf v) \cap \L(\mathbf v')\right) = 1$. 
\end{corollary}

\begin{lemma}\label{lem:CapTrans}Assume that $n \ge 2$.  If $0 \neq \mathbf v, \mathbf v' \in \F^n$, then $\L(\mathbf v) \cap \L^t(\mathbf v') = \left\langle I_n, J\left(\mathbf v', \mathbf v\right)\right\rangle.$ 
\end{lemma}
\begin{proof}The inclusion $\left\langle I_n, J\left(\mathbf v', \mathbf v\right)\right\rangle \subseteq \L(\mathbf v) \cap \L^t(\mathbf v')$ clearly holds. In order to prove the inverse inclusion, let $X \in \L(\mathbf v) \cap \L^t(\mathbf v')$. Corollary~\ref{cor:ExpressionForLElement} implies that there are $\mbf a, \mbf a' \in \F^n$ and $c, c' \in \F$ such that
\begin{equation}\label{lem:CapTrans:eq1}
X = J(\mbf a, \mbf v) + cI_n = J(\mbf v',\mbf a') + c'I_n.
\end{equation} 
Then, by applying to~\eqref{lem:CapTrans:eq1} the same arguments as in Lemma~\ref{lem:CapSame}, we eventually obtain $X \in \left\langle I_n, J\left(\mathbf v', \mathbf v\right)\right\rangle$. Therefore, the required equality is established.

\end{proof}

\begin{corollary}\label{cor:DimCapTrans}Assume that $n \ge 2$.  If $0 \neq \mathbf v, \mathbf v' \in \F^n$, then\linebreak $\dim\left(\L(\mathbf v) \cap \L^t(\mathbf v')\right) = 2.$ 
\end{corollary}

\begin{lemma}\label{lem:DimCupLSpaceAll} 
If $0 \neq v^{(1)}, \ldots, v^{(k)} \in \F^n$
\begin{equation}\label{lem:DimCupLSpaceAll:eqq1}
\dim\left(\L(v^{(1)}) + \ldots + \L(v^{(k)})\right) \le kn + 1.
\end{equation}
\end{lemma}
\begin{proof}Indeed, by extending $I_n$ to a basis $I_n, A^{(i\,1)}, \ldots A^{(i\,n)}$ for every $1 \le i \le k$ we obtain that 
\[
\L(v^{(1)}) + \ldots + \L(v^{(k)}) \subseteq \left\langle I_n, A^{(i\,1)}, \ldots, A^{(i\,n)}, \ldots, A^{(k\,n)}\right\rangle,
\]
which implies~\eqref{lem:DimCupLSpaceAll:eqq1}.
\end{proof}

\begin{lemma}\label{lem:DimCupLSpaceIndepCrit}Assume that $n \ge 2$. Let $1 \le l \le n$  and $0 \neq \mathbf v^{(1)}, \ldots, \mathbf v^{(k)} \in \F^n$. Then $\mathbf v^{(1)}, \ldots, \mathbf v^{(k)}$ are linearly independent if and only if
\begin{equation}\label{lem:DimCupLSpaceIndep:eqq1}
\dim\left(\L(v^{(1)}) + \ldots + \L(v^{(k)}) \right) \ge kn.
\end{equation}
\end{lemma}
\begin{proof}

Assume that $\mathbf v^{(1)}, \ldots, \mathbf v^{(k)} \in \F^n$ are linearly independent. This means that 
\begin{equation}\label{lem:DimCupLSpaceIndep:eq1}
\dim(E) = k,
\end{equation}
 where $E = \left\langle v^{(1)}, \ldots, v^{(k)}\right\rangle$. Let $W = \{A \in \M_{n}(\F) \mid \colsp (A) \subseteq E\}.$ Then
\begin{equation}\notag
\dim(W) \overset{\mbox{\scriptsize Lemma~\ref{lem:DimColSpInE}}}{=\joinrel=\joinrel=\joinrel=\joinrel=\joinrel=} n\dim(E) \overset{\eqref{lem:DimCupLSpaceIndep:eq1}}{=\joinrel=\joinrel=\joinrel=}  kn.
\end{equation}
In addition, the definitions of $W$ and $\L(\mathbf v^{(1)}), \ldots, \L(\mathbf v^{(k)})$ imply that 
\begin{equation}\label{lem:DimCupLSpaceIndep:eq2}
W \subseteq \L(\mathbf v^{(1)}) + \ldots + \L(\mathbf v^{(k)}).
\end{equation}
Therefore,
\[
nk = \dim(W) \le \dim\left( \L(\mathbf v^{(1)}) + \ldots + \L(\mathbf v^{(k)})\right).
\]
Thus, the required inequality holds.

Assume that $\mathbf v^{(1)}, \ldots, \mathbf v^{(k)} \in \F^n$ are not linearly independent. Since all $\mathbf v^{(1)}, \ldots, \mathbf v^{(k)}$ are nonzero vectors, then in addition we assume without loss of generality that $1 \le l < k$ is such that 
\begin{equation}\label{lem:DimCupLSpaceIndep:eq3}
\mathbf v^{(i)} \in \left\langle\mathbf v^{(1)}, \ldots, \mathbf v^{(l)}\right\rangle\;\;\mbox{for all}\;\;l < i \le k.
\end{equation}  
Therefore, from the definition of $\L(\mathbf v^{(1)}), \ldots, \L(\mathbf v^{(k)})$ we conclude that
\begin{equation*}
\L(\mathbf v^{(i)} \subseteq \L(v^{(1)}) + \ldots + \L(v^{(l)})\;\;\mbox{for all}\;\;l < i \le k
\end{equation*}
and consequently
\begin{equation}\label{lem:DimCupLSpaceIndep:eq4}
\L(v^{(1)}) + \ldots + \L(v^{(k)}) \subseteq \L(v^{(1)}) + \ldots + \L(v^{(l)}).
\end{equation}
In addition
\begin{equation}\label{lem:DimCupLSpaceIndep:eq5}
\dim\left(\L(v^{(1)}) + \ldots + \L(v^{(l)})\right) \overset{\mbox{\scriptsize Lemma~\ref{lem:DimCupLSpaceAll}}}{\le } ln + 1
\end{equation}
and consequently
\begin{equation}\label{lem:DimCupLSpaceIndep:eq6}
\dim\left(\L(v^{(1)}) + \ldots + \L(v^{(k)})\right) \overset{\eqref{lem:DimCupLSpaceIndep:eq4}}{\le} \dim\left(\L(v^{(1)}) + \ldots + \L(v^{(l)})\right) \overset{\eqref{lem:DimCupLSpaceIndep:eq5}}{\le} ln + 1.
\end{equation}
Since $l < k$, then $l \le k - 1$, which implies that 
\begin{equation}\label{lem:DimCupLSpaceIndep:eq7}
ln + 1\le (k-1)n + 1 = kn + (1-n) < kn.
\end{equation}
Therefore,
\[
\dim\left(\L(v^{(1)}) + \ldots + \L(v^{(k)})\right) \overset{\eqref{lem:DimCupLSpaceIndep:eq6}}{\le} ln + 1 \overset{\eqref{lem:DimCupLSpaceIndep:eq7}}{<} kn,
\]
which concludes the statement of the lemma in this case.
\end{proof}

\begin{lemma}\label{lem:DiscPresPresIncidence}Assume that $n \ge 2$. Let $T$ be a bijective linear map on $\M_{n}(\F)$ transforming the spaces of the form $\L(\mathbf v)$ to the spaces of the same form, $1 \le k \le n$ is a natural number, and nonzero vectors $\mathbf u^{(1)}, \ldots, \mathbf u^{(k)}, \mathbf v^{(1)}, \ldots, \mathbf v^{(k)} \in \F^n$ are such that
\[T(\L(\mathbf u^{(i)})) = \L(\mathbf v^{(i)})\;\;\mbox{for all}\;\;1 \le i \le k.\]
Then $\mathbf u^{(1)}, \ldots, \mathbf u^{(k)}$ are linearly independent if and only if  $\mathbf v^{(1)}, \ldots, \mathbf v^{(k)}$ are linearly independent.
\end{lemma}
\begin{proof}Suppose that $\mathbf u^{(1)}, \ldots, \mathbf u^{(k)}$ are linearly independent. Then it follows from Lemma~\ref{lem:DimCupLSpaceIndepCrit} that
\[\dim\left(\L(v^{(1)}) + \ldots + \L(v^{(k)})\right)\ge kn.\]
Since the linear map $T$ is bijective and $T(\L(\mathbf u^{(i)})) = \L(\mathbf v^{(i)})$ for all $1 \le i \le k,$ then
\begin{multline*}
\dim\left(\L(\mathbf v^{(1)}) + \ldots + \L(\mathbf v^{(k)})\right)\\
= \dim\left(T\left(\L(\mathbf u^{(1)}) + \ldots + \L(\mathbf u^{(k)})\right)\right)\\
= \dim\left(\L(\mathbf u^{(1)}) + \ldots + \L(\mathbf u^{(k)})\right) \ge kn.
\end{multline*}
Therefore, $\mathbf v^{(1)}, \ldots, \mathbf v^{(k)}$ are linearly independent by Lemma~\ref{lem:DimCupLSpaceIndepCrit}.

The proof in the opposite direction is carried out by the similar arguments.
\end{proof}

\begin{corollary}\label{cor:DiscPresPresIncidence}Assume that $n \ge 2$.  Let $T$ be a bijective linear map on $\M_n(\F)$ satisfying the conditions of Lemma~\ref{lem:DiscPresPresIncidence}, $1 \le k < n$ and \linebreak $\mathbf v^{(1)}, \ldots, \mathbf v^{(k+1)},$ $\mathbf u^{(1)}, \ldots, \mathbf u^{(k+1)}$ are nonzero vectors such that
\[T(\L (\mathbf u^{(i)})) = \L(\mathbf v^{(i)})\;\;\mbox{for all}\;\;1\le i \le k+1.\]
Then 
\[\mathbf u^{(k+1)} \in \left\langle\mathbf u^{(1)},\ldots, \mathbf u^{(k)}\right\rangle\]
implies that
\[\mathbf v^{(k+1)} \in \left\langle\mathbf v^{(1)},\ldots, \mathbf v^{(k)}\right\rangle.\]
\end{corollary}

\begin{lemma}\label{lem:MatMultOnLn}If  $M \in \M_{n}(\F)$ is an invertible matrix, then 
\[
M\L(\mbf v)M^{-1} = \L(M\mbf v)\;\;\mbox{for all}\;\;\mbf v \in \F.
\] 
\end{lemma}
\begin{proof}Let $\mathbf v \in \F^n$. Since the linear map $X \mapsto MXM^{-1}$ is invertible, then it is sufficient to verify that 
\begin{equation}\label{lem:MatMultOnLn:eq11}
M\L(\mbf v)M^{-1} \subseteq \L(M\mbf v).
\end{equation}
For establishing~\eqref{lem:MatMultOnLn:eq11}, let $B \in \L(\mbf v)$. By Lemma~\ref{cor:ExpressionForLElement}, there is $c \in \F$ and $A \in \M_{n}(\F)$ such that
\begin{equation}\label{lem:MatMultOnLn:eq1}
\colsp(A) \subseteq \langle \mathbf v\rangle
\end{equation}
and 
\begin{equation}\label{lem:MatMultOnLn:eq2}
B = A + cI_{n}.
\end{equation}
Then
\begin{equation}\label{lem:MatMultOnLn:eq4}
MBM^{-1} \overset{\eqref{lem:MatMultOnLn:eq2}}{=\joinrel=\joinrel=} M(A + cI_{n})M^{-1} = MAM^{-1} + c I_{n}.
\end{equation}
Consider $MAM^{-1}$. The condition~\eqref{lem:MatMultOnLn:eq1} is equivalent to the inclusion
\begin{equation}\label{lem:MatMultOnLn:eq3}
A\mathbf u \in \langle\mbf v\rangle\;\;\mbox{for all}\;\;\mathbf u \in \F^n.
\end{equation}
Relying on that fact we obtain that
\begin{multline}\label{lem:MatMultOnLn:eq5}
\left(MAM^{-1}\right)\mathbf u = M\left(A\left(M^{-1}\mathbf u\right)\right) \overset{\eqref{lem:MatMultOnLn:eq3}}{\in} M (\langle\mathbf v\rangle)\\
\overset{\mbox{\scriptsize $M$ is invertible}}{=\joinrel=\joinrel=\joinrel=\joinrel=\joinrel=\joinrel=\joinrel=\joinrel=} \langle M\mathbf v\rangle\;\;\mbox{for all}\;\;\mathbf u \in \F^n.
\end{multline}
Since~\eqref{lem:MatMultOnLn:eq5} is equivalent to that $\colsp(MAM^{-1}) \subseteq \langle M\mathbf v\rangle$, then using~\eqref{lem:MatMultOnLn:eq4} we conclude~\eqref{lem:MatMultOnLn:eq11}.
\end{proof}

\section{Linear maps preserving the matrix discriminant}
\label{sec:LinDiscPres}

In this section, relying on the facts obtained above, we study the properties of linear maps preserving the matrix discriminant, which allows to provide their characterisation as it is stated in Theorem~\ref{thm:MainThmInIntro}  (see Theorem~\ref{thm:DPresHasForm}).

\begin{lemma}\label{lem:IfPresDThenPresRk1PSc}Assume that $n \ge 3$ and  $|\F| > n(n-1)$. Let $T\colon \M_{n}(\F) \to \M_{n}(\F)$ be a bijective linear map such that $\D(T(X)) = \D(X)$ for all $X \in \M_{n}(\F)$. If $B \in \M_{n}(\F)$ is such that $B = J + cI_n$ for some $c$ and a rank one non-nilpotent matrix $J \in \M_{n}(\F)$, then $T(B) \in \R_n(\F)$.
\end{lemma}
\begin{proof}The proof is divided into three parts as follows.
\begin{enumerate}[label=(\roman*), ref=(\roman*)]
\item\label{lem:IfPresDThenPresRk1PSc:part1} We prove the statement of the lemma assuming that $\F$ is algebraically closed and $B$ is similar to $E_{1\,1} + cI_n$ for some $c \in \F$.
\item\label{lem:IfPresDThenPresRk1PSc:part2} We prove the statement of the lemma assuming that $\F$ is algebraically closed and without any additional assumption on $B$.
\item\label{lem:IfPresDThenPresRk1PSc:part3} We prove the statement of the lemma without any additional assumption on $\F$ and $B$.    
\end{enumerate}
\paragraph{\ref{lem:IfPresDThenPresRk1PSc:part1}}Assume that $\F$ is algebraically closed and $B$ is similar to $E_{1\,1} + cI_n$ for some $c \in \F$. It follows from Lemma~\ref{lem:RankOneScalImpliesLowDeg} that
\begin{equation}\label{lem:IfPresDThenPresRk1PSc:eq11}
\deg_{\lambda} (\D(A + \lambda B)) \le 2n - 2\;\;\mbox{for all}\;\;A \in \M_{n}(\F).
\end{equation}
Let $B' = T(B)$. Fix $A \in \M_{n}(\F)$. Since $T$ is bijective, then there is $A' \in \M_{n}(\F)$ such that $T(A') = A$. Consider the polynomial $P = \D(A + \lambda B') \in \F[\lambda]$. We have
\begin{multline}\label{lem:IfPresDThenPresRk1PSc:eq2}
\D(A + \lambda B') = \D(T(A') + \lambda T(B)) = \D(T(A' + \lambda B))\\
 = \D(A' + \lambda B)\;\;\mbox{for all}\;\lambda \in \F
\end{multline}
because $T$ preserves $\D$. Since $\F$ is algebraically closed, then $\F$ infinite, and this implies that $\D(A + \lambda B')$ and $ \D(A' + \lambda B)$ are equal as elements of $\F[\lambda]$ by Lemma~\ref{DegFGLessEqual}. Therefore,
\begin{equation}\label{lem:IfPresDThenPresRk1PSc:eq3}
\deg_{\lambda} (\D(A + \lambda B')) \overset{\eqref{lem:IfPresDThenPresRk1PSc:eq2}}{=\joinrel=\joinrel=\joinrel=} \deg_{\lambda} (\D(A' + \lambda B))  \overset{\eqref{lem:IfPresDThenPresRk1PSc:eq11}}{\le} 2n - 2.
\end{equation}
The inequality~\eqref{lem:IfPresDThenPresRk1PSc:eq3} holds for all $A \in \M_{n}(\F)$. Consequently,
\begin{equation*}
\deg_{\lambda} (\D(A + \lambda B')) \le 2n - 2\;\;\mbox{for all}\;\;A \in \M_{n}(\F).
\end{equation*}
This allows us to apply Lemma~\ref{lem:LowDegImpliesLowRankOrSeveralCasesForAlgebraicallyClosed} to $B'$ and conclude that $B'$ satisfies one of the following conditions:
\begin{enumerate}[label=C\arabic*., ref=C\arabic*]
\item\label{lem:IfPresDThenPresRk1PSc:cond1} $B' \in \R_n(\F)$.
\item\label{lem:IfPresDThenPresRk1PSc:cond2} $n = 3$ and $B'$ is similar to $\begin{psmallmatrix}0 & 1 & 0\\0 & 0 & 1\\ 0 & 0 & 0\end{psmallmatrix} + cI_n$ for some $c \in \F$.
\item\label{lem:IfPresDThenPresRk1PSc:cond3} $n = 4$ and $B'$ is similar to $\begin{psmallmatrix}0 & 1 & 0 & 0\\0 & 0 & 0 & 0\\ 0 & 0 & 0 & 1\\ 0 & 0 & 0 & 0\end{psmallmatrix}  + cI_n$ for some $c \in \F$.
    
\item\label{lem:IfPresDThenPresRk1PSc:cond4} $n = 5$, $\Char(\F) = 2$ and  $B'$ is similar to $\begin{psmallmatrix}0 & 1 & 0 & 0 & 0\\0 & 0 & 0 & 0 & 0\\ 0 & 0 & 0 & 1 & 0\\ 0 & 0 & 0 & 0 & 0\\ 0 & 0 & 0 & 0 & 0\end{psmallmatrix}  + cI_n$ for some $c \in \F$.
\end{enumerate}

The condition~\ref{lem:IfPresDThenPresRk1PSc:cond1} directly implies the statement of the lemma. Let us show that the conditions~\ref{lem:IfPresDThenPresRk1PSc:cond2}--\ref{lem:IfPresDThenPresRk1PSc:cond4} could not hold.

Suppose that the condition~\ref{lem:IfPresDThenPresRk1PSc:cond2} holds. On the one hand, since $B$ is similar to $E_{1\,1} + cI_n$, then $\deg_\lambda \left(\D_{(B')}\right)$  and $\dim \left(\mathcal L_{\D_{(B')}^{(i)}}\right)$ (or $\deg_\lambda \left(\B_{(B')}\right)$  and $\dim \left(\mathcal L_{\B_{(B')}^{(i)}}\right)$ for $\Char(\F) = 2$) are equal to the corresponding invariants of $\D_{(B)}$ by Corollary~\ref{cor:DegAndLPAreReallyInvariant} (Lemma~\ref{lem:PresDEqPresD2} is used additionally for $\Char(\F) = 2$). In turn, the corresponding invariants of $\D_{(B)}$ are equal to the corresponding invariants of $\D_{(E_{1\,1})}$ Lemma~\ref{lem:SimPlusScalSameInvD} (or Lemma~\ref{lem:SimPlusScalSameInvD2}) 

On the other hand, the condition~\ref{lem:IfPresDThenPresRk1PSc:cond2} and Lemma~\ref{lem:SimPlusScalSameInvD} (or Lemma~\ref{lem:SimPlusScalSameInvD2}) imply that these invariants of $B'$ are equal to the corresponding invariants of $E_{1\,2} + E_{2\,3} = \begin{psmallmatrix}0 & 1 & 0\\ 0 & 0 & 1\\ 0 & 0 & 0 \end{psmallmatrix}$. According to Table~\ref{tbl:compTab}, these invariants are not equal for $E_{1\,1}$ and $E_{1\,2} + E_{2\,3}$ in all possible cases of $\Char(\F)$. This lead to a contradiction. Hence, this condition does not hold.

The conditions~\ref{lem:IfPresDThenPresRk1PSc:cond2}--\ref{lem:IfPresDThenPresRk1PSc:cond4} lead to a contradiction similar to that was shown above for the condition~\ref{lem:IfPresDThenPresRk1PSc:cond2}, where we used Table~\ref{tbl:compTab} to distinguish between the required matrices.

Thus, none of the conditions~\ref{lem:IfPresDThenPresRk1PSc:cond2}--\ref{lem:IfPresDThenPresRk1PSc:cond4} could not hold. Therefore, only the condition~\ref{lem:IfPresDThenPresRk1PSc:cond1} holds, which concludes the statement of the lemma in this case.

\paragraph{\ref{lem:IfPresDThenPresRk1PSc:part2}}  Assume that $\F$ is algebraically closed. Let $B \in \M_{n}(\F)$ be such that $B = J + cI_n$ for some  $c \in \F$ and a non-nilpotent rank one matrix $J \in \M_{n}(\F)$. Lemma~\ref{lem:RK1Sim} implies that $J$ is similar to $eE_{1\,1}$ for some $0 \neq e \in \F$. Therefore, if $C \in \M_{n}(\F)$ is defined by $C = (ec)^{-1}B$, then  $C$ is similar to $E_{1\,1} + fI_n$ for some $f \in \F$. The part~\ref{lem:IfPresDThenPresRk1PSc:part1} of the lemma implies that $T(C) \in \R_n(\F)$. Hence,
\[
T(B) = T((ec)\cdot C) = ec \cdot T(C) \in \R_n(\F)
\]
because $ec \neq 0$. 

\paragraph{\ref{lem:IfPresDThenPresRk1PSc:part3}} 
Let $B \in \M_{n}(\F)$ be such that $B = J + cI_n$ for some  $c\in \F$ and a non-nilpotent rank one matrix $J \in \M_{n}(\F)$. In addition, let $\overline{\F}$ be an algebraic closure of $\F$. Define a linear $\overline{T} \colon \overline{\F} \to \overline{\F}$ by 
\begin{equation*}
\overline{T}(A) = a_{1\,1} T(E_{1\,1}) + \ldots + a_{n\,n} T(E_{n\,n})\;\;\mbox{for all}\;\;A = (a_{i\,j}) \in \M_n(\overline{\F}).
\end{equation*}
Since $E_{1\,1}, \ldots, E_{n\,n}$ form a basis of $\M_n(\overline{\F})$, this definition is correct.\footnote{We refer the reader to the section "Change of Base Field" in~\cite{roman_advanced_2007} for the detailed overview of this kind of operations on vector spaces and linear maps.}  In addition,
\begin{equation}\label{lem:IfPresDThenPresRk1PSc:eq102}
\overline{T}(A) = T(A)\;\;\mbox{for all}\;\;A \in \M_{n}(\F).
\end{equation}
Let us show that $\overline{T}$ preserves $\D$ on $\M_n(\overline{\F})$. For this, consider a polynomial $P(X) = \D(X) - \D(T(X)) \in \F[x_{1\,1}, \ldots, x_{n\,n}]$. Since $T$ is a linear map, then $\deg(P) \le \deg(\D)$. Hence, the initial assumption of the lemma  on $\deg(\D)$ and the property~\ref{enum:DiscPropertyHP} of the matrix discriminant imply that $\deg(P) < |\F|$. Therefore, 
\begin{equation}\label{lem:IfPresDThenPresRk1PSc:eq101}
P = 0\;\;\mbox{as an element of $\F[x_{1\,1}, \ldots, x_{n\,n}]$}
\end{equation}
by Lemma~\ref{DegFLessZero} because $T$ preserves $\D$.

The definition of $\overline{T}$ implies that $\D(X) - \D(\overline{T}(X)) = \D(X) - \D(T(X))$ as elements of $\overline{\F}[x_{1\,1}, \ldots, x_{n\,n}]$. Hence, 
\[
\D(X) - \D(\overline{T}(X)) = \D(X) - \D(T(X)) = P \overset{\eqref{lem:IfPresDThenPresRk1PSc:eq101}}{=\joinrel=\joinrel=} 0
\]
an consequently
\[
\D(X) - \D(\overline{T}(X)) = 0\;\;\mbox{for all}\;\;X \in \M_n(\overline{\F})
\]
which means that $\overline{T}$ preserves $\D$ on $\M_{n}(\overline{\F})$. Therefore, we can apply the part~\ref{lem:IfPresDThenPresRk1PSc:part2} of the lemma to $\overline{\F}$, $\overline{T}$ and $B$ (as an element of $\M_{n}(\F)$ and obtain that $\overline{T}(B) \in \R_n(\F')$. Note that $\overline{T}(B) = T(B) \in \M_n(\F)$.  Therefore, by Lemma~\ref{lem:RankOneScalarFromExtToGroundField} applied to $\overline \F, \F$ and $T(B)$ we obtain that $T(B) \in \R_n(\F)$.

This concludes the proof of this part of the lemma.
\end{proof}

\begin{lemma}\label{lem:TPresDiscSendsLtoLorLt}Assume that $n \ge 3$ and $|\F| > n(n-1)$. Let $T$ be a bijective linear map preserving the matrix discriminant. Then for every $0 \neq \mathbf u \in \F^n$ there is $0 \neq \mathbf v \in \F^n$ such that $T\left(\mbf L(\mbf u)\right) = \mathbf L(\mathbf v)$  or $T\left(\mbf L(\mbf u)\right) = \mathbf L^t(\mathbf v).$
\end{lemma}

\begin{proof}Assume that $T$ is a linear map preserving $\D$ and $0\neq \mbf u \in \F$. Let $\mbf a^{(1)}, \ldots, \mbf a^{(n)} \in \F^n$ be vectors forming a basis of $\F$  such that 
\begin{equation}\label{lem:TPresDiscSendsLtoLorLt:eq5}
(\mbf a^{(i)}, \mbf u) \neq 0\;\;\mbox{for all}\;\;1 \le i \le n.
\end{equation}
Such vectors exist because the bilinear form $(*,*)$ is non-degenerate an consequently $\dim(\{\mbf w \in \F^n \mid (\mbf w, \mbf u) = 0\} = n - 1$. Let $A^{(1)}, \ldots, A^{(n)}, B^{(1)}, \ldots, B^{(n)} \in \M_{n}(\F),$ $c_1, \ldots, c_n \in \F$ be defined by
\begin{equation}\label{lem:TPresDiscSendsLtoLorLt:eq1}
A^{(i)} = J(\mbf a^{i}, \mbf u),\;\; B^{(i)} = T\left(A^{(i)}\right) \quad 1 \le i \le n
\end{equation}
and
\begin{equation}\notag
c_i = (\mbf a^{(i)}, \mbf u).
\end{equation}
Let us show that $B^{(1)}, \ldots, B^{(n)}$ and $c_1, \ldots, c_n$ satisfy the conditions of \linebreak Lemma~\ref{lem:RNonzeroRAll}. Indeed, assume that $\alpha_1, \ldots, \alpha_n \in \F$ are such that
\begin{equation}\notag
\sum_{i=1}^{n} c_i \alpha_i \neq 0.
\end{equation}
Let $A, B \in \M_{n}(\F)$ be defined by
\begin{equation}\label{lem:TPresDiscSendsLtoLorLt:eq4}
A = \alpha_1 A^{(1)} + \ldots + \alpha_n A^{(n)}
\end{equation}
Then
\begin{multline}\label{lem:TPresDiscSendsLtoLorLt:eq2}
A = \alpha_1 A^{(1)} + \ldots + \alpha_n A^{(n)}\\
\overset{\eqref{lem:TPresDiscSendsLtoLorLt:eq1}}{=\joinrel=} \alpha_1 J(\mbf a^{1}, \mbf u) + \ldots + \alpha_1 J(\mbf a^{1}, \mbf u)\\
\overset{\mbox{\scriptsize \ref{prop:J1},\,\ref{prop:J2}}} J(\alpha_1 \mbf a^{(1)} + \ldots + \alpha_n \mbf a^{(n)}, \mbf u)
\end{multline}
Hence, $\rk(A) \le 1$ by~\ref{prop:J0}. In addition~\eqref{lem:TPresDiscSendsLtoLorLt:eq2} implies that $A$ is non-nilpotent by 
Lemma~\eqref{prop:JNilpotentCrit}. Hence, by applying Lemma~\ref{lem:IfPresDThenPresRk1PSc} to $T$ and $A$ we obtain that 
\begin{equation}\label{lem:TPresDiscSendsLtoLorLt:eq3}
T(A) \in \R_n(\F).
\end{equation}
Therefore,
\begin{multline}\notag
\alpha_1 B^{(1)} + \ldots + \alpha_n B^{(n)} = \alpha_1 T\left(A^{(1)}\right) + \ldots + \alpha_n \left(A^{(n)}\right)\\
= T(\alpha_1 A^{(1)} + \ldots + \alpha_n A^{(n)}) \overset{\eqref{lem:TPresDiscSendsLtoLorLt:eq4}}{=\joinrel=} T(A) \overset{\eqref{lem:TPresDiscSendsLtoLorLt:eq2}}{\in} \R_{n}(\F).
\end{multline}
Thus, $B^{(1)}, \ldots, B^{(n)}$ and $c_1, \ldots, c_n$ satisfy the conditions of Lemma~\ref{lem:RNonzeroRAll}. This implies that
\begin{equation}\label{lem:TPresDiscSendsLtoLorLt:eq8}
\left\langle B^{(1)}, \ldots, B^{(n)}\right\rangle \subseteq \R_{n}(\F).
\end{equation}

By Corollary~\ref{cor:DRad} and Lemma~\ref{lem:KernelContainedInRad}, $\ker(T) \subseteq \langle I_n\rangle$. Since $T$ is bijective by the initial assumptions of the lemma, this implies that there is $0 \neq d \in \F$ such that
\begin{equation}\label{lem:TPresDiscSendsLtoLorLt:eq7}
T(I_n) = dI_n.
\end{equation}
From the  definition of $A^{(1)}, \ldots, A^{(n)}$ and $\mbf a^{(1)}, \ldots, \mbf a^{(n)}$ in~\eqref{lem:TPresDiscSendsLtoLorLt:eq1} and~\eqref{lem:TPresDiscSendsLtoLorLt:eq5}, correspondingly we conclude that 
\begin{equation}\label{lem:TPresDiscSendsLtoLorLt:eq6}
\left\langle I_n, A^{(1)}, \ldots, A^{(n)}\right\rangle = \mbf L(\mbf u).
\end{equation}
Therefore, 
\begin{multline}\notag
T\left(\mbf L(\mbf u)\right) \overset{\eqref{lem:TPresDiscSendsLtoLorLt:eq6}}{=\joinrel=} T\left(\left\langle I_n, A^{(1)}, \ldots, A^{(n)}\right\rangle\right)\\
\overset{\eqref{lem:TPresDiscSendsLtoLorLt:eq1}}{=\joinrel=}\langle dI_n\rangle + \left\langle B^{(1)}, \ldots, B^{(n)}\right\rangle \overset{\eqref{lem:TPresDiscSendsLtoLorLt:eq8}}{\subseteq} \R_n(\F).
\end{multline}
In addition,  $I_n \in  T\left(\mbf L(\mbf u)\right)$
by~\eqref{lem:TPresDiscSendsLtoLorLt:eq7} and
\begin{equation}\notag
\dim\left(T\left(\mbf L(\mbf u)\right) \right) = \dim\left(\mbf L(\mbf u)\right) = n + 1
\end{equation}
by bijectivity of $T$. Therefore, $W =  T\left(\mbf L(\mbf u)\right)$ satisfies the conditions of\linebreak Lemma~\ref{lem:RepAsRkOnePScalarNP1EqToL}, which implies that there is $0 \neq \mbf v \in \F^n$ such that $T\left(\mbf L(\mbf u)\right) = \mathbf L(\mathbf v)$  or $T\left(\mbf L(\mbf u)\right) = \mathbf L^t(\mathbf v).$
\end{proof}

\begin{corollary}\label{cor:IfPresDThenPresRk1PScFull}
Assume that $n \ge 3$ and $|\F| > n(n-1)$. Let $T$ be a bijective linear map on $\M_{n}(\F)$  such that $\D(T(X)) = \D(X)$ for all $X \in \M_{n}(\F)$. Then $T(\R_n(\F)) = \R_n(\F)$.
\end{corollary}

\begin{remark}
It is possible to obtain Corollary~\ref{cor:IfPresDThenPresRk1PScFull} by distinguishing $E_{1\,2} \in \M_{n}(\F)$ from $E_{1\,2} + E_{2\,3}$ and from $E_{1\,2} + E_{3\,4}$ for $n = 3, 4, 5$ as it is done for $E_{1\,1}$ in Section~\ref{sec:InvariantsOfSmallMatrices}. On the one hand, this approach does not involve Lemma~\ref{lem:RNonzeroRAll} and all underlying theory. On the other hand, we need to show that  $E_{1\,2}$ could not be sent to $E_{1\,2} + E_{2\,3}$ or $E_{1\,2} + E_{3\,4}$  by a linear map preserving $\D$. However, if $n = 4$ and $\Char(\F) = 2$, then $E_{1\,2}$ and $E_{1\,2} + E_{3\,4}$ have the same invariants of the form $\deg_\lambda \left(\B_{(X)}\right)$ and $\dim \left(\B_{(X)}^{(i)}\right)$. Thus, the additional invariants must be introduced to cover this case.
\end{remark}

\begin{remark}
As the reader can see from the statement of Lemma~\ref{lem:LowDegImpliesLowRankOrSeveralCasesForAlgebraicallyClosed} and the proof of Lemma~\ref{lem:IfPresDThenPresRk1PSc}, in the case when $n \ge 6$ we do not need the results from Section~\ref{sec:OrbInvariants} and Section~\ref{sec:InvariantsOfSmallMatrices} in order to obtain Corollary~\ref{cor:IfPresDThenPresRk1PScFull}.
\end{remark}

\begin{lemma}\label{lem:TPresDiscSendsLtoLorLtSimultan}Assume that $n \ge 3$ and $|\F| > n(n-1)$. Let $T$ be a bijective linear map preserving the matrix discriminant. Then either for every $0 \neq \mathbf u \in \F^n$ there is $0\neq \mathbf v \in \F^n$ such that $T\left(\mathbf L(\mathbf u)\right) = \mathbf L(\mathbf v)$, or for every $0 \neq \mathbf u \in \F^n$ there is $0\neq \mathbf v \in \F^n$ such that $T\left(\mathbf L(\mathbf u)\right) = \mathbf L^t(\mathbf v).$
\end{lemma}

\begin{proof}By contradiction. Suppose that $0\neq \mbf u^{(1)}, \mbf u^{(2)} \in \F^n $ are such that
\begin{equation}\notag
T\left(\mathbf L(\mathbf u_1)\right) \neq \mathbf L(\mathbf v)\;\;\mbox{for all} 0 \neq \mbf v \in \F^n
\end{equation}
and
\begin{equation}\notag
 T\left(\mathbf L(\mathbf u_2)\right) \neq \mathbf L^t(\mathbf v)\;\;\mbox{for all} 0 \neq \mbf v \in \F^n.
\end{equation}
Then Lemma~\ref{lem:TPresDiscSendsLtoLorLtSimultan} implies that there are $\mbf v^{(1)}, \mbf v^{(2)} \in \F^n$ such that
\begin{equation}\notag
T\left(\mathbf L(\mathbf u_1)\right) \neq \mathbf L^t(\mathbf v_1)\quad\mbox{and}\quad T\left(\mathbf L(\mathbf u_2)\right) = \mathbf L(\mathbf v_2).
\end{equation}

Consider $V = \mbf L^t(\mbf v_1) \cap \mbf L^t(\mbf v_2)$. On the other hand, since
\[
\dim \left(\mathbf L(\mathbf u_1) \cap \mathbf L(\mathbf u_2)\right) =
\begin{cases}
1, & \mbf u_1\;\mbox{and}\;\mbf u_2\;\mbox{are linearly independent}\\
n + 1, & \mbox{otherwise}
\end{cases}
\]
by Corollary~\ref{cor:DimCapSame}, then
\begin{multline}\label{lem:TPresDiscSendsLtoLorLtSimultan:eq1}
\dim (V) = \dim\left(\mbf L^t(\mbf v_1) \cap \mbf L^t(\mbf v_2)\right)\\
 = \dim\left(T\left(\mbf L^t(\mbf u_1)\right) \cap T\left(\mbf L^t(\mbf u_2)\right)\right)
= \dim\left(T\left(\mathbf L(\mathbf u_1) \cap \mathbf L(\mathbf u_2)\right)\right)\\
 = 
\begin{cases}
1, & \mbf u_1\;\mbox{and}\;\mbf u_2\;\mbox{are linearly independent}\\
n + 1, & \mbox{otherwise}
\end{cases}
\end{multline}
by bijectivity of the linear map $T$. On the other hand,
\[
\dim (V) = \dim \left(\mathbf L(\mathbf u_1) \cap \mathbf L(\mathbf u_2)\right) = 2
\]
by Corollary~\ref{cor:DimCapTrans}, which contradicts~\eqref{lem:TPresDiscSendsLtoLorLtSimultan:eq1}.
\end{proof}

\begin{lemma}\label{lem:TSendLtoLHasRightForm}Assume that $n \ge 3$ and $|\F| > n(n-1)$. Let $T$ be a bijective linear map on $\M_{n}(\F)$ such that $\D(T(X)) = \D(X)$ for all $X \in \M_{n}(\F)$ and for every $0 \neq \mathbf v \in \F^n$ there is $0 \neq \mathbf u \in \F^n$ such that $T(\mathbf L(\mathbf v)) = \mathbf L(\mathbf u)$. Then there are an invertible matrix $M \in \M_n(\F)$, a linear map $\alpha \colon \M_n(\F) \to \F$ and $c \in \F$ such that $c^{n(n-1)} = 1$ and
\begin{equation}\label{lem:TSendLtoLHasRightForm:eqq1}
T(X) = cM^{-1}XM + \alpha(X)I_n\;\;\mbox{for all}\;\; X \in \M_n(\F)
\end{equation}
\end{lemma}
\begin{proof}We will prove sequently that the statement of the lemma holds with the following assumptions:
\begin{enumerate}[label=(\alph*), ref=(\alph*)]

\item \label{lem:TSendLtoLHasRightForm:partm1} We assume that there are $0 \neq c \in \F$ and $D = (d_{i\,j}) \in \M_{n}(\F)$ such that 
    \begin{equation}\label{lem:TSendLtoLHasRightForm:eq10001}
    T\left(E_{i\,j}\right) = cE_{i\,j} + d_{i\,j}I_n\;\;\mbox{for all}\;\;1 \le i, j \le n.
    \end{equation}
\item \label{lem:TSendLtoLHasRightForm:part0} We assume that there are $C = (c_{i\,j}), D = (d_{i\,j}) \in \M_{n}(\F)$ such that $c_{i\,i} = c_{j\,j} \neq 0$ and 
    \begin{equation}\label{lem:TSendLtoLHasRightForm:eq10002}
    T\left(E_{i\,j}\right) = c_{i\,j}E_{i\,j} + d_{i\,j}I_n\;\;\mbox{for all}\;\;1 \le i, j \le n.
    \end{equation}
\item\label{lem:TSendLtoLHasRightForm:part1} We assume that 
\begin{equation}\label{lem:TSendLtoLHasRightForm:eq1AAA1}
T\left(\L\left(\mathbf e^{(i)}\right)\right) = \L\left(\mathbf e^{(i)}\right)\;\;\mbox{for all}\;\;1 \le i \le n.
\end{equation}
\item\label{lem:TSendLtoLHasRightForm:part2} We prove the statement of the lemma without any additional assumptions on $T$.
\end{enumerate}
\paragraph{\ref{lem:TSendLtoLHasRightForm:partm1}} From linearity of $T$ and the assumption~\eqref{lem:TSendLtoLHasRightForm:eq10001} we obtain that 
\begin{multline}\label{lem:TSendLtoLHasRightForm:eq111}
T(X) = T\left(\sum_{1 \le i, j \le n}x_{i\,j}E_{i\,j}\right)\\
= \sum_{1 \le i, j \le n}x_{i\,j} \left(cE_{i\,j} + d_{i\,j}I_n\right) = c\sum_{1 \le i, j \le n}x_{i\,j} E_{i\,j} +  \left(\sum_{1 \le i, j \le n} x_{i\,j} d_{i\,j}\right)I_n\\
= cX + \alpha(X)I_n\;\;\mbox{for all}\;\;X \in \M_{n}(\F),
\end{multline}
where 
\[
\alpha(X) = \sum_{1 \le i, j \le n} x_{i\,j} d_{i\,j}\;\;\mbox{for all}\;\;X \in \M_{n}(\F).
\]
Thus, in order to show that $T$ has the form~\eqref{lem:TSendLtoLHasRightForm:eqq1} it is only necessary to check that $c^{n(n-1)} = 1$. For this, let $A = \diag (a_1, \ldots, a_n)$ for some pairwise distinct $a_1, \ldots, a_n \in F^n$ (such $a_1,\ldots, a_n$ exist because $|\F| > n(n-1)$). Then 
\begin{multline}\label{lem:TSendLtoLHasRightForm:eq112}
\D(A) \overset{\mbox{\scriptsize $T$ preserves $\D$}}{=\joinrel=\joinrel=\joinrel=\joinrel=\joinrel=\joinrel=\joinrel=}\D(T(A))\\ \overset{\eqref{lem:TSendLtoLHasRightForm:eq111}}{=\joinrel=\joinrel=} \D(cA + \alpha(A) I_n) \overset{\ref{enum:DiscPropertyC}}{=\joinrel=\joinrel=} \D(cA) \overset{\ref{enum:DiscPropertyH}}{=\joinrel=\joinrel=} c^{n(n-1)}\D(A).
\end{multline}
Since $a_1, \ldots, a_n$ are pairwise distinct, then $\D(A) \neq 0$ by the property~\ref{enum:DiscPropertyDiag} of the matrix discriminant. Therefore, by dividing the both sides of~\eqref{lem:TSendLtoLHasRightForm:eq112} by $\D(A)$ we obtain that $c^{n(n-1)} = 1$. This establishes the statement of the lemma in this case.

\paragraph{\ref{lem:TSendLtoLHasRightForm:part0}} Let us show that $\rk(C) = 1$. Indeed, let $1 \le i_1 < i_2 \le n$ and $1 \le j_1 < j_2 \le n$ and let $B \in \M_{n}(\F)$ be defined by 
\[
B = E_{i_1\,j_1} + E_{i_1\,j_2} + E_{i_2\,j_1} + E_{i_2\,j_2}.
\]

On the one hand,
\begin{equation}\label{lem:TSendLtoLHasRightForm:eq100001}
T(B) = T\left(\sum_{1 \le \alpha, \beta \le 2} E_{i_{\alpha}\,j_{\beta}}\right) = \sum_{1 \le \alpha, \beta \le 2} c_{i_{\alpha}\,j_{\beta}} E_{i_{\alpha}\,j_{\beta}} + \sum_{1 \le \alpha, \beta \le 2} d_{i_{\alpha}\,j_{\alpha}}I_n
\end{equation}
by our initial assumption on $T$ and the linearity of $T$.

On the other hand, using the fact that $B \in \L(\mbf e_{i_1} + \mbf e_{i_2}) \cap \L^t(\mbf e_{j_1} + \mbf e_{j_2})$ and relying on Lemma~\ref{lem:TPresDiscSendsLtoLorLt} and Corollary~\ref{cor:DiscPresPresIncidence} we obtain that $T(B) \in \L(\mbf u) \cap \L^t(\mbf v)$ for some $0\neq \mbf u \in \left\langle \mbf e_{i_1}, \mbf e_{i_2}\right\rangle$ and $0\neq \mbf v \in \left\langle \mbf e_{j_1}, \mbf e_{j_2}\right\rangle$ . This implies that
\begin{equation}\label{lem:TSendLtoLHasRightForm:eq101}
T(B) = J(\mathbf v, \mathbf u) + dI_n
\end{equation}
Since $n \ge 3$, then by comparing~\eqref{lem:TSendLtoLHasRightForm:eq100001} and~\eqref{lem:TSendLtoLHasRightForm:eq101} we obtain that 
\begin{equation}\label{lem:TSendLtoLHasRightForm:eq100002}
\sum_{1 \le \alpha, \beta \le 2} c_{i_{\alpha}\,j_{\beta}} E_{i_{\alpha}\,j_{\beta}} = J(\mathbf v, \mathbf u)
\end{equation}
by Lemma~\ref{cor:OneWayRankOnePlusScalar}. Hence, relying on the equality
\[
C[i_1, i_2| j_1, j_2] = \sum_{1 \le \alpha, \beta \le 2} c_{i_{\alpha}\,j_{\beta}} E_{i_{\alpha}\,j_{\beta}} \overset{\eqref{lem:TSendLtoLHasRightForm:eq100002}}{=\joinrel=}  J(\mathbf v, \mathbf u)
\] 
we obtain that 
\begin{equation}\label{lem:TSendLtoLHasRightForm:eq100003}
\rk\left(C[i_1, i_2| j_1, j_2]\right) \le 1.
\end{equation} 
because $\rk\left(J(\mathbf v, \mathbf u)\right) \le 1$ by the property~\ref{prop:J0}.
The equality~\eqref{lem:TSendLtoLHasRightForm:eq100003} holds for all $1 \le i_1 < i_2 \le n$ and $1 \le j_1 < j_2 \le n$. This implies that $\rk(C) \le 1$. In addition, $c_{i\,i} = c_{j\,j} \neq 0$ for all $1 \le i, j \le n$ by the assumptions of this part of the lemma. Therefore, by Lemma~\ref{lem:MatRkOneConstDiagDecomp} we obtain that there are $0\neq c, c_1, \ldots, c_n$ such that 
\[
c_{i\,j} = c c_{i} c_{j}^{-1}\;\;\mbox{for all}\;\;1 \le i, j \le n.
\]

Let $A \in \M_{n}(\F)$ be defined by 
\begin{equation}\label{lem:TSendLtoLHasRightForm:eq102}
  A = \diag(c_{1}, \ldots, c_{n}).
\end{equation}
Since $c_i \neq 0$, then $A$ is invertible and $A^{-1} = \diag(c^{-1}_{1}, \ldots, c^{-1}_{n}).$ Consider a linear map $T_A \colon \M_{n}(\F) \to \M_{n}(\F)$ defined by
\begin{equation}\label{lem:TSendLtoLHasRightForm:eq103}
T_A(X) = A^{-1}T(X)A\;\;\mbox{for all}\;\;X \in \M_{n}(\F).
\end{equation}
Then
\begin{multline*}
T_A(E_{i\,j}) = A^{-1}T(E_{i\,j})A =  A^{-1}\left(c_{i\,j}E_{i\,j} + d_{i\,j}I_n\right)A\\
=  A^{-1}\left(c_{i\,j}E_{i\,j}\right)A +  d_{i\,j}I_n = c^{-1}_i \left(c_{i\,j}E_{i\,j}\right) c_j +   d_{i\,j}I_n\\
 = \left(c^{-1}_i\left(c c_{i} c_{j}^{-1}\right) c_{j} \right) E_{i\,j} +   d_{i\,j}I_n\\
= cE_{i\,j} +   d_{i\,j}I_n\;\;\mbox{for all}\;\;1 \le i, j \le n.
\end{multline*}
In addition, $T_A$ preserves $\D$ by~\ref{enum:DiscPropertyS}. The linear map $T_A$ is bijective by bijectivity of matrix similarity. Similarly, since $A^{-1}\mbf L(\mathbf v)A = \mbf L(A^{-1}\mathbf v)$ for all $\mbf v \in \F^n$, then for every $0 \neq \mathbf v \in \F^n$ there is $0 \neq \mathbf u \in \F^n$ such that $T_A(\mathbf L(\mathbf v)) = \mathbf L(\mathbf u)$. Hence, $T_A$ satisfies the conditions of the part~\ref{lem:TSendLtoLHasRightForm:partm1} of the lemma. Therefore, $T_A$ has the form~\eqref{lem:TSendLtoLHasRightForm:eqq1}. This implies that $T$ also has the required form because by~\eqref{lem:TSendLtoLHasRightForm:eq103} we have $T(X) =  AT_A(X)A^{-1}$ for all $X \in \M_{n}(\F)$.

\paragraph{\ref{lem:TSendLtoLHasRightForm:part1}}Since $T$ sends the spaces of the form $\L(\mathbf v)$ to the spaces of the same form, then Lemma~\ref{lem:TPresDiscSendsLtoLorLtSimultan}, Corollary~\ref{cor:DimCapSame} and Corollary~\ref{cor:DimCapTrans} imply that there are $0 \neq \mathbf u^{(1)}, \ldots, \mathbf u^{(n)} \in \F$ such that
\begin{equation}\label{lem:TSendLtoLHasRightForm:eq1}
T\left(\L^t\left(\mbf e^{(i)}\right)\right) = \L^t\left(\mbf u^{(i)}\right)\;\;\mbox{for all}\;\;1 \le i \le n.
\end{equation}

Since $T$ is bijective, then
\begin{equation}\label{lem:TSendLtoLHasRightForm:eq2}
T\left(\L\left(\mbf e^{(i)}\right)\cap \L^t\left(\mbf e^{(j)}\right)\right) = T\left(\L\left(\mbf e^{(i)}\right)\right)\cap T\left( \L^t\left(\mbf e^{(j)}\right)\right)\;\;\mbox{for all}\;\; 1 \le i, j \le n.
\end{equation}
Hence,
\begin{multline}\label{lem:TSendLtoLHasRightForm:eq3}
\L\left(\mbf e^{(i)}\right)\cap \L^t\left(\mbf u^{(j)}\right) \overset{\eqref{lem:TSendLtoLHasRightForm:eq1AAA1}}{=\joinrel=} T\left(\L\left(\mbf e^{(i)}\right)\right)\cap \L^t\left(\mbf u^{(j)}\right)\\
\overset{\eqref{lem:TSendLtoLHasRightForm:eq1}}{=\joinrel=} T\left(\L\left(\mbf e^{(i)}\right)\right)\cap T\left(\L^t\left(\mbf e^{(i)}\right)\right) \overset{\eqref{lem:TSendLtoLHasRightForm:eq2}}{=\joinrel=} T\left(\L\left(\mbf e^{(i)}\right)\cap \L^t\left(\mbf e^{(j)}\right)\right)
\end{multline}
for all $1 \le i, j \le n.$

Let $1 \le i, j \le n$. 
Since $E_{i\, j} \in \L\left(\mbf e^{(i)}\right)\cap \L^t\left(\mbf e^{(j)}\right)$, then $T(E_{i\, j}) \in \L\left(\mbf e^{(i)}\right)\cap \L^t\left(\mbf u^{(j)}\right)$ by~\eqref{lem:TSendLtoLHasRightForm:eq3}. Therefore, Lemma~\ref{lem:CapTrans} infers that there are $c'_{i\,j}$ and $d'_{i\,j}$ such that
\begin{equation}\label{lem:TSendLtoLHasRightForm:eq4}
T(E_{i\, j}) = c'_{i\,j}J(\mbf u^{(j)}, \mbf e^{(i)}) + d'_{i\,j}I_{n}.
\end{equation}

On the one hand, the assumption that $T$ is bijective and preserves $\D$, Corollary~\ref{cor:DRad} and Lemma~\ref{lem:RadicalInvariant}\ref{lem:RadicalInvariant:part2} imply that there is $c \neq 0$ such that
\begin{equation}\label{lem:TSendLtoLHasRightForm:eq5}
T(I_n) = cI_n.
\end{equation}

On the other hand, the linearity of $T$ implies that 
\begin{equation}\label{lem:TSendLtoLHasRightForm:eq6}
T(I_n) = T(E_{1\,1} + \ldots + E_{n\,n}) = \sum_{i=1}^{n}T(E_{i\,i}) = \sum_{i=1}^{n} c'_{i\,i}J(\mbf u^{(i)}, \mbf e^{(i)}) + d'_{i\,i}I_n
\end{equation}
Note that by comparing~\eqref{lem:TSendLtoLHasRightForm:eq4} and~\eqref{lem:TSendLtoLHasRightForm:eq5} we obtain that 
\begin{equation}\label{lem:TSendLtoLHasRightForm:eq8}
c'_{i\,j} \neq 0\;\;\mbox{for all}\;\;1 \le i, j \le n.
\end{equation}
Indeed, since $c \neq 0$, then $cE_{i\,j} - d'_{i\,j}I_{n} \neq 0$ and consequently
$T(cE_{i\,j} - d'_{i\,j}I_{n}) \neq 0.$ This implies that
\begin{multline*}
0 \neq T(cE_{i\,j} - d'_{i\,j}I_{n})\\ \overset{\eqref{lem:TSendLtoLHasRightForm:eq4}\;\;\mbox{and}\;\;\eqref{lem:TSendLtoLHasRightForm:eq5}}{=\joinrel=\joinrel=\joinrel=\joinrel=\joinrel=\joinrel=\joinrel=\joinrel=\joinrel=\joinrel=} 
c\left(c'_{i\,j}J(\mbf u^{(j)}, \mbf e^{(i)}) + d'_{i\,j}I_{n}\right) - d'_{i\,j}\left(cI_{i\,j}\right)\\
 = c c'_{i\,j}J(\mbf u^{(j)}, \mbf e^{(i)}).
\end{multline*}
Hence, $c'_{i\,j} \neq 0$ for all $1 \le i, j \le n$.

By expressing $\mbf u^{(j)}$ through the basis $\mbf e^{(1)}, \ldots, \mbf e^{(n)}$ we obtain that
\begin{multline}\label{lem:TSendLtoLHasRightForm:eq7}
J(\mbf u^{(j)}, \mbf e^{(i)}) = \sum_{l = 1}^{n} J(\mbf u^{(j)}, \mbf e^{(i)}) = \sum_{k = 1}^{n} \mbf u^{(j)}_k J(\mbf e^{(k)}, \mbf e^{(i)})\\
\overset{\mbox{\scriptsize the definition of $J(\mbf u, \mbf v)$}}{=\joinrel=\joinrel=\joinrel=\joinrel=\joinrel=\joinrel=\joinrel=\joinrel=\joinrel=\joinrel=\joinrel=\joinrel=\joinrel=} \sum_{k = 1}^{n} \mbf u^{(j)}_k E_{i\,k}\quad\mbox{for all}\;\;1 \le j \le n.
\end{multline}
Using~\eqref{lem:TSendLtoLHasRightForm:eq7}, we rewrite the equality~\eqref{lem:TSendLtoLHasRightForm:eq6} as follows
\begin{multline}\label{lem:TSendLtoLHasRightForm:eq9}
T(I_n) = \sum_{i=1}^{n} \left(c'_{i\,i} \left(\sum_{k = 1}^{n} \mbf u^{(i)}_k E_{i\,k}\right) + d'_{i\,i}I_n\right)\\
= \left(\sum_{l=1}^{n} d'_{l\,l}\right)I_n + \sum_{1 \le i, k \le n}  c'_{i\,i}\mathbf u^{(i)}_k E_{i\,k}
\end{multline}
By entrywise comparison of the matrices on the left-hand sides of~\eqref{lem:TSendLtoLHasRightForm:eq6} and~\eqref{lem:TSendLtoLHasRightForm:eq9} we infer that
\[
c'_{i\,i}\mathbf u^{(i)}_k = \delta_{i\,k} \left(c - \sum_{l=1}^{n} d'_{l\,l}\right)\;\;\mbox{for all}\;\;1\le i, k \le n
\]
and obtain in consequence that
\begin{equation*}
\mathbf u^{(i)}_k = \delta_{i\,k} \frac{c - \left(\sum_{l=1}^{n} d'_{l\,l}\right)}{c'_{i\,i}};\;\mbox{for all}\;\;1\le i, k \le n
\end{equation*}
by~\eqref{lem:TSendLtoLHasRightForm:eq8}. Hence, 
\begin{equation}\label{lem:TSendLtoLHasRightForm:eq10}
\mathbf u^{(i)} = \frac{c - \left(\sum_{l=1}^{n} d'_{l\,l}\right)}{c'_{i\,i}}\mathbf e^{(i)}\;\;\mbox{for all}\;\;1 \le i \le n.
\end{equation}
By substituting~\eqref{lem:TSendLtoLHasRightForm:eq10} into~\eqref{lem:TSendLtoLHasRightForm:eq4} we conclude that
\begin{multline}\label{lem:TSendLtoLHasRightForm:eq11}
T(E_{i\, j}) = c'_{i\,j}\frac{c - \left(\sum_{l=1}^{n} d'_{l\,l}\right)}{c'_{i\,i}}J(\mbf e^{(j)}, \mbf e^{(i)}) + d'_{i\,j}I_{n}\\
\overset{\mbox{\scriptsize the definition of $J(\mbf u, \mbf v)$}}{=\joinrel=\joinrel=\joinrel=\joinrel=\joinrel=\joinrel=\joinrel=\joinrel=\joinrel=\joinrel=\joinrel=\joinrel=\joinrel=} c'_{i\,j}\frac{c - \left(\sum_{l=1}^{n} d'_{l\,l}\right)}{c'_{i\,i}}E_{i\,j} + d'_{i\,j}I_{n}\;\;\mbox{for all}\;\;1 \le i \le n.
\end{multline}
Let $C = (c_{i\,j}), D = (d_{i\,j})\in \M_{n}(\F)$ be defined by
\begin{equation}\label{lem:TSendLtoLHasRightForm:eq12}
c_{i\,j} = c'_{i\,j}\frac{c - \left(\sum_{l=1}^{n} d'_{l\,l}\right)}{c'_{i\,i}}\qquad\mbox{and}\qquad d_{i\,j} = d'_{i\,j}\;\;\mbox{for all}\;\;1 \le i, j \le n.
\end{equation}
Using~\eqref{lem:TSendLtoLHasRightForm:eq12}, the equality~\eqref{lem:TSendLtoLHasRightForm:eq11} is rewritten as follows
\begin{equation*}
T(E_{i\, j}) = c_{i\,j}E_{i\,j} + d_{i\,j}I_n\;\;\mbox{for all}\;\;1 \le i \le n.
\end{equation*}
Let us check that $C$ satisfies the assumptions of the part~\ref{lem:TSendLtoLHasRightForm:part0} of the lemma. Indeed, $c_{i\,i}$ by the definition of $c_{i\,j}$ in~\eqref{lem:TSendLtoLHasRightForm:eq12} and~\eqref{lem:TSendLtoLHasRightForm:eq10} because $\mathbf u^{(i)} \neq 0$ for all $1 \le i \le n$. In turn, for all $1 \le i, j \le n$ we have
\begin{equation*}
c_{i\,i} \overset{\eqref{lem:TSendLtoLHasRightForm:eq12}}{=\joinrel=\joinrel=}  c'_{i\,i}\frac{c - \left(\sum_{i=1}^{n} d'_{i\,i}\right)}{c'_{i\,i}} = c - \left(\sum_{i=1}^{n} d'_{i\,i}\right)\\
= c'_{j\,j}\frac{c - \left(\sum_{i=1}^{n} d'_{i\,i}\right)}{c'_{j\,j}} \overset{\eqref{lem:TSendLtoLHasRightForm:eq12}}{=\joinrel=\joinrel=} c_{j\,j}.
\end{equation*}
Hence, the required assumptions of the part~\ref{lem:TSendLtoLHasRightForm:part0} of the lemma are satisfied. Using this part of the lemma we conclude that $T$ has the required form.

\paragraph{\ref{lem:TSendLtoLHasRightForm:part2}} Since $T$ sends the spaces of the form $\L(\mathbf v)$ to the spaces of the same form, then Lemma~\ref{lem:TPresDiscSendsLtoLorLtSimultan} implies that there are $0 \neq \mathbf u^{(1)}, \ldots, \mathbf u^{(n)} \in \F$ such that
\begin{equation}\label{lem:TSendLtoLHasRightForm:eq1301}
T\left(\L\left(\mbf e^{(i)}\right)\right) = \L\left(\mbf u^{(i)}\right)\;\;\mbox{for all}\;\;1 \le i \le n.
\end{equation}
The linear map $T$ is bijective and $\mbf e^{(1)}, \ldots, \mbf e^{(n)}$ are linearly independent. Hence, Lemma~\ref{lem:DiscPresPresIncidence} infers that $\mbf u^{(1)}, \ldots, \mbf u^{(n)}$ are linearly independent as well. Therefore, $\mbf u^{(1)}, \ldots, \mbf u^{(n)}$ form a basis of $\F^n$ and there exists a matrix $A \in M_{n}(\F)$  such that 
\begin{equation}\label{lem:TSendLtoLHasRightForm:eq1302}
A\mathbf u^{(i)} = \mathbf e^{(i)}\;\;\mbox{for all}\;\;1 \le i \le n.
\end{equation}

Let $T_A \colon \M_{n}(\F) \to \M_{n}(\F)$ be defined by
\begin{equation}\label{lem:TSendLtoLHasRightForm:eq1303}
T_A(X) = AT(X)A^{-1}\;\;\mbox{for all}\;\;X \in \M_n(\F).
\end{equation}
Then $T'$ is bijective and
\begin{multline*}
T_A\left(\L\left(\mbf e^{(i)}\right)\right) \overset{\eqref{lem:TSendLtoLHasRightForm:eq1303}}{=\joinrel=\joinrel=} AT\left(\L\left(\mbf e^{(i)}\right)\right)A^{-1} \overset{\eqref{lem:TSendLtoLHasRightForm:eq1301}}{=\joinrel=\joinrel=} AL\left(\mbf u^{(i)}\right)A^{-1}\\
\overset{\mbox{\scriptsize Lemma~\ref{lem:MatMultOnLn}}}{=\joinrel=\joinrel=\joinrel=\joinrel=\joinrel=\joinrel=\joinrel=} L\left(A\mbf u^{(i)}\right) \overset{\eqref{lem:TSendLtoLHasRightForm:eq1302}}{=\joinrel=\joinrel=}  L\left(\mbf e^{(i)}\right)\;\;\mbox{for all}\;\;1 \le i \le n.
\end{multline*}
In addition, $T_A$ preserves $\D$ by~\ref{enum:DiscPropertyS}. The linear map $T_M$ is bijective by bijectivity of matrix similarity. Similarly, since $A\mbf L(\mathbf v)A^{-1} = \mbf L(A\mathbf v)$ for all $\mbf v \in \F^n$, then for every $0 \neq \mathbf v \in \F^n$ there is $0 \neq \mathbf u \in \F^n$ such that $T_A(\mathbf L(\mathbf v)) = \mathbf L(\mathbf u)$. Thus, $T_A$ satisfies all the conditions of the part~\ref{lem:TSendLtoLHasRightForm:part1} of the lemma and consequently has the form~\eqref{lem:TSendLtoLHasRightForm:eqq1}. Therefore, $T$  also has the required form because $T = A^{-1}T_A(A)A$ for all $X \in \M_n(\F)$ by~\eqref{lem:TSendLtoLHasRightForm:eq1303}.
\end{proof}

\begin{theorem}\label{thm:DPresHasForm}Assume that $n \ge 3$ and $|\F| > n(n-1)$. Let $T$ be a linear map on $\M_{n}(\F)$ such that $\D(T(X)) = \D(X)$ for all $X \in \M_{n}(\F)$. Then there are an invertible matrix $M \in \M_n(\F)$, a linear map $\alpha \colon \M_n(\F) \to \F$ and $c \in \F$ such that $c^{n(n-1)} = 1$ and
\begin{equation}\label{thm:DPresHasForm:eqq1}
T(X) = cM^{-1}XM + \alpha(X)I_n\;\;\mbox{for all}\;\; X \in \M_n(\F)
\end{equation}
or
\begin{equation}\label{thm:DPresHasForm:eqq2}
T(X) = cM^{-1}X^tM + \alpha(X)I_n\;\;\mbox{for all}\;\; X \in \M_n(\F).
\end{equation}
\end{theorem}
\begin{proof}We will prove sequently that the statement of the theorem holds with the following assumptions:
\begin{enumerate}[label=(\alph*), ref=(\alph*)]

\item\label{lem:TSendLtoLHasRightForm:part3} We assume that $T$ is bijective.
\item\label{lem:TSendLtoLHasRightForm:part4} We prove the statement of the lemma without any additional assumptions on $T$.
\end{enumerate}
\paragraph{\ref{lem:TSendLtoLHasRightForm:part3}} Since $T$ is bijective and preserves $\D$, then by Lemma~\ref{lem:TPresDiscSendsLtoLorLtSimultan} we obtain that either for every $0 \neq \mathbf v \in \F^n$ there is $\mathbf u \in \F^n$ such that $T(\mathbf L(\mathbf v)) = \mathbf L(\mathbf u)$, or for every $0 \neq \mathbf v \in \F^n$ there is $\mathbf u \in \F^n$ such that $T(\mathbf L(\mathbf v)) = \mathbf L^t(\mathbf u).$ In the first case $T$ satisfies the assumptions of Lemma~\ref{lem:TSendLtoLHasRightForm} and consequently has the form~\eqref{thm:DPresHasForm:eqq1}. 

In the second case we consider $T' \colon \M_{n}(\F) \to \M_{n}(\F)$ defined by 
\begin{equation}\label{thm:DPresHasForm:eq1301}
T'(X) = T^t(X)\;\;\mbox{for all}\;\;X \in \M_n(\F).
\end{equation}
The property~\ref{enum:DiscPropertyT} infers that $T'$ preserves $\D$. Thus, $T'$ satisfies the assumptions of Lemma~\ref{lem:TSendLtoLHasRightForm} and consequently has the form~\eqref{thm:DPresHasForm:eqq1}.  Therefore, by~\eqref{thm:DPresHasForm:eq1301} we conclude that $T$ has the form~\eqref{thm:DPresHasForm:eqq2}.

\paragraph{\ref{lem:TSendLtoLHasRightForm:part4}} By Corollary~\ref{cor:DRad} and Lemma~\ref{lem:KernelContainedInRad}, $\ker(T) \subseteq \langle I_n\rangle$. Therefore, if $T(I_n) \neq 0,$ then $T$ is bijective and consequently the statement of the lemma holds by applying  the part~\ref{lem:TSendLtoLHasRightForm:part3} of the lemma to $T$. Otherwise, 
\begin{equation}\label{lem:TSendLtoLHasRightForm:eq1402}
T(I_n) = O_n.
\end{equation}
In this case, let $T' \colon \M_n(\F) \to \M_n(\F)$ defined by
\begin{equation}\label{lem:TSendLtoLHasRightForm:eq1401}
T'(X) = T(X) + \alpha(X)I_n\;\;\mbox{for all}\;\;X \in \M_{n}(\F),
\end{equation}
where $\alpha(X) = x_{1\,1}$ for all $X = (x_{i\,j}) \in \M_{n}(\F)$. By the property~\ref{enum:DiscPropertyC} of the matrix discriminant, $T'$ preserves $\D$. In addition, $T'$ is bijective because
\[
T'(I_n) \overset{\eqref{lem:TSendLtoLHasRightForm:eq1401}}{=\joinrel=\joinrel=} T(I_n) + \left(I_n\right)_{1\,1}I_n \overset{\eqref{lem:TSendLtoLHasRightForm:eq1402}}{=\joinrel=\joinrel=} O_n + I_n = I_n \neq O_n.
\]
By applying the part~\ref{lem:TSendLtoLHasRightForm:part3} of the lemma to $T'$ we obtain that $T'$ has the \linebreak form~\eqref{thm:DPresHasForm:eqq1} or~\eqref{thm:DPresHasForm:eqq2}. Therefore, by expressing $T$ through $T'$ and $\alpha$ using~\eqref{lem:TSendLtoLHasRightForm:eq1401} we conclude that $T$ has the required form or as well.
 \end{proof}

\begin{remark}\label{rem:CaseN2} For $n = 2$ there are two cases possible: $\Char(\F) \neq 2$ and $\Char(\F) = 2$. In~\cite{Yurkov2025} we prove that Theorem~\ref{thm:DPresHasForm} holds in the case when $\Char(\F) \neq 2$ . If $\Char(\F) = 2$, then formula~\eqref{lem:D2PolyN2Char2:eqq1} implies that $\D(X) = \B^2(X) = \tr^2(X)$ for all $X \in \M_{2}$. This implies that the class of linear maps preserving $\D$ coincides with the class of linear maps preserving the matrix trace, which is wider than the class of linear maps having the form~\eqref{thm:DPresHasForm:eqq1} or~\eqref{thm:DPresHasForm:eqq2} (see Example~\ref{rem:CaseN2Ch2Example} below).
\end{remark}

\begin{example}\label{rem:CaseN2Ch2Example}Assume that $\Char(\F) = 2$. Let $T \colon \M_2(\F) \to \M_2(\F)$ be defined by
\[
T(X) = (x_{1\,1} + x_{2\,2})E_{1\,1} + x_{1\,1}E_{1\,2}\;\;\mbox{for all}\;\;X = (x_{i\,j}) \in \M_2(\F).
\]  
Then $T$ clearly preserves $\tr(X)$. In addition, $T(I_2) = E_{1\,2}$, which cannot occur if $T$ has the form~\eqref{thm:DPresHasForm:eqq1} or~\eqref{thm:DPresHasForm:eqq2} because these maps send $I_2$ to $cI_2$ for some $c \in \F$.
\end{example}

The author is grateful to his supervisor Prof. Alexander Guterman for setting the problem, attention to the work and helpful discussions. The research was supported by the scholarship of the Center for Absorption in Science, the Ministry of Aliyah and Integration, the State of Israel.

\bibliographystyle{plain}

\end{document}